\documentclass[12pt]{amsart}
\usepackage{amscd,verbatim}
\usepackage{amssymb}
\usepackage[all]{xy}
\usepackage[colorlinks,linkcolor=blue,citecolor=blue,urlcolor=red]{hyperref}

\newcommand{\un}{\mathbf{1}}
\newcommand{\A}{\mathbf{A}}

\renewcommand{\L}{\mathbb{L}}

\renewcommand{\P}{\mathbf{P}}

\newcommand{\Z}{\mathbb{Z}}
\newcommand{\sA}{\mathcal{A}}
\newcommand{\sB}{\mathcal{B}}
\newcommand{\sC}{\mathcal{C}}
\newcommand{\sD}{\mathcal{D}}
\newcommand{\sE}{\mathcal{E}}

\newcommand{\sH}{\mathcal{H}}

\newcommand{\sK}{\mathcal{K}}

\newcommand{\sP}{\mathcal{P}}

\newcommand{\sR}{\mathcal{R}}
\newcommand{\sS}{\mathcal{S}}
\newcommand{\sT}{\mathcal{T}}
\newcommand{\sU}{\mathcal{U}}

\newcommand{\sX}{\mathcal{X}}
\newcommand{\sY}{\mathcal{Y}}
\newcommand{\sZ}{\mathcal{Z}}

\newcommand{\Mod}{\operatorname{Mod}\hbox{--}}

\newcommand{\Cor}{\operatorname{\mathbf{Cor}}}

\newcommand{\ulSigma}{\ul{\Sigma}}

\newcommand{\Cone}{\operatorname{Cone}}

\newcommand{\ul}[1]{{\underline{#1}}}

\newcommand{\Sq}{{\operatorname{\mathbf{Sq}}}}

\newcommand{\PST}{{\operatorname{\mathbf{PST}}}}

\newcommand{\NST}{\operatorname{\mathbf{NST}}}

\newcommand{\Chow}{\operatorname{\mathbf{Chow}}}
\newcommand{\DM}{\operatorname{\mathbf{DM}}}

\newcommand{\DR}{\operatorname{\mathbf{\underline{M}DM}}}
\newcommand{\ulMDM}{\operatorname{\mathbf{\underline{M}DM}}}
\newcommand{\MDM}{\operatorname{\mathbf{MDM}}}

\newcommand{\Hom}{\operatorname{Hom}}
\newcommand{\uHom}{\operatorname{\underline{Hom}}}

\newcommand{\Ker}{\operatorname{Ker}}
\newcommand{\IM}{\operatorname{Im}}
\renewcommand{\Im}{\operatorname{Im}}
\newcommand{\Coker}{\operatorname{Coker}}

\newcommand{\Pic}{\operatorname{Pic}}

\newcommand{\Spec}{\operatorname{Spec}}

\newcommand{\Sm}{\operatorname{\mathbf{Sm}}}

\newcommand{\Sch}{\operatorname{\mathbf{Sch}}}
\newcommand{\Ab}{\operatorname{\mathbf{Ab}}}

\newcommand{\by}{\xrightarrow}
\newcommand{\yb}{\xleftarrow}
\newcommand{\iso}{\by{\sim}}
\newcommand{\osi}{\yb{\sim}}

\newcommand{\tr}{{\operatorname{tr}}}
\newcommand{\proper}{{\operatorname{prop}}}

\newcommand{\eff}{{\operatorname{eff}}}

\newcommand{\fin}{{\operatorname{fin}}}

\newcommand{\amb}{{\operatorname{fin}}}
\renewcommand{\o}{{\operatorname{o}}}
\newcommand{\op}{{\operatorname{op}}}

\newcommand{\Nis}{{\operatorname{Nis}}}

\newcommand{\inj}{\hookrightarrow}

\newcommand{\id}{{\operatorname{Id}}}

\newcommand{\Tot}{\operatorname{Tot}}

\renewcommand{\lim}{\operatornamewithlimits{\varprojlim}}
\newcommand{\colim}{\operatornamewithlimits{\varinjlim}}

\newcommand{\ol}{\overline}

\newcommand{\car}{\operatorname{char}}

\renewcommand{\phi}{\varphi}
\renewcommand{\epsilon}{\varepsilon}

\newcommand{\gm}{{\operatorname{gm}}}

\newcommand{\MV}{\operatorname{MV}}
\newcommand{\ulMV}{\operatorname{\underline{MV}}}

\newcommand{\MNST}{\operatorname{\mathbf{MNST}}}

\newcommand{\MSm}{\operatorname{\mathbf{MSm}}}
\newcommand{\MCor}{\operatorname{\mathbf{MCor}}}
\newcommand{\MP}{\operatorname{\mathbf{MSm}}}

\newcommand{\MPST}{\operatorname{\mathbf{MPST}}}

\newcommand{\Bl}{{\mathbf{Bl}}}
\newcommand{\Cb}{\operatorname{\mathbf{Cube}}}
\newcommand{\ECb}{\operatorname{\mathbf{ECube}}}
\newcommand{\Sets}{\operatorname{\mathbf{Sets}}}

\newcommand{\bcube}{{\ol{\square}}}

\newcommand{\Mb}{{\overline{M}}}
\newcommand{\Nb}{{\overline{N}}}

\def\rmapo#1{\overset{#1}{\longrightarrow}}

\newcommand{\ulMSm}{\operatorname{\mathbf{\underline{M}Sm}}}
\newcommand{\ulMP}{\operatorname{\mathbf{\underline{M}Sm}}}
\newcommand{\ulMNS}{\operatorname{\mathbf{\underline{M}NS}}}
\newcommand{\ulMPS}{\operatorname{\mathbf{\underline{M}PS}}}
\newcommand{\ulMPST}{\operatorname{\mathbf{\underline{M}PST}}}
\newcommand{\ulMNST}{\operatorname{\mathbf{\underline{M}NST}}}
\newcommand{\ulMCor}{\operatorname{\mathbf{\underline{M}Cor}}}
\newcommand{\ulomega}{\underline{\omega}}

\newcommand{\Comp}{\operatorname{\mathbf{Comp}}}

\newcommand{\ulMEt}{\operatorname{\mathbf{\underline{M}Et}}}
\newcommand{\MEt}{\operatorname{\mathbf{MEt}}}

\newcommand{\OD}{\mathrm{OD}}

\newcounter{spec}
\newenvironment{thlist}{\begin{list}{\rm{(\roman{spec})}}%
{\usecounter{spec}\labelwidth=20pt\itemindent=0pt\labelsep=10pt}}%
{\end{list}}%

\newtheorem{Th}{Theorem}
\newtheorem{lemma}{Lemma}[subsection]
\newtheorem{thm}[lemma]{Theorem}

\newtheorem{prop}[lemma]{Proposition}
\newtheorem{proposition}[lemma]{Proposition}
\newtheorem{cor}[lemma]{Corollary}

\theoremstyle{definition}
\newtheorem{def-prop}[lemma]{Definition-Proposition}
\newtheorem{defn}[lemma]{Definition}

\newtheorem{definition}[lemma]{Definition}

\newtheorem{hyp}[lemma]{Hypothesis}

\theoremstyle{remark}
\newtheorem{qn}[lemma]{Question}
\newtheorem{rk}[lemma]{Remark}
\newtheorem{remark}[lemma]{Remark}
\newtheorem{remarks}[lemma]{Remarks}
\newtheorem{ex}[lemma]{Example}

\newtheorem{claim}[lemma]{Claim}

\numberwithin{equation}{section}

\begin{document}
\title[Motives with modulus, III]{Motives with modulus, III: \\ The categories of motives}
\author[B. Kahn]{Bruno Kahn}
\address{IMJ-PRG\\Case 247\\
4 place Jussieu\\
75252 Paris Cedex 05\\
France}
\email{bruno.kahn@imj-prg.fr}
\author[H. Miyazaki]{Hiroyasu Miyazaki}
\address{RIKEN iTHEMS, Wako, Saitama 351-0198, Japan}
\email{hiroyasu.miyazaki@riken.jp}
\author[S. Saito]{Shuji Saito}
\address{Graduate School of Mathematical Sciences\\
University of Tokyo\\
3-8-1 Komaba, \\ Tokyo 153-8941\\ Japan}
\email{sshuji@msb.biglobe.ne.jp}
\author[T. Yamazaki]{Takao Yamazaki}
\address{Institute of Mathematics\\ Tohoku University\\ Aoba\\ Sendai 980-8578\\ Japan}
\email{ytakao@math.tohoku.ac.jp}
\date{July 23, 2021}
\thanks{The first author acknowledges the support of Agence Nationale de la Recherche (ANR) under reference ANR-12-BL01-0005. 
The second author is supported by RIKEN Interdisciplinary Theoretical and Mathematical Sciences Program, and by JSPS KAKENHI Grant (19K23413).
The third author is supported by JSPS KAKENHI Grant (15H03606).
The fourth author is supported by JSPS KAKENHI Grant (15K04773). 
}

\begin{abstract}
We construct and study a triangulated category
of motives with modulus $\MDM_\gm^\eff$
over a field $k$, that
extends Voevodsky's category $\DM_\gm^\eff$
in such a way as to encompass non-homotopy
invariant phenomena. 
In a similar way as $\DM_\gm^\eff$ is constructed out of 
smooth $k$-varieties,
$\MDM_\gm^\eff$ is constructed out of 
\emph{proper modulus pairs}, introduced in Part I of this work.
To such a modulus pair 
we associate its motive  in $\MDM_\gm^\eff$.
In some cases, the $\Hom$ group in $\MDM_\gm^\eff$
between the motives of two modulus pairs
can be described
in terms of Bloch's higher Chow groups.
\end{abstract}

\subjclass[2010]{19E15 (14F42, 19D45, 19F15)}

\maketitle

\tableofcontents

\enlargethispage*{20pt}

\section*{Introduction} In this paper, we construct triangulated categories of ``motives with modulus'' over a field $k$, in parallel with Voevodsky's construction of triangulated categories of motives in \cite{voetri}. Our motivation comes from the reciprocity sheaves studied in \cite{rec}; the link between the present theory and \cite{rec}  is established in  \cite{modrec}.

Let $\Sm$ be the category of smooth separated $k$-schemes of finite type. Voevodsky's construction starts from  an additive category $\Cor$, whose objects are those of $\Sm$ and morphisms are finite correspondences. The category of effective geometric motives $\DM^\eff_\gm$ is then  defined to be the pseudo-abelian envelope of  the localisation of
the homotopy category $K^b(\Cor)$ of bounded complexes by two types of ``relations'':  
\begin{description}
\item[(HI)] $[X\times \A^1]\to [X]$, $X\in \Cor$; 
\item[(MV)] $[U\cap V]\to [U]\oplus [V]\to [X]$, $X, U, V \in \Cor$;
\end{description}
where, in the latter, $U \sqcup V \to X$
ranges over all open covers of $X$.
This makes  $\DM^\eff_\gm$ a tensor triangulated category. We denote by $M^V$ the canonical functor $\Sm\to \DM_\gm^\eff$. The following fundamental result computes some Hom groups in concrete terms:

\begin{Th}[{\cite[6.7.3]{be-vo} and \cite[Cor. 2]{allagree}}]\label{Th.Voe}
Assume that $k$ is perfect.
For $X,Y\in \Sm$, with $X$ proper of dimension $d$ and $j\in \Z$, there is a canonical isomorphism
\[\Hom_{\DM_\gm^\eff}(M^V(Y)[j],M^V(X))\simeq CH^{d}(Y\times X, j)\]
where the right hand side is Bloch's higher Chow group. 
In particular, this group is $0$ for $j<0$ 
and isomorphic to $CH^d(Y\times X)$ for $j=0$.
\end{Th}

In the present work, we construct a tensor triangulated category $\MDM_\gm^\eff$ in a parallel way. The category $\Cor$ is replaced by a category $\MCor$ whose objects are \emph{modulus pairs}, which only played an auxiliary r\^ole in \cite{rec}: this category has been studied in \cite{kmsy1}. A modulus pair $M=(\ol{M}, M^\infty)$ consists of a proper $k$-variety $\ol{M}$ and an effective Cartier divisor $M^\infty$ such that $\ol{M} - |M^\infty| \in \Sm$. A morphism from $(\ol{M}, M^\infty)$ to $(\ol{N}, N^\infty)$ is a finite correspondence from $\ol{M}-|M^\infty|$ to $\ol{N}-|N^\infty|$ which satisfies a certain inequality on Cartier divisors (Definition-Proposition \ref{dp1}).

The category $\MCor$ enjoys a symmetric monoidal structure (Definition \ref{def:tensor-mod}). The right object replacing $\A^1$ in this context turns out to be its minimal compactification
\begin{equation}\label{eq:def-bcube}
\bcube=(\P^1,\infty),
\end{equation}
the compactification of $\A^1\simeq \P^1-\{\infty \}$ 
with reduced divisor at infinity. This provides an analogue of (HI):
\begin{description}
\item[(CI)] $[M\otimes \bcube]\to [M]$, $M\in \MCor$. 
\end{description}

To introduce an analogue\footnote{This is not quite accurate: see \S \ref{s6.1}.} of $\mathbf{(MV)}$, we use the cd-structure on $\MCor$ which was introduced in \cite{mzki2} and developed in \cite{kmsy2}.
This yields a Mayer-Vietoris condition in $K^b(\MCor)$ (\S \ref{s6.1}). We may then define a tensor triangulated category $\MDM_\gm^\eff$ in a similar fashion to Voevodsky (Definition \ref{def:MDMgm}), with a ``motive'' functor $M:\MCor\to \MDM_\gm^\eff$. 
We can also compute the Hom-groups of $\MDM_\gm^\eff$:

\begin{Th}[see Theorem \ref{thm:jLC-special}] \label{T3}
For any $\sX, \sY \in \MCor$
and $i \in \Z$, we have an isomorphism
\begin{equation}\label{eq;T3} 
\Hom_{\MDM^{\eff}_{\gm}}(M(\sX), M(\sY)[i])
\simeq \colim_{\sX' \in \ulSigma^\fin \downarrow \sX} \mathbb{H}^i_{\Nis}(\ol{\sX}', RC_*^\bcube (\sY)_{\sX'}).
\end{equation}
\end{Th}

Here $\ulSigma^\fin \downarrow \sX$ denotes a certain category of morphisms $\sX'\to \sX$ which (in particular) induce isomorphisms on the interiors (see Theorem \ref{t1.1}), $RC_*^\bcube (\sY)$ is the \emph{derived Suslin complex} of the modulus pair $\sY$ (see Definition \ref{d5.1}), 
and $RC_*^\bcube (\sY)_{\sX'}$ denotes its restriction to $\ol{\sX}'_\Nis$
(see Definition \ref{defn:smallpresheaf}).
Briefly, $RC_*^\bcube (\sY)$ is defined like the Suslin complex, 
with three differences: a) we use $\bcube$ instead of $\A^1$; b) we use a cubical version instead of Suslin-Voevodsky's simplicial version (see Remarks \ref{r8.1} and \ref{r4.1}  for an important comment on this point); c) we use derived internal Homs instead of classical internal Homs. 

Recall that a key technical tool of Voevodsky for proving Theorem \ref{Th.Voe} is to embed $\DM_\gm^\eff$ into a larger triangulated category $\DM^\eff$ of \emph{motivic complexes}. The situation is similar here:  $\MDM_\gm^\eff$ is embedded into a category $\MDM^\eff$.  This is how the derived Suslin complex $RC_*^\bcube (\sY)$ arises. 

On the other hand, there is a canonical ``forgetful'' functor $\omega:\MCor\allowbreak\to \Cor$ sending $(\ol{X}, X^\infty)$ to $\ol{X} - |X^\infty|$, whence a comparison between our theory and Voevodsky's. This is summarised in the following diagram, assuming $k$ perfect:
\begin{equation}\label{eq:intro}
\begin{CD}
\MCor @>{M}>> \MDM_\gm^\eff @>{\iota}>> \MDM^\eff\\
@V{\omega}VV @V{\omega_{\eff,\gm}}VV @V{\omega_\eff}VV
\\
\Cor @>{M^V}>> \DM_\gm^\eff @>{\iota}>> \DM^\eff,
\end{CD}
\end{equation}
in which the functors denoted $\iota$ are fully faithful. The yoga of Thomas\-on-Neeman of compactly generated categories \cite{neeman} shows that  $\omega_\eff$ has a right adjoint 
$\omega^\eff : \DM^\eff \to \MDM^\eff$, 
and we have:

\begin{Th}[See Th. \ref{t8.1} and Cor. \ref{c8.1}]\label{T4} a) Let $X$ be a smooth proper $k$-variety. Then $\omega^\eff M^V(X)=M(X,\emptyset)$.\\
b) If $p$ is the exponential characteristic of $k$, then  $\omega^\eff(\DM_\gm^\eff[1/p])\subset \MDM_\gm^\eff[1/p]$. The functor $\omega^\eff$ is symmetric monoidal.
\end{Th}

Note that $\omega^\eff$ is fully faithful (Propositions \ref{prop:eq6.2} and \ref{prop:eq6.2a}). As a corollary, we get the following partial analogue of Theorem \ref{Th.Voe}:

\begin{Th}[See Cor. \ref{c8.2}]\label{Th1.intro} 
Suppose $k$ is perfect. Let $X$ be a smooth proper $k$-variety of dimension $d$.  We have a canonical isomorphism:
\[\Hom_{\MDM_\gm^\eff}(M(\sY)[j],M(X,\emptyset))\simeq CH^{d}((\ol{\sY}-|\sY^\infty|)\times X, j)\]
for any modulus pair $\sY=(\ol{\sY},\sY^\infty)$ and $j \in \Z$. 
\end{Th}

Though we consider proper modulus pairs in the above, we can also construct a theory of motives with modulus for pairs $(\ol{X},X^\infty)$ with $\ol{X}$ not necessarily proper.  They come with a reasonable topology (see \cite{kmsy1}).  By a similar construction to the one above, this leads to triangulated categories $\ulMDM_\gm^\eff$, $\ulMDM^\eff$ of motives with modulus for non-proper modulus pairs. Even though $\MDM^\eff$ seems to be the central object of interest (e.g. it is the sheaf theory on \emph{proper} modulus pairs which has a relationship with reciprocity sheaves in \cite{modrec}), it is impossible to avoid developing $\ulMDM^\eff$ at the same time. This leads to a regrettably heavy exposition, for which we apologize to the reader. Besides, the non-proper version is used in \cite{shuji-purity} as an important technical tool. 
\subsection*{Relationship with previous work} This work completes the revision of \cite{motmod}, which was started in \cite{kmsy1} and \cite{kmsy2}. For the readers aware of this first version, the categories $\MDM_\gm^\eff$ and $\MDM^\eff$ are the same as in \cite{motmod}, as well as their non-proper versions. (The constructions given here are different and simpler.) The only difference with the results of \cite{motmod}  is in Theorem \ref{T3}, where the formula for the Hom group given in \cite{motmod} missed the direct limit.

\subsection*{Perspectives}

The story of motives with modulus does not stop here: there are many further things to explore, some being explored now. We quote a few:

\begin{itemize}
\item Extend more of Voevodsky's results to this modulus context. See already Section \ref{s7} (and the references therein), and \cite{shuji-purity}. In particular, extend Voevodsky's theorem on strict $\A^1$-invariance \cite[Th. 5.6]{voepre} to the $\bcube$-invariant context. See Example \ref{exB.1}, as well as \cite[Question 1]{kmsy1}, and \cite[Th. 0.6]{shuji-purity} for a partial result.
\item A comparison with the category of log-motives of Binda-Park-\O stv\ae r \cite{logmot}; see already \cite{saitolog}.
\item Versions for other topologies, notably topologies related to the \'etale topology on schemes. The theory developed here is intrinsically restricted to the Nisnevich topology, via the theory of cd-structures. One may think of the model-theoretic techniques of \cite{logmot} --- as soon as one has guessed the right definition of topologies on modulus pairs!
\item A $\bcube$-homotopy theory of modulus pairs similar to that of \cite{mv}. It should be easy to develop from the material here. 

\item Contrary to $\DM_\gm$, there is evidence that the category $\MDM_\gm$ obtained by $\otimes$-inverting the Tate object is \emph{not} rigid, see Proposition \ref{p6.2}. Another fact is that many cohomology theories, starting with higher Chow groups with modulus, satisfy, not the modulus condition, but its \emph{opposite}. This suggests that one should construct an even larger category based on ``modulus triples'' (two Cartier divisors at infinity with opposite modulus conditions). Work in this direction has been made by Binda \cite{binda} in the context of $\bcube$-homotopy theory (as above), and by Ivorra-Yamazaki \cite{iy1,iy2} in the additive context.
\item Of course, develop the various theories over a base.
\end{itemize}

\subsection*{Acknowledgements} Part of this work was done while three of the authors stayed at the University of Regensburg, supported by the SFB grant ``Higher Invariants". Another part was done in a Research in trio in CIRM, Luminy. 
Yet another part was done while the fourth author was visiting IMJ-PRG, supported by Fondation Sciences Ma\-th\'e\-ma\-ti\-ques de Paris. We are grateful to the support and hospitality received in all places. 

Independently of our work, Moritz Kerz conjectured in \cite{Kerz} the existence of 
a category of motives with modulus and gave a list of expected properties.
His conjectures are closely related to our construction. We thank him for communicating his ideas, which were very helpful to us. 

We thank Joseph Oesterl\'e for his help in the proof of Lemma \ref{lem:oesterle}, Joseph Ayoub for Remarks \ref{r4.3} and \ref{r8.1}, and Shane Kelly for useful discussions on \S \ref{sec-tec}. 
Thanks are also due to the referees for their careful reading and valuable comments.

We would like to dedicate this work to the memories of Michel Raynaud and Laurent Gruson, whose article \cite{rg} was essential for developing this theory (as it was previously for that of Voevodsky).

\enlargethispage*{20pt}

\subsection*{Organisation of this paper} In Section \ref{sec-basics}, we review part of our previous work on modulus pairs which will be used here. (More reminders are inserted in the sequel at appropriate places.) In Section \ref{s.tensor}, we introduce a new ingredient: the tensor structure. In Section \ref{sec-motmod}, we give an elementary construction of the triangulated categories of motives with modulus in the spirit of \cite[\S 4.2]{birat-tri}, and prove their basic properties in Theorem \ref{thm:DMgm-DM-full-faith}. In Section \ref{s.bg}, we describe $\MDM^\eff$ and $\ulMDM^\eff$ in terms of sheaves with transfers: this yields in particular a first computation of Hom groups in terms of Nisnevich cohomology (Corollary \ref{c4.3}). In Section \ref{s.sc}, we use the theory of intervals from Appendix \ref{section:interval} to reformulate this computation in more concrete terms (Theorem \ref{thm:jLC-special}); we also prove that the natural functor $\MDM^\eff\to \ulMDM^\eff$ is fully faithful (Th. \ref{t5.1}). In Section \ref{s.comp}, we compare our categories with Voevodsky's categories, as well as with the category of Chow motives. The most important result there is Theorem \ref{t8.1}. In Section \ref{s7}, we prove various results on $\MDM^\eff$ and $\ulMDM^\eff$ similar to those of Voevodsky in \cite{voetri}, and include for the reader's convenience some which were proven by other authors.

There are two appendices. Appendix \ref{sect:app} gathers new technical categorical results. Appendix \ref{section:interval} is central to the results of Section \ref{s.sc}: it generalises Voevodsky's theory of intervals (in a cubical version) to the case of symmetric monoidal categories. Its most important results are Theorems \ref{thm:10.34} and \ref{t7.1} (the latter is used in the proofs of Theorems \ref{t5.1} and \ref{t8.1}).

\subsection*{Notation and terminology}
Throughout this paper, we fix a base field $k$.
We denote by $\Sch$ the category of separated schemes of finite type over $k$, and by $\Sm$ the full subcategory of $\Sch$ consisting of those objects which are smooth over $k$.
For any Cartier divisor $D$ on a scheme $X$, the support of $D$ is denoted by $|D|$.
We write $\Cor$ for Voevodsky's 
category of finite correspondences \cite{voetri}.

Let $\Sq$ be the square of the category  $\{ 0 \to 1 \}$, depicted as
\[
\xymatrix{
00 \ar[r] \ar[d] &01 \ar[d]\\
10 \ar[r] & 11.
}
\]
For any category $\sC$, denote by $\sC^\Sq$ for the category of  functors from $\Sq$ to $\sC$.
A functor $f : \sC \to \sC'$  induces a functor $f^\Sq : \sC^\Sq \to {\sC'}^\Sq$.

A \emph{$\otimes$-category} is a unital symmetric monoidal category; a \emph{$\otimes$-functor} $F$ between $\otimes$-categories is a strong unital symmetric monoidal functor (this means that the maps $F(A)\otimes F(B)\to F(A\otimes B)$ are all isomorphisms).

\section{Review of modulus pairs} \label{sec-basics}

\subsection{The categories $\ulMCor$, $\MCor$, $\ulMSm$ and $\MSm$}

\begin{defn}\label{def:modpair}
A \emph{modulus pair} is a pair $M=(\ol{M},M^\infty)$ such that $\ol{M} \in \Sch$, $M^\infty$ is an effective Cartier divisor on $\ol{M}$, and $M^\o := \ol{M} - |M^\infty| \in \Sm$.
We call $\ol{M}$, $M^\infty$ and $M^\o$ the \emph{ambient space}, \emph{modulus} and \emph{interior} of $M$, respectively. 
A modulus pair $M$ is called \emph{proper} if its ambient space $\ol{M}$ is proper over $k$.
\end{defn}

The ambient space $\ol{M}$ is reduced for any modulus pair $M$ \cite[Rem. 1.1.2 (3)]{kmsy1}.

As in Voevodsky's theory \cite{voetri}, 
we define two types of categories with the same objects
out of modulus pairs.
One, analogous to $\Sm$, is used as a support for Grothendieck topologies. The other, analogous to $\Cor$, is used to define a transfer structure on (pre)sheaves. We start with the latter:

\begin{def-prop} \label{dp1}
Let $M$ and $N$ be modulus pairs.  An \emph{elementary modulus correspondence} $\alpha : M \to N$ is an elementary correspondence $\alpha^\o : M^\o \to N^\o$ between the interiors which satisfies the following properties:
\begin{description}
\item[Properness condition] Let $\ol{\alpha}$ be the scheme-theoretic closure of $\alpha^\o$ in $\ol{M} \times \ol{N}$. Then $\ol{\alpha}$ is proper over $\ol{M}$. (This is automatic  if $N$ is a proper modulus pair.)
\item[Modulus condition] Let $\nu:\ol{\alpha}^N \to \ol{M} \times \ol{N}$ be the composition of the normalisation $\ol{\alpha}^N \to \ol{\alpha}$ 
with the inclusion $\ol{\alpha} \hookrightarrow \ol{M} \times \ol{N}$. Then we have $\nu^\ast (M^\infty \times \ol{N}) \geq \nu^\ast (\ol{M} \times N^\infty )$ (inequality of Cartier divisors).
We say $\alpha$ is \emph{admissible} if this condition is satisfied.
\end{description}
A \emph{modulus correspondence} $\alpha : M \to N$ is a formal $\Z$-linear sum of elementary modulus correspondences. 
By \cite[Def. 1.3.1]{kmsy1}, the composition of modulus correspondences 
(regarded as finite correspondences \cite[\S 2.1]{voetri}) 
is again a modulus correspondence. 
Therefore, modulus pairs and modulus correspondences define the \emph{category of modulus correspondences}, denoted by $\ulMCor$.
The full subcategory of $\ulMCor$ consisting of proper modulus pairs is denoted by $\MCor$.
\end{def-prop}

\begin{defn} \label{def-prop1.1.3} We denote by $\ulMSm$ the category whose objects are modulus pairs and in which a morphism $f:M\to N$ is a morphism $f^\o\in \Sm(M^\o, N^\o)$ whose graph defines an elementary modulus correspondence as in Definition \ref{dp1}; we write $\MSm$ for the full subcategory of $\ulMSm$ consisting of proper modulus pairs.
\end{defn}

There is a commutative diagram of natural functors
\begin{equation}\label{diag1}
\begin{gathered}
\xymatrix{
 \MSm \ar[rd]^{\omega_s} \ar[dd]_{\tau_s} \ar[rrr]^{c} &&& \MCor \ar[dd]^{\tau} \ar[ld]_{\omega} \\
& \Sm \ar[r]^{c^V} & \Cor \\
\ulMSm \ar[ru]_{\ulomega_s} \ar[rrr]^{\ul{c}} &&& \ulMCor \ar[lu]^{\ulomega}.
}
\end{gathered}
\end{equation}

Here, the vertical functors $\tau$ and $\tau_s$ are the full embeddings mentioned above, and the horizontal functors $c$ and $\ul{c}$ are graph functors similar to Voevodsky's graph functor $c^V$; the diagonal functors are all induced by $M\mapsto M^\o$ (retain the interior). By  \cite[Lemma 1.5.1, Th. 1.5.2 and Prop. 1.10.4]{kmsy1} (see also \cite[Cor. 2.2.5]{mzki2}), we have the following important results:

\begin{thm}\label{t2.1} The functors $\omega$, $\ulomega$ and  $\tau$ have pro-left adjoints $\omega^!$, $\lambda$ and $\tau^!$, given respectively by 
\[\omega^! X = "\lim\nolimits"_{M\in\Sigma\downarrow X} M,\quad \lambda X = (X,\emptyset),\quad \tau^! M="\lim\nolimits"_{N\in \Comp(M)} N\]
($\lambda$ is an honest left adjoint), where $\Sigma=\{u\mid \omega(u) \text{ is an isomorphism}\}$ (it admits a calculus of right fractions), and $\Comp(M)$ is the category whose objects are pairs
$(N,j)$ consisting of a modulus pair $N=(\Nb,N^\infty)\allowbreak\in\MSm$ 
equipped with a dense open immersion $j:\Mb\hookrightarrow \Nb$ such that
$N^\infty=M_N^\infty+C$
for some effective Cartier divisors $M_N^\infty, C$ on $\ol{N}$
satisfying
$\ol{N} - |C| = j(\ol{M})$
and  $j^*N^\infty = M^\infty$.

Similarly for $\omega_s$, $\ulomega_s$ and $\tau_s$ (same formulas). 
\end{thm}

\begin{thm}\label{thm:fiberprod-mod}
Let $f_1 : M_1 \to N$ and $f_2 : M_2 \to N$ be two morphisms in $\ulMSm$ (resp. $\MSm$), and assume that $M_1^\o \times_{N^\o} M_2^\o$ belongs to $\Sm$.
Then the fibre product $M_1 \times_N M_2$ exists in $\ulMSm$ (resp. in $\MSm$). (In other terms, $\omega_s$ and $\ulomega_s$ reflect fibre products.) Moreover, $\ulMSm$ and $\MSm$ have the final object $\un:=(\Spec k,\emptyset)$. In particular, finite products exist in these categories.
\end{thm}

\subsection{The category $\ulMSm^\fin$}

There is another category, which plays a technically important r\^ole:

\begin{defn}\label{def:1.1.2} We write $\ulMSm^\fin$ for the subcategory of $\ulMSm$ with the same objects and such that $f\in \ulMSm(M,N)$ belongs to $\ulMSm^\fin(M,N)$ if and only if the rational map $\ol{M}\to \ol{N}$ defined by $f^\o$ is a \emph{morphism}. We say that $f$ is \emph{ambient}.\\
We write $\ul{b}_s:\ulMSm^\fin \to \ulMSm$ for the inclusion functor. 
\end{defn}

Note that, for an ambient morphism $f$, the properness condition is trivial and the modulus condition simplifies to
\begin{equation}\label{eq1.3}
\nu^\ast M^\infty \geq \nu^\ast \ol{f}^\ast N^\infty .
\end{equation}
where $\nu : \ol{M}^N \to \ol{M}$ is the normalisation. 

\begin{defn}\label{deff}
We say that $f$ is \emph{minimal} if there is equality in \eqref{eq1.3}. 
\end{defn}

By  \cite[Prop. 1.9.2]{kmsy1}, we have

\begin{thm}\label{t1.1} Let $\ulSigma^\amb$ be the class of minimal morphisms $f : M \to N$ in $\ulMSm^\amb$ such that  $f^\o : M^\o \to N^\o$ is an isomorphism in $\Sm$ and $\ol{f}:\ol{M}\to \ol{N}$ is proper. Then $\ulSigma^\amb$ admits a calculus of right fractions, any morphism in $\ulSigma^\amb$ becomes invertible in $\ulMSm$ and the induced functor $(\ulSigma^\amb)^{-1}\ulMSm^\amb\to \ulMSm$ is an equivalence of categories.
\end{thm}

\begin{rk} Contrary to Theorem \ref{thm:fiberprod-mod}, fibre products (or even cartesian products) are not representable in $\ulMSm^\fin$ in general.
\end{rk}

\subsection{Topologies on modulus pairs}\label{sec:topmod}

The categories $\ulMSm$ and $\MSm$ have Grothendieck topologies which are derived from the Nisnevich topology on $\Sm$.
We recall their definitions from \cite{kmsy1} and \cite{mzki2}.

First, we consider the ``naive'' topology on $\ulMSm^\amb$.

\begin{defn}\label{d1.3.1}
An \emph{$\ulMV^\amb$-cover} is an ambient morphism $f : U \to M$ in $\ulMSm^\amb$ which is minimal and such that the underlying morphism $\ol{f} : \ol{U} \to \ol{M}$ in $\Sch$ is a Nisnevich cover. 
The Grothendieck topology on $\ulMV^\amb$ generated by $\ulMV^\amb$-covers is called the \emph{$\ulMV^\amb$ topology}.
\end{defn}

There is another characterisation of this topology from the ``cd-structure'' point of view (see \cite{cdstructures} for the definitions and properties of cd-structures).

\begin{defn}
Let $P_{\ulMV^\amb}$ be the cd-structure on $\ulMSm^\amb$ consisting of commutative squares in $\ulMSm^\amb$
\[\xymatrix{
W \ar[r] \ar[d] & V \ar[d] \\
U \ar[r] & M
}\]
such that all arrows are minimal and the underlying square of schemes 
\[\xymatrix{
\ol{W} \ar[r] \ar[d] & \ol{V} \ar[d] \\
\ol{U} \ar[r] & \ol{M}
}\]
is a distinguished Nisnevich square.
(We always take $\ol{U} \to \ol{M}$ to be an open immersion.)
An element of $P_{\ulMV^\amb}$ is called an \emph{$\ulMV^\amb$-sqaure}.
\end{defn}

By ``transport of structure'' of the Nisnevich cd-structure, which is complete, regular and bounded in the sense of \cite{cdstructures}, we obtain 

\begin{prop}\label{prop:MVamb}
The $\ulMV^\amb$ topology coincides with the topology associated with the cd-structure $P_{\ulMV^\amb}$.
Moreover, the latter is complete, regular and bounded.
\end{prop}

Next, we introduce topologies on $\ulMSm$ and $\MSm$.
The former is easily defined as follows.

\begin{defn}\label{def:uMVsquare}
Let $P_{\ulMV}$ be the ``smallest'' cd-structure on $\ulMSm$ which contains the images of all squares in $P_{\ulMV^\amb}$ under the functor $\ul{b}_s:\ulMSm^\amb \to \ulMSm$ of Theorem \ref{t1.1}.
In other words, a commutative square in $\ulMSm$ belongs to $P_{\ulMV}$ if and only if it is isomorphic in $\ulMSm^\Sq$ to the image of an $\ulMV^\amb$-square (see \textbf{Notation} for the definition of $\Sq$). 
An element of $P_{\ulMV}$ is called an \emph{$\ulMV$-square}.

The Grothendieck topology on $\ulMSm$ associated with the cd-structure $P_{\ulMV}$ is called the \emph{$\ulMV$ topology}.
\end{defn}

\begin{remark}
For any $\ulMV^\fin$-square or $\ulMV$-square $S$, the square $S^\o$ in $\Sm^\o$ is a distinguished Nisnevich square. 
\end{remark}

By \cite[Prop. 3.2.2]{kmsy1}, we have

\begin{thm}
The cd-structure $P_{\ulMV}$ is complete and regular. 
\end{thm}

The definition of the topology on $\MSm$ is a bit tricky. 
It is designed to satisfy completeness and regularity, and to be compatible with the $\ulMV$ topology. 
First we need to recall the ``off-diagonal functor'' from \cite{mzki2}.

\begin{defn}
Let $\ulMEt$ be the category whose objects are morphisms $f : U \to M$ in $\ulMSm$ such that $f^\o$ is \'etale, a morphism from $f:U \to M$ to $g:U' \to M'$ being given by a pair of morphisms $s : U \to U'$, $t : M \to M'$ which are compatible with $f$ and $g$ and such that $s^\o$ and $t^\o$ are open immersions. 
Let $\MEt$ be the full subcategory of $\ulMEt$ consisting of those $f:M \to N$ with $M,N \in \MSm$.
\end{defn}

By \cite[Th. 3.1.3]{mzki2}, we have

\begin{prop}
There exists a functor $\OD : \ulMEt \to \ulMSm$
together with a natural isomorphism 
\[
U \sqcup \OD (f) \iso U \times_M U,
\]
for each $(f : U \to M) \in \ulMEt$, 
where $\sqcup$ denotes coproduct and the right hand side denotes the fiber product in $\ulMSm$.
This functor restricts to a functor $\MEt \to \MSm$.
\end{prop}

\begin{defn}\label{def:MV-sq}
Let $P_{\MV}$ be the cd-structure on $\MSm$ consisting of those commutative squares $T$ of the form
\begin{equation}\label{eq:T}\begin{gathered}\xymatrix{
T(00) \ar[r] \ar[d]_q & T(01) \ar[d]^p \\
T(10) \ar[r] & T(11)
}\end{gathered}\end{equation}
which satisfy the following properties:
\begin{enumerate}
\item $T$ is cartesian in $\MSm$.
\item There exists an $\ulMV$-square $S$ and a morphism $\iota  : S \to T$ in $\ulMSm^\Sq$ such that $\iota (11) : S(11) \to T(11)$ is an isomorphism in $\ulMSm$, and $\iota (ij)^\o : S(ij)^\o \to T(ij)^\o$ is an isomorphism in $\Sm$ for any $(ij) \in \Sq$.
In particular, the square $S^\o \cong T^\o$ in $\Sm^\Sq$ is a distinguished Nisnevich square. 
\item The morphism $\OD (q) \to \OD (p)$, which is induced by the funtoriality of $\OD$, is an isomorphism.
\end{enumerate}
An element of $P_{\MV}$ is called an \emph{$\MV$-square}.
The topology associated with the cd-structure $P_{\MV}$ is called the \emph{$\MV$ topology}.
\end{defn}

\begin{ex}\label{kp3} Let $X$ be proper and let $D,D_1,D_2,D'$ be effective Cartier divisors on $X$ such that
\begin{gather*}
X-D \text{ is smooth }\label{keqd}\\
D\le D_i\le D'\label{keqa}\\
|D_1-D|\cap |D_2-D| =\emptyset\label{keqb}\\ 
D'-D_2 =D_1-D.\label{keqc}
\end{gather*}
Then 
\[
T=\begin{CD}
(X,D')@>>> (X,D_1)\\
@VVV @VVV\\
(X,D_2)@>>> (X,D)
\end{CD}\] 
is an $\MV$-square. Indeed, (1) holds by \cite[Lemma 1.10.1 and Prop. 1.10.4]{kmsy1}. Let $\ol{S}(01)=X-|D_1-D|$, $\ol{S}(10) = X-|D_2-D|$, $\ol{S}(00) = X-|D'-D|$ and $S(ij)^\infty = j(ij)^*T(ij)^\infty$ where $j(ij)$ is the inclusion $\ol{S}(ij)\inj X$. This yields a square $S$ as in (2), and (3) is trivial since $T^\o$ is a Zariski square.
\end{ex}

By  \cite[Th. 4.3.1, 4.4.1]{mzki2}, we have

\begin{thm}
The cd-structure $P_{\MV}$ is complete and regular.
\end{thm}

(Condition (3) in Definition \ref{def:MV-sq} is crucial for the proof of regularity.)

We now recall the main result of \cite{KaMi2}, its Theorem 1.5.6. It will be used in the proof of Theorem \ref{prop:DGguide} (see (iv) in \S \ref{s3.8}). Recall from \cite[Def. 1.5.3]{KaMi2} that the category $\Comp (M)$ of  Theorem \ref{t2.1} can be extended to squares of modulus pairs.

\begin{thm}\label{t1.2} Let $S\in \ulMSm^\Sq$, and let $\Comp^{\MV} (S)$ denote the full subcategory of $\Comp (S)$ consisting of those $T$ which are $\MV$-squares. Then $\Comp^{\MV} (S)$ is cofinal in $\Comp (S)$.
\end{thm}

We shall need the following lemma in the proof of Theorem \ref{thm:DMgm-DM-full-faith} (4) below.

\begin{lemma}\label{l1.6} Let $T\in \ulMSm^\Sq$, verifying Conditions (1) and (3) of Definition \ref{def:MV-sq}. Assume that $p$ has a section $s_p$. Then $q$ has a section $s_q$, one can write $T(00) = s_q(T(10))\sqcup T'(00)$, $T(01) = s_p(T(11))\sqcup T'(11)$, and the morphism $T(00)\to T(01)$ induces an isomorphism $u:T'(00)\iso T'(01)$.
\end{lemma}

\begin{proof} The section $s_q$ is obtained from $s_p$ because $T$ is cartesian (property (1)). The decompositions exist because $p^\o$ and $q^\o$ are \'etale (cf. \cite[proof of Th. 3.1.3]{mzki2}). The morphism $u$ exists by construction; it remains to see that it is an isomorphism. But an easy computation provides decompositions
\begin{align*}
\OD(q) &\simeq T'(00)\coprod T'(00)\coprod\OD(T'(00)\to T(10)) \\
\OD(p) &\simeq T'(01)\coprod T'(01)\coprod\OD(T'(01)\to T(11)) 
\end{align*} 
respected by the isomorphism $\OD(q)\iso \OD(p)$ (property (3)). This yields the conclusion.
\end{proof}

\section{The tensor structure on modulus pairs} \label{s.tensor}

\subsection{Definition}

Recall that the tensor structure on Voevodsky's category $\DM_\gm^\eff$ comes from the product of smooth varieties. However, it turns out that the product of Theorem \ref{thm:fiberprod-mod} cannot be used to construct categories $\MDM_\gm^\eff$ and $\MDM^\eff$ with good properties. The basic reason is that the product morphism $\A^1\times \A^1\to \A^1$ used by Voevodsky to define an interval structure on $\A^1$ does not define a morphism $\bcube\times \bcube\to \bcube$, where $\bcube=(\P^1,\infty)$ (see Remark \ref{r5.2}). This and other considerations lead us to introduce another tensor structure:

\begin{defn}\label{def:tensor-mod}
For two modulus pairs $M$ and $N$, we define their \emph{tensor product} $M \otimes N$ to be
\[
M \otimes N := (\ol{M} \times \ol{N}, M^\infty \times \ol{N} + \ol{M} \times N^\infty).
\]
\end{defn}

\begin{rk}\label{r2.1}  If we pull back the modulus of Definition \ref{def:tensor-mod} by the projection $p:\Bl_{M^\infty \times N^\infty}(\ol{M} \times \ol{N})\to \ol{M} \times \ol{N}$, we get an isomorphic modulus pair with modulus $p^*(M^\infty \times \ol{N} + \ol{M} \times N^\infty)=p^\#(M^\infty \times \ol{N} + \ol{M} \times N^\infty)+2E$, where $p^\#$ denotes proper transform and $E$ is the exceptional divisor. By contrast, the cartesian product $M\times N$ is represented by $ (\Bl_{M^\infty \times N^\infty}(\ol{M} \times \ol{N}), p^\#(M^\infty \times \ol{N} + \ol{M} \times N^\infty) +E)$: this is a special case of \cite[Prop. 1.10.4 (3)]{kmsy1} and its proof. See also Remark \ref{w2.1} below.
\end{rk}

Definition \ref{def:tensor-mod} provides the categories  $\MP$, $\ulMP$, $\MCor$, $\ulMCor$ and $\ulMSm^\fin$ 
of Definitions \ref{dp1}, \ref{def-prop1.1.3} and \ref{def:1.1.2} with symmetric mon\-oid\-al structures
with unit $\un=(\Spec k,\emptyset)$, for which the various functors between them are $\otimes$-functors.  To see this, we have to check:

\begin{lemma} Let $f\in \ulMCor(M_1,N_1)$ and $g\in \ulMCor(M_2,N_2)$. Consider the tensor product correspondence $f\otimes g\in \Cor(M_1^\o\times M_2^\o,N_1^\o\times N_2^\o)$. Then $f\otimes g\in \ulMCor(M_1\otimes M_2, N_1\otimes N_2)$.
\end{lemma}

\begin{proof} We may assume that $f$ and $g$ are given by integral cycles 
$Z\subset M_1^\o\times N_1^\o$ 
and $T\subset M_2^\o\times N_2^\o$.
Then $f\otimes g$ is given by the product cycle $Z\times T$. 
Let 
$\ol{Z}^N \to \ol{Z}$ 
be the normalizations of the closures $\ol{Z}$ of $Z$,
and similarly for $\ol{T}^N \to \ol{T}$.
By hypothesis, we have
\[(p_1^Z)^*M_1^\infty\ge (p_2^Z)^*N_1^\infty, \quad (p_1^T)^*M_2^\infty\ge (p_2^T)^*N_2^\infty,\]
where $p_1^Z$ is the composition
$\ol{Z}^N \to \ol{Z} 
\subset \ol{M}_1 \times \ol{N}_1 \to \ol{M}_1$,
and likewise for $p_2^Z, p_1^T, p_2^T$.
Hence:
\begin{multline*}
(p_1^Z\times p_1^T)^*(M_1^\infty\times \ol M_2+\ol M_1\times M_2^\infty)=(p_1^Z)^*M_1^\infty \times \ol T^N + \ol Z^N\times (p_1^T)^*M_2^\infty\\
\ge (p_2^Z)^*N_1^\infty \times \ol T^N + \ol Z^N\times (p_2^T)^*N_2^\infty=(p_2^Z\times p_2^T)^*(N_1^\infty\times \ol N_2+\ol N_1\times N_2^\infty).
\end{multline*}
We conclude that $Z\times T$ is admissible,
because the projection $(\ol Z\times \ol T)^N\to \ol Z\times \ol T$ factors through $\ol Z^N\times \ol T^N$. Finally, $\ol Z\times \ol T$ is obviously proper over $\Mb_1\times \Mb_2$.
\end{proof}

To conclude checking that we have indeed defined symmetric mono\-id\-al structures, we need to verify identities of the form $f\otimes (g\circ h)=(f\otimes g)\circ (f\otimes h)$, and to define associativity, commutativity and unit constraints. The first point holds because it holds in $\Cor$ and $\Sm$; for the second one, we leave to the reader the pleasure to check that the constraints of the symmetric monoidal structure on $\Cor$ are proper and admissible, hence induce similar ones on $\ulMCor$, etc., which enjoy the correct identities. The (symmetric) monoidality of the various functors is tautological.

\begin{remark}[see also Remark \protect{\ref{r2.1}}]\label{w2.1}  Given two modulus pairs $M,N$, the canonical morphisms $M\to \un$, $N\to\un$ give morphisms $M\otimes N\to M\otimes \un=M$, $M\otimes N\to \un\otimes N=N$. 
The universal property of the product then yields a natural transformation
\begin{equation}\label{eq1.1}
M\otimes N\to M\times N.
\end{equation}

Take $M=N$. The right hand side comes with a diagonal morphism $M\to M\times M$, corresponding to the diagonal morphism $\Delta:M^\o\to M^\o\times M^\o$  (indeed, products commute with $\ulomega_s$ and $\omega_s$ since they have pro-left adjoints). If \eqref{eq1.1} were an isomorphism, $\Delta$ would define a morphism $M\to M\otimes M$; but Definition \ref{def:tensor-mod} and the modulus condition show that this happens if and only if $M^\infty=0$. Conversely, it can easily be shown that \eqref{eq1.1} is an isomorphism if $M^\infty=0$ or $N^\infty=0$.
\end{remark}

\begin{prop}\label{prop:monoidal-tauomega}
All functors in Diagram \eqref{diag1} are symmetric mono\-id\-al. This also applies to the functors of Theorem \ref{t2.1}, and to $\ul{b}_s$ in Definition \ref{def:1.1.2}. 
\end{prop}

\begin{proof}
For Diagram \eqref{diag1} and $\ul{b}_s$, it is easily deduced from the construction of the functors. For the functors $\omega^!$ and $\omega_s^!$, recall from \cite[Def. 1.4.1]{kmsy1} the functors $(-)^{(n)}$ given by $M^{(n)} = (\ol M,nM^\infty)$. Clearly, $(-)^{(n)}$ is monoidal. On the other hand, by \cite[Lemma 1.7.4 b)]{kmsy1}, an inverse system defining $\omega^! X$ for $X\in \Sm$ is given by $(M^{(n)})_{n\ge 1}$ for any $M$ such that $\omega(M)=X$; this proves the claim in this case, and similarly for $\omega_s^!$.

We now show the monoidality of $\tau^!$, arguing as in the case of $\omega^!$ (although we cannot quite use the functors $(-)^{(n)}$ here). Let $M\in \ulMCor$. We use the category $\Comp(M)$ of Theorem \ref{t2.1}. Take $N\in \Comp(M)$ and write $N^\infty = M^\infty_N +C$ as in loc. cit. Define $\Comp(N,M)$ as the full subcategory of $\Comp(M)$ consisting of those $P$ such that $\ol{P}=\Nb$ (compatibly with the open immersions $\Mb\inj \Nb$, $\Mb\inj \ol P$) and $P^\infty = M^\infty_N +nC$ for some $n>0$.
(Strictly speaking, $\Comp(N, M)$ depends on the choice of the decomposition $N^\infty = M^\infty_N +C$.) The proof of \cite[Claim 1.8.4]{kmsy1} shows that $\Comp(N,M)$ is cofinal in $\Comp(M)$. If $M'\in \ulMCor$ is another object and  $N'\in \Comp(M')$ with a decomposition ${N'}^\infty={M'}^\infty_{N'}+C'$, then $N\otimes N' \in \Comp(M\otimes M')$ as
$(N\otimes N')^\infty= (M^\infty_N \times \ol{N'}+ \ol{N} \times {M'}^\infty_{N'})+ (C\times \ol{N'}+ \ol{N} \times C')$, and it is easy to see that the obvious functor
\[\Comp(N,M)\times \Comp(N',M')\to \Comp(N\otimes N',M\otimes M')\]
is cofinal. The same proof applies to $\tau_s^!$.
\end{proof}

\enlargethispage*{30pt}

\subsection{Tensor product and cd-structures}
The following lemma will be used later to define tensor structures for motives with modulus (see Theorem \ref{thm:DMgm-DM-full-faith} (3)). 

\begin{lemma}\label{lem:tensor-cd} \ 
\begin{enumerate}
\item 
If $S\in \ulMSm^\Sq$ is cartesian, so is $S \otimes M$ for any $M \in \ulMSm$.

\item 
If $T\in \MSm^\Sq$ is cartesian, so is $S \otimes M$ for any $M \in \MSm$.

\item 
If $S$ is an $\ulMV^\amb$-square, then so is $S \otimes M$ for any $M \in \ulMSm^\amb$.

\item 
If $S$ is an $\ulMV$-square, then so is $S \otimes M$ for any $M \in \ulMSm$.

\item 
If $T$ is an $\MV$-square, then so is $T \otimes M$ for any $M \in \MSm$.
\end{enumerate}
\end{lemma}

\begin{proof}
(1): By the construction of fiber products in \cite[\S 2.2]{mzki2}, we may assume that the square $S$ is of the form
\[S: \quad \begin{gathered}\xymatrix{
L \ar[r]^{p_2} \ar[d]_{p_1} & M_2 \ar[d]^{f_2} \\
M_1 \ar[r]_{f_1} & N
}\end{gathered}\]
where all arrows are ambient, $E:=\ol{p}_1^\ast M_1^\infty \times_{\ol{L}} \ol{p}_2^\ast M_2^\infty$ is an effective Cartier divisor on $\ol{L}$, and $L^\infty = \ol{p}_1^\ast M_1^\infty + \ol{p}_2^\ast M_2^\infty - E$.

Set $S':=S \otimes M$ and write
\[S': \quad \begin{gathered}\xymatrix{
L' \ar[r]^{p'_2} \ar[d]_{p'_1} & M'_2 \ar[d]^{f'_2} \\
M'_1 \ar[r]_{f'_1} & N',
}\end{gathered}\]
where the arrows are obviously ambient, and there is a natural morphism $S' \to S$ in $(\ulMSm^{\amb})^\Sq$. 
In particular, we have an ambient morphism $\pi : L' \to L$.

(1): Set $E' := \ol{p}_1^{\prime \ast} M_1^{\prime \infty} \times_{\ol{L}'} \ol{p}_2^{\prime \ast} M_2^{\prime \infty}$.
Then $E'$ is an effective Cartier divisor on $\ol{L}'$. Indeed:

\begin{claim}
$E' = \ol{\pi}^\ast E + \ol{L}' \times M^\infty$.
\end{claim}

\begin{proof}
 Let $I_1$, $I_2$, $I$ be the ideals of definition of $\ol{p}_1^{\prime *} (M_1^\infty \times \ol{M})$, $\ol{p}_2^{\prime *} (\ol{M}_1 \times M^\infty)$, $\ol{L}' \times M^\infty$, respectively.
Then the ideal of definition of $E'$ is given by $I \cdot I_1 + I\cdot I_2 = I \cdot (I_1 + I_2)$.
Since $E=\ol{p}_1^\ast M_1^\infty \times_{\ol{L}} \ol{p}_2^\ast M_2^\infty$ by definition, $I_1 + I_2$ is the ideal of definition of $\ol{\pi}^* E$, hence the claim.
\end{proof}

This also shows the following: by definition we have $L^{\prime \infty} 
 = L^\infty \times \ol{M} + \ol{L}' \times M^\infty$. Thus we obtain:
\begin{align*}
L^{\prime \infty} 
&= L^\infty \times \ol{M} + \ol{L}' \times M^\infty \\
&= (\ol{p}_1^\ast M_1^\infty + \ol{p}_2^\ast M_2^\infty - E) \times \ol{M} + \ol{L}' \times M^\infty \\
&= \ol{p}_1^{\prime \ast} M_1^{\prime \infty} + \ol{p}_2^{\prime \ast} M_2^{\prime \infty} - \pi^\ast E - \ol{L}' \times M^\infty \\
&= \ol{p}_1^{\prime \ast} M_1^{\prime \infty} + \ol{p}_2^{\prime \ast} M_2^{\prime \infty} - E',
\end{align*}
which implies that $S'$ is cartesian.

(2): This is a direct consequence of (1).

(3): This is obvious by the definition of $\ulMV^\amb$-squares. 

(4): Let $S$ be an $\ulMV$-square. By definition, there exists an $\ulMV^\amb$-square $S'$ which is isomorphic to $S$ in $\ulMSm^\Sq$. 
Then we have $M \otimes S \cong M \otimes S'$ in $\ulMSm^\Sq$, and $M \otimes S'$ is an $\ulMV^\amb$-square by (3).

(5): Since $T$ is an $\MV$-square, it is cartesian, there exist an $\ulMV$-square $S$ and a morphism $S \to T$ in $\ulMSm^\Sq$ such that $S(11) = T(11)$ and, finally, $\OD (q) \cong \OD (p)$ (see \eqref{eq:T} for the notation).
Set  $S' := S \otimes M$ and $T' := T \otimes M$. 
Then $T'$ is cartesian by (1), $S'$ is an $\ulMV$-square by (4), and $S'(11) = T'(11)$.

It remains to show that $\OD (q') \cong \OD (p')$, where $p':=p \otimes M$, $q':=q \otimes M$.
For this, it suffices to show that for any $(f:U \to N) \in \ulMEt$, we have $\OD (f) \otimes M \cong \OD (f \otimes M)$. 

We use a similar argument to the one in the proof of \cite[Prop. 3.1.4]{mzki2}.
Set $f':=f\otimes M, U':=U \otimes M, N':=N \times N$.
By construction of $\OD$, we have canonical isomorphisms 
\begin{align*}
U' \sqcup \OD (f') &\cong   U' \times_{N'} U',  \\
U \sqcup \OD (f) &\cong   U \times_N U.
\end{align*}
Therefore, we obtain
\[
U' \sqcup (\OD (f) \otimes M) \cong (U \times_N U) \otimes M \cong^\dag U' \times_{N'} U' \cong U' \sqcup \OD (f'),
\]
where $\cong^\dag$ follows from (1).
It is easy to see that this isomorphism restricts to each component (indeed, it is the identity on the interiors).
Therefore, we conclude $\OD (f) \otimes M \cong \OD (f') = \OD (f \otimes M)$, finishing the proof.
\end{proof}

\section{Motives with modulus}\label{sec-motmod}

In this section, we construct the categories of motives with modulus $\ulMDM^\eff_\gm$ and $\MDM^\eff_\gm$ and their sheaf-theoretic versions $\ulMDM^\eff$, $\MDM^\eff$, and prove their fundamental properties. 

In the sequel, we write $K^b(\sA)$ (resp. $K(\sA)$) for the bounded (resp. unbounded) category of complexes on an additive category $\sA$, and $D(\sA)$ for its unbounded derived category when $\sA$ is abelian. We also write $(-)^\natural$ for pseudo-abelianisation. We also write $\frac{(-)}{(-)}$ for the (Verdier) localisation of a triangulated category. If $X$ is a subset of a triangulated category $\sT$, we write $\langle X \rangle$ (resp. $<X>$) for the thick (resp. localising) subcategory generated by $X$, i.e. the smallest triangulated subcategory of $\sT$ containing $X$ and closed under direct summands (resp. direct summands and infinite direct sums).

\subsection{Geometric motives} \label{s6.1}

Recall that Voevodsky's category $\DM^\eff_\gm$ is defined as 
\[
\DM^\eff_\gm = \left[ \frac{K^b (\Cor)}{\langle \mathrm{HI^V}, \mathrm{MV^V} \rangle} \right]^\natural,
\]
where $\mathrm{HI^V}$ and $\mathrm{MV^V}$ are the objects of the form
\begin{description}
\item[(HI$\mathbf{^V}$)] $[X \times \A^1] \by{1_X \times p} [X]$,
\item[(MV$\mathbf{^V}$)] $[U \cap V] \to [U] \oplus [V] \to [X]$,
\end{description}
where $U \sqcup V \to X$ runs over all elementary Zariski covers of all $X\in \Sm$.

The category which compares naturally with our constructions is a variant of this one (cf.  \cite[Definition 4.3.3]{birat-tri}):

\[
\DM^\eff_{\gm,\Nis} = \left[ \frac{K^b (\Cor)}{\langle \mathrm{HI^V}, \mathrm{MV^V_\Nis} \rangle} \right]^\natural,
\]
where 
\begin{align*}
&\mathbf{(MV^V_\Nis):} \quad [U \times_X V] \to [U] \oplus [V] \to [X], \quad X \in \Sm,
\end{align*}
in which $U \sqcup V \to X$ runs over all elementary Nisnevich covers of $X$, i.e., covers associated with distinguished Nisnevich squares. 

It is a highly nontrivial theorem of Voevodsky that the obvious functor $\DM^\eff_\gm\to \DM^\eff_{\gm,\Nis}$ is an equivalence of categories when $k$ is perfect; we shall not enter the corresponding issue for modulus pairs in this paper.

Our definitions of $\ulMDM^\eff_\gm$ and $\MDM^\eff_\gm$ faithfully mimic that of $\DM^\eff_{\gm,\Nis}$:

\begin{defn}\label{def:MDMgm}
We define 
\[
\ulMDM^\eff_\gm = \left[ \frac{K^b (\ulMCor)}{\langle \mathrm{\ul{CI}}, \ulMV \rangle} \right]^\natural, \quad 
\MDM^\eff_\gm = \left[ \frac{K^b (\MCor)}{\langle \mathrm{CI}, \MV \rangle} \right]^\natural,
\]
where  $\mathrm{\ul{CI}}$, $\mathrm{\ulMV}$, $\mathrm{CI}$, $\mathrm{\MV}$ are the  objects of the form
\begin{description}
\item[$\mathbf{(\ul{CI})}$] $[X \otimes \bcube] \by{1_{X} \otimes p} [X]$,
\item[$\mathbf{(\ul{MV})}$] $[U \times_X V] \to [U] \oplus [V] \to [X]$, 
\end{description}
in which $U \sqcup V \to X$ runs over all elementary $\ulMV$-covers of all $X\in \ulMSm$, i.e. those covers associated with 
$\ulMV$-squares (Definition \ref{def:uMVsquare}), and 
\begin{description}
\item[(CI)] $[X \otimes \bcube] \by{1_X \otimes p} [X]$,
\item[(MV)] $[U \times_X V] \to [U] \oplus [V] \to [X]$,
\end{description}
in which $U \sqcup V \to X$ runs over all elementary $\MV$-covers of all $X\in \MSm$, i.e. those covers associated with 
$\MV$-squares (Definition \ref{def:MV-sq}), and 
we write $\ul{M}_\gm:\ulMCor\to \ulMDM_\gm^\eff$, $M_\gm:\MCor\to \MDM_\gm^\eff$ for the corresponding canonical functors. Moreover, if $f:\sX\to \sY$ is a morphism in $\ulMCor$, we write $\ul{M}_\gm[f]$ for the image in $\ulMDM_\gm^\eff$ of the complex $[f]=[\sX]\by{f} [\sY]\in K^b(\ulMCor)$ with $[\sY]$ placed in degree $0$, so that we have a distinguished triangle
\[\ul{M}_\gm(\sX)\by{f_*} \ul{M}_\gm(\sY)\to \ul{M}_\gm[f]\by{+1}.\]
(This notation will only be used in the proof of Theorem \ref{t7.2}.)
\end{defn}

\subsection{Sheaf-theoretic motives} 
The sheaf-theoretic category of motives $\DM^\eff$ (unbounded version) is defined to be
\[
\DM^\eff = \frac{D(\NST)}{< \Z_\tr^V(\mathrm{HI^V}) >},
\]
where $\NST$ is the category of Nisnevich sheaves with transfers and  $\Z_\tr^V$ is the additive Yoneda functor.

In the following, we shall replace $\NST$ with the categories of \emph{modulus sheaves with transfers} which were studied in \cite{kmsy1} and \cite{kmsy2}. Let us recall them:

\begin{enumerate}
\item $\ulMPST$ (resp. $\MPST$) is the category of left modules (additive contravariant functors) on $\ulMCor$ (resp. $\MCor$).
\item $\ulMNST$ (resp. $\MNST$) is the full subcategory of $\ulMPST$ (resp. $\MPST$) of functors whose restriction to $\ulMSm$ (resp. $\MSm$) via the graph functor is a sheaf for the $\ulMV$ (resp. $\MV$) topology.
\end{enumerate}

All these categories are abelian Grothendieck: the ones of (1) as categories of left modules \cite[Th. A.10.2]{kmsy1}, and the ones of (2) by \cite[Th. 4.5.5]{kmsy1} and \cite[Th. 4.2.4]{kmsy2}.

Recall from \cite{kmsy1} the following notion:

\begin{defn}\label{def:stradd}
A functor between additive categories is \emph{strongly additive} if it preserves all direct sums.
\end{defn}

By \cite[Prop. 2.4.1 and Th. 4.5.5]{kmsy1} and \cite[Lemma-Def. 4.2.1, Th. 4.2.4 and Th. 5.1.1]{kmsy2} (see also \cite[Prop. A.4.1 b)]{kmsy1}), we have

\begin{prop}\label{p3.5}\
\begin{enumerate}
\item The inclusion functors $\ul{i}_\Nis:\ulMNST\inj \ulMPST$ and $i_\Nis:\MNST\inj \MPST$ have exact left adjoints $\ul{a}_\Nis$ and $a_\Nis$.
\item The inclusion functor $\tau:\MCor\inj \ulMCor$ of \eqref{diag1} induces fully faithful, exact, strongly additive functors
\[\tau_!:\MPST\to \ulMPST,\quad \tau_\Nis:\MNST\to \ulMNST\]
and we have an isomorphism $\tau_\Nis a_\Nis \simeq \ul{a}_\Nis \tau_!$. Moreover, $\tau_!$ and $\tau_\Nis$ have exact right adjoints $\tau^*$ and $\tau^\Nis$.
\end{enumerate}
\end{prop}

The additive Yoneda functors
\[\Z_\tr:\ulMCor\to \ulMPST,\quad \Z_\tr:\MCor\to \MPST\]
induce triangulated functors
\begin{equation}\label{eq:KMCor-DMPST}
  K^b(\ulMCor) \to D(\ulMPST), \quad  K^b(\MCor) \to D(\MPST).
\end{equation}

\begin{lemma}\label{claim:ff}
The functors $D(\tau_!):D(\MPST)\to D(\ulMPST)$ and $D(\tau_\Nis):D(\MNST)\to D(\ulMNST)$ are fully faithful.
\end{lemma}

\begin{proof} This follows from Propositions \ref{p3.5} and \ref{pA.3}.
\end{proof}

We now slightly diverge from Voevodsky to define $\ulMDM^\eff$ and $\MDM^\eff$. We will get back to the analogues of his definiton in Theorem \ref{prop:DGguide} below.

\begin{defn}\label{def:motmod}
We define 
\[
\ulMDM^\eff = \frac{D(\ulMPST)}{< \mathrm{\ul{CI},\mathrm{\ul{MV}}} >}, \quad \MDM^\eff = \frac{D(\MPST)}{< \mathrm{CI},\mathrm{MV} >},
\]
where  ${< \mathrm{\ul{CI},\mathrm{\ul{MV}}} >}$ and ${< \mathrm{CI},\mathrm{MV} >}$ are the localising subcategories of $D(\ulMPST)$ and $D(\MPST)$ generated by the images of ${\langle \mathrm{\ul{CI}}, \ulMV \rangle}$ and ${\langle \mathrm{CI}, \MV \rangle}$ by the functors \eqref{eq:KMCor-DMPST}.
\end{defn}

\subsection{Full embeddings and tensor structures}

\begin{thm}\label{thm:DMgm-DM-full-faith}\ 
\begin{enumerate}
\item The functors \eqref{eq:KMCor-DMPST} induce triangulated functors
\begin{equation}\label{eq:DRgm-DR}
 \ul{\iota}_\eff : \ulMDM_{\gm}^{\eff} \to \ulMDM^{\eff}, \quad   \iota_\eff : \MDM_{\gm}^{\eff} \to \MDM^{\eff}.
\end{equation}
We write $\ul{M} =  \ul{\iota}_\eff\circ \ul{M}_\gm$ and $M=\iota_\eff\circ M_\gm$ (cf. Definition \ref{def:MDMgm}).
\item The functors $ \ul{\iota}_\eff$ and $\iota_\eff $  are fully faithful with dense images; their essential images consist of
the compact objects of the target categories.  In particular, $ \ulMDM^{\eff}$ and $\MDM^{\eff}$ are compactly generated.
\item The tensor structure on $\ulMCor$ induces tensor structures on $D(\ulMPST)$, $D(\MPST)$ and all categories of {\rm (1)}. The functors of \eqref{eq:DRgm-DR} are $\otimes$-triangulated.
\item The functors $\tau$ and $\tau_!$ of \eqref{diag1} and Proposition \ref{p3.5} induce $\otimes$-triangulated functors
\[\tau_{\eff,\gm}:\MDM_\gm^\eff\to \ulMDM_\gm^\eff,\quad \tau_\eff:\MDM^\eff\to\ulMDM^\eff.\]
\item The functor $\tau_\eff$ is strongly additive and has a right adjoint $\tau^\eff$.
\end{enumerate}
\end{thm}

\begin{proof} (1) is obvious by construction. For (2), apply Theorem \ref{tA.4} and Example \ref{exA.5}. Let us prove (3). To start with, Theorem \ref{t.mon} provides tensor structures on $\ulMPST$, $\MPST$, $K^b(\ulMCor)$, $K^b(\MCor)$, $D(\ulMPST)$ and $D(\MPST)$ such that the functors of \eqref{eq:KMCor-DMPST} are tensor functors. Then Lemma \ref{lem:tensor-cd} implies that $\langle \mathrm{\ul{CI}}, \ulMV \rangle$ and $\langle \mathrm{CI}, \MV \rangle$ are $\otimes$-ideals in $K^b(\ulMCor)$ and $K^b(\MCor)$, thus so are $< \mathrm{\ul{CI},\mathrm{\ul{MV}}}>$ and $< \mathrm{CI},\mathrm{MV} >$ in $D(\ulMPST)$ and $D(\MPST)$.

Let us prove (4), which is the most difficult point. Since $D(\tau_!)$ is strongly additive, it suffices to show that $K^b(\tau)(\langle \mathrm{CI} \rangle) \subseteq \langle \mathrm{\ul{CI}} \rangle$ and $K^b(\tau)(\langle \mathrm{MV}\rangle ) \subseteq \langle \mathrm{\ulMV} \rangle$. The first inclusion is obvious. The second one is a consequence of the continuity of $\tau_s$ \cite[Th. 1]{KaMi2}, but we provide a direct and concrete proof. Let $T\in \MSm^\Sq$ be an $\MV$-square, and let $S\to \tau_s(T)$ be an  associated $\ulMV$-square (property (2) of Definition \ref{def:MV-sq}). Consider the $\Sq\times \Sq$-object of $\ulMSm$ 
\[X=S\times_{\tau_s(T(11))} \tau_s(T). \]

We can compute $\Tot(X)$ in $K^b(\ulMSm)$ in two different ways (see \cite[I.2.2]{verd} for the totalisation of multicomplexes, and loc. cit., (2.2.4.1) for a ``Fubini theorem''): we drop $\tau_s$ to lighten the notation.

\begin{thlist}
\item For every $(kl)$, $S\times_{T(11)} T(kl)$ is an $\ulMV$-square. Hence $\Tot(X)\in \langle \ulMV \rangle$.
\item For every $(ij)\ne (11)$, $T_{ij}=S(ij)\times_{T(11)} T$ is a cartesian square in which the map $S(ij)\to T(ij)$ over $T(11)$ provides a section of the projection $S(ij) \times_{T(11)} T(ij)\to S(ij)$. For $(ij)=(10)$, this projection is an isomorphism on the open parts by property (2) of Definition \ref{def:MV-sq}, so is a monomorphism, hence an isomorphism. In other words, $T_{10}(10)\to T_{10}(11)$ is an isomorphism, and the same holds for its pull-back $T_{00}$. Hence $\Tot (T_{10})$ and $\Tot(T_{00})$ are contractible. 
For $(ij)=(01)$, we apply Lemma \ref{l1.6}: the hypotheses of this lemma are satisfied.
We check Condition (3): Suppose $T$ is of the form \eqref{eq:T}.
Since $T$ is an $\MV$-square by definition, we have $\OD (q) \cong \OD(p)$.
Moreover, by \cite[Prop. 3.1.4]{mzki2}, the off-diagonal $\OD$ is stable under pullbacks.
More precisely, if we set $p':=p \times_{T(11)} S(01)$ and $q':=q \times_{T(11)} S(01)$, we have $\OD(p') = \OD(p) \times_{T(11)} S(01)$ and $\OD(q') = \OD(q) \times_{T(11)} S(01)$. 
Therefore, we have $\OD(q') \cong \OD(p')$.
It follows that $\Tot(T_{01})$ is also contractible. \emph{In fine}, $\Tot (T_{ij})$ is contractible except for the square $T_{11} = T$.
\end{thlist}

This shows that $\Tot \tau_s(T)\in \langle \ulMV \rangle$ (more precisely, $\Tot \tau_s(T)$ belongs to the triangulated $\otimes$-ideal generated by $\Tot(S)$).

Finally, we prove (5). Let $\pi:D(\MPST)\to \MDM^\eff$ and $\ul{\pi}:D(\ulMPST)\to \ulMDM^\eff$ be the projection functors. By Corollary \ref{cA.1}, they have right adjoints $i$ and $\ul{i}$ which themselves have  right adjoints. In particular, all these functors are strongly additive. This, and the strong additivity of $D(\tau_!)$ easily implies that $\tau_\eff$ is strongly additive; the existence of $\tau^\eff$ then follows from (2) (compact generation of $\MDM^\eff$), Theorem \ref{tA.3} and Lemma \ref{lA.3}.
\end{proof}

\bigskip\bigskip

\begin{remarks}\label{r3.1}\
\begin{enumerate}
\item Similar to the proof of Theorem \ref{thm:DMgm-DM-full-faith} (2), Theorem  \ref{tA.4} and Example \ref{exA.5} imply that the canonical functor $\iota_\eff^V:\DM_{\gm,\Nis}^\eff\allowbreak\to \DM^\eff$ is fully faithful over any $k$ (compare \cite[p. 6801]{birat-tri}).
\item We shall see in Theorem \ref{t5.1} that  $\tau_\eff$ and $\tau_{\eff,\gm}$ are fully faithful. This seems to require the theory of intervals from Appendix \ref{section:interval}.
\end{enumerate}
\end{remarks}

\section{Brown-Gersten property}\label{s.bg}

\subsection{Main result} The adjunctions of Proposition \ref{p3.5} (1) induce adjunctions
\begin{align*}
D(\ul{a}_\Nis) = R\ul{a}_\Nis : D(\ulMPST) \leftrightarrows D(\ulMNST) : R\ul{i}_\Nis, \\
D(a_\Nis) = Ra_\Nis : D(\MPST) \leftrightarrows D(\MNST) : Ri_\Nis ,
\end{align*}
and $D(\ul{a}_\Nis)$, $D(a_\Nis)$ are localisations by Proposition \ref{pA.3}.

\begin{thm}\label{prop:DGguide}\ 

\begin{enumerate}
\item The kernel of the localisation functor \[D(\ul{a}_\Nis):D(\ulMPST) \to D(\ulMNST)\] equals $< \ulMV >$.
\item The kernel of the localisation functor \[D(a_\Nis):D(\MPST) \to D(\MNST)\]  equals $< \MV >$.
\item The localisation functor 
$D(\ulMPST) \to \ulMDM^{\eff}$
induces an equivalence of triangulated categories
\[
\frac{D(\ulMNST)}{< \ul{\mathrm{CI}}>} \cong \ulMDM^\eff.
\]
\item The localisation functor
$D(\MPST) \to \MDM^{\eff}$
induces an equivalence of triangulated categories
\[
\frac{D(\MNST)}{< \mathrm{CI} >} \cong \MDM^\eff.
\]
\item 
The categories $D(\ulMNST)$ and $D(\MNST)$ are compactly generated and inherit tensor structures from those of $D(\ulMPST)$ and $D(\MPST)$
obtained in Theorem \ref{thm:DMgm-DM-full-faith} (3). 
The functor $D(\tau_\Nis):D(\MNST)\to D(\ulMNST)$ is $\otimes$-triangulated.
\end{enumerate}
\end{thm}

The proof is given in Subsection \ref{s3.8}.

\subsection{Corollaries}
\begin{cor}[cf. Hypothesis \protect{\ref{h4.1}} (i)]\label{c4.1} Via $\ul{a}_\Nis$ and $a_\Nis$, the tensor structures on $\ulMPST$ and $\MPST$ induce right exact tensor structures on $\ulMNST$ and $\MNST$.
\end{cor}

\begin{proof}  Let $F,G\in \ulMNST$. We define
\[F\otimes_{\ulMNST} G = H_0(F[0]\otimes_{D(\ulMNST)}G[0]).\]
for the tensor structure of Theorem \ref{prop:DGguide} (5).  For its right exactness, one sees that $H_i(F[0]\otimes_{D(\ulMNST)}G[0])=0$ for $i<0$ by reducing to $F=\Z_\tr(M)$, $G=\Z_\tr(N)$. By Theorem \ref{prop:DGguide} (5), the functor $D(\ul{a}_\Nis)$ is monoidal, hence we have $\ul{a}_\Nis F_0\otimes_{\ulMNST}\ul{a}_\Nis G_0=\ul{a}_\Nis (F_0\otimes_{\ulMPST} G_0)$ for $F_0,G_0\in \ulMPST$. Same argument with $\MNST$.
\end{proof}

\begin{rk}\label{r4.3}
Proceeding as in \cite[Proof of Prop. 4.1.22]{ayoub}, 
it can be shown that $\otimes_{D(\ulMNST)}$ is actually the total left derived functor of $\otimes_{\ulMNST}$. Similarly for $\MNST$.
\end{rk}

\begin{cor}\label{c4.2}
The localisation functors 
$\ul{L}^\bcube : D(\ulMNST) \to \ulMDM^\eff$, $L^\bcube : D(\MNST) \to \MDM^\eff$ are strongly additive, symmetric monoidal and have right adjoints $\ul{j}^\bcube$, $j^\bcube$, which themselves have right adjoints. We have a natural isomorphism of functors
\begin{equation}\label{eq2.10.2}
\tau_\eff L^\bcube \simeq \ul{L}^\bcube D(\tau_\Nis).
\end{equation}
\end{cor}

\begin{proof} Via Theorem \ref{prop:DGguide}, strong additivity and symmetric monoidality follow from those of the localisation functors $D(\ulMPST) \to \ulMDM^\eff$ and $D(\MPST) \to \MDM^\eff$. The sequel then follows from Corollary \ref{cA.1}. The isomorphism \eqref{eq2.10.2} follows from the construction of $\tau_\eff$ in Theorem \ref{thm:DMgm-DM-full-faith} (4), and the fact that $D(\ul{a}_\Nis)$ and $D(a_\Nis)$ are localisations.
\end{proof}

For the next corollary, we recall an important notation from \cite{kmsy1}.

\begin{defn}\label{defn:smallpresheaf}
For $F \in \ulMPST$ and for $\sX \in \ulMSm$, we write $F_\sX$ for the presheaf on the small \'etale site of $\ol{\sX}$ given by
\[
F_\sX (\ol{\sU} \to \ol{\sX}) := F(\ol{\sU},\sX^\infty \times_{\ol{\sX}} \ol{\sU}).
\]
For $F \in \MPST$ and $\sX \in \MSm$, we set
\[
F_\sX := (\tau_! F)_\sX.
\]
We extend this notation to complexes of (pre)sheaves in the obvious way.
\end{defn}

\begin{cor} \label{c4.3}
For any $\sX\in \ulMCor$ and $K\in \ulMDM^\eff$,  we have an isomorphism
\[ 
\ulMDM^{\eff}(M(\sX), K)
\simeq \colim_{\sX' \in \ulSigma^\fin \downarrow \sX} \mathbb{H}^0_{\Nis}(\ol{\sX}', (\ul{j}^\bcube K)_{\sX'}).
\]
The same formula holds in  $\MDM^\eff$ if $\sX\in \MCor$ and $K\in \MDM^\eff$.
\end{cor}

\begin{proof} This is obvious by adjunction from Corollary \ref{c4.2} and \cite[Prop. 7.4.2 and Th. 7.5.1]{kmsy2}.
\end{proof}

Since $D(\ulMNST)$ and $\DR^\eff$ are compactly generated (Theorem
\ref{prop:DGguide}), Brown's representability theorem applied to their tensor structures provides them with internal Homs; similarly for $D(\MNST)$ and $\MDM^\eff$.   The following is an application of Lemma \ref{l1.5}:

\begin{prop} Let $K\in D(\ulMNST)$ and $L\in \DR^\eff$. Then we have a natural isomorphism
\[\ul{j}^\bcube \uHom_{\DR^\eff}(\ul{L}^\bcube(K),L)\simeq \uHom_{D(\ulMNST)}(K,\ul{j}^\bcube L)\]
hence, for $K',L\in \DR^\eff$, a natural isomorphism
\[\ul{j}^\bcube \uHom_{\DR^\eff}(K',L)\simeq \uHom_{D(\ulMNST)}(\ul{j}^\bcube K',\ul{j}^\bcube L).\]
Same formulas in  $D(\MNST)$ and $\MDM^\eff$, with $L^\bcube$ and $j^\bcube$. \qed
\end{prop}

\subsection{Sheaves on $\ulMSm^\fin$ and $\ulMSm$}
To prove Theorem \ref{prop:DGguide}, we have to go back to these categories. As in \cite{kmsy1}, we write $\ulMPS^\fin$ and $\ulMPS$ for the categories of presheaves of abelian groups on $\ulMSm^\fin$ and $\ulMSm$, and $\ulMNS^\fin$ and $\ulMNS$ for the corresponding categories of sheaves (for the $\ulMV^\fin$ and $\ulMV$ topology, respectively).  For general reasons, the inclusion functors $\ul{i}_{s,\Nis}^\fin:\ulMNS^\fin\inj \ulMPS^\fin$, $\ul{i}_{s,\Nis}:\ulMNS\inj \ulMPS$ have exact left adjoint sheafification functors $\ul{a}_{s,\Nis}^\fin$, $a_{s,\Nis}$. Moreover, the adjoint functors 
\[\ul{b}_{s,!}:\ulMPS^\fin\leftrightarrows\ulMPS:\ul{b}_s^* \]
associated to the functor $\ul{b}_s$ of Definition \ref{def:1.1.2} are both exact, and they preserve sheaves \cite[Lemma 4.2.3 and Prop. 4.3.3]{kmsy1}. For further reference, we record the corresponding tautological identity for the right adjoints:
\begin{equation}\label{eq:bnis-i}
\ul{b}_s^* \ul{i}_{s,\Nis} = \ul{i}_{s,\Nis}^\amb \ul{b}_{s}^\Nis.
\end{equation}

In the adjoint pair $(\ul{b}_{s,\Nis},\ul{b}_s^\Nis)$, $\ul{b}_{s,\Nis}$ is exact (but not $\ul{b}_s^\Nis$) .

\begin{defn}\label{d4.2}  Let $\Z^p:\ulMSm^\fin\to \ulMPS^\amb$, $\Z^p:\ulMSm\to \ulMPS$ denote the ``free presheaf'' functors (cf.  \cite[Proposition 2.6.1]{kmsy1}). That is, for $M, N \in \ulMSm^\amb$ (resp. $\ulMSm$), the section $\Z^p (M)(N)$ is given by the free abelian group $\Z \Hom_{\ulMSm^\amb} (N,M)$ (resp. $\Z \Hom_{\ulMSm} (N,M)$). We write $< \ulMV_s ^\amb>$ (resp. $<\ulMV_s>$) for the localizing subcategory of  $D(\ulMPS^\amb)$ (resp.  $D(\ulMPS)$) generated by the objects of the form
\[
\Z^p (U \times_M V) \to \Z^p (U) \oplus \Z^p (V) \to \Z^p (M)
\]
where $M \in \ulMSm^\amb$ (resp. $M\in \ulMSm$) and $U \sqcup V \to M$ runs over all elementary $\ulMV^\amb$-covers (resp. $\ulMV$-covers). Note that $<\ulMV_s>=D(\ul{b}_{s,!})(< \ulMV_s ^\amb>)$. 
\end{defn}

\subsection{Technical results}\label{sec-tec}

\begin{prop} \label{p3.4}\
\begin{enumerate}
\item For any $\ulMV$-square $T\in \ulMSm^\Sq$, the sequence
\[0\to \Z_\tr(T(00))\to \Z_\tr(T(01))\oplus \Z_\tr(T(10))\to \Z_\tr(T(11))\to 0\]
is exact in $\ulMNST$.
\item For any $\MV$-square $T\in \MSm^\Sq$, the sequence
\[0\to \Z_\tr(T(00))\to \Z_\tr(T(01))\oplus \Z_\tr(T(10))\to \Z_\tr(T(11))\to 0\]
is exact in $\MNST$.
\end{enumerate}
\end{prop}

\begin{proof} (1) is \cite[Th. 4.5.7]{kmsy1}. For (2), by \cite[Cor. 5.2.7]{mzki2} we have the desired exactness if we consider the terms as sheaves on $\ulMSm$; equivalently, the sequence becomes exact after applying $\tau_\Nis$. The conclusion then follows from Proposition \ref{p3.5} (2).
\end{proof}

\begin{prop}\label{l3.3} We have a naturally commutative diagram:
\begin{equation}\label{eq3.7}
\begin{gathered}
\xymatrix{
<\MV> \ar[r]\ar[d]&\sK \ar[r] \ar[d] & D(\MPST) \ar[r]^{D(a_\Nis)} \ar[d]_{D(\tau_!)} & D(\MNST) \ar[d]^{D(\tau_\Nis)} \\
< \ulMV > \ar[r] \ar[d] & \ul{\sK} \ar[r] \ar[d] & D(\ulMPST) \ar[r]^{D(\ul{a}_\Nis)} \ar[d]_{D(\ul{c}^\ast)} & D(\ulMNST) \ar[d]^{D(\ul{c}^\Nis)} \\
< \ulMV_s > \ar[r] & \ul{\sK}_s \ar[r] & D(\ulMPS) \ar[r]^{D(\ul{a}_{s,\Nis})} & D(\ulMNS) \\
< \ulMV_s^\amb> \ar[r] \ar[u] & \ul{\sK}_s^{\amb} \ar[r] \ar[u] & D(\ulMPS^\amb) \ar[r]^{D(\ul{a}^\amb_{s,\Nis})} \ar[u]^{D(\ul{b}_{s,!})} & D(\ulMNS^\amb) \ar[u]_{D(\ul{b}_{s,\Nis})},
}
\end{gathered}
\end{equation}
where $\sK$, $\ul{\sK}$, $\ul{\sK}_s$ and $\ul{\sK}_s^{\amb}$ are the kernels of $D(a_\Nis)$, $D(\ul{a}_\Nis)$, $D(\ul{a}_{s,\Nis})$ and $D(\ul{a}^\amb_{s,\Nis})$, respectively.
\end{prop}

\begin{proof} Note that $F\mapsto D(F)$ is functorial in exact functors $F$  by \cite[Lemma A.2.4]{kmsy2}. The commutativity of the upper (resp. middle, lower) right square in the diagram follows from Proposition \ref{p3.5} (2) (resp. from $\ul{a}_{s,\Nis} \ul{c}^* = \ul{c}^\Nis \ul{a}_\Nis$, which is proven in \cite[Prop. 4.5.6]{kmsy1}, resp. from the isomorphism $\ul{a}_{s,\Nis} \ul{b}_{s,!} \simeq \ul{b}_{s,\Nis} \ul{a}_{s,\Nis}^\amb$ 
which we obtain by taking left adjoints of both sides of \eqref{eq:bnis-i}).

A fortiori, this provides the vertical functors in the second column.

We have the inclusions $< \ulMV_s > \subset \ul{\sK}_s$ and $< \ulMV_s^\fin > \subset \ul{\sK}_s^\fin$ by \cite[Lemma 2.18]{cdstructures} and the regularity of the cd-structures $P_{\ulMV}$ and $P_{\ulMV^\fin}$. 
We also have $< \ulMV > \subset \ul{\sK}$ and $< \MV > \subset \sK$  by Proposition \ref{p3.4}. 

Finally, the arrows in the left column follow tautologically from the definitions of $< \ulMV >$, $< \ulMV_s >$ and $< \ulMV_s^\amb>$, except for the top one which follows from the proof of Theorem \ref{thm:DMgm-DM-full-faith} (4) (this arrow will not be used in the proof of Theorem \ref{prop:DGguide}).
\end{proof}

\begin{rk} It would be more natural to use $D(\ul{b}_s^*)$ and $R\ul{b}_s^\Nis$ in Diagram \eqref{eq3.7}. Unfortunately, the commutation of the corresponding square would imply the exactness of $\ul{b}_s^\Nis$, which is false. This will force us to use a more indirect argument for the proof of (ii) in Subsection \ref{s3.8} below.
\end{rk}

\begin{lemma}\label{l3.5} The functor $D(\ul{c}^\ast):D(\ulMPST)\to D(\ulMPS)$ is conservative, and $<\ulMV>\to <\ulMV_s>$ is essentially surjective.
\end{lemma}

\begin{proof} Let $C\in D(\ulMPST)$ be such that $D(\ul{c}^\ast)C=0$. For any $M\in \ulMSm$ and any $i\in\Z$, we have 
$c_! \Z^p(M) = \Z_\tr(M)$ by \cite[Proposition 2.6.1]{kmsy1}, where $\Z^p$ is as in Definition \ref{d4.2}. 
Moreover, the presheaf $\Z^p (M) \in \ulMPS$ is a projective object, since $\Hom_{\ulMPS} (\Z^p (M), F) = F(M)$ by definition of $\Z^p$. 
This implies $Lc_! \Z^p (M) = c_! \Z^p(M)$.
Therefore,
\[0=D(\ulMPS)(\Z^p(M),D(\ul{c}^\ast)C[i])=  D(\ulMPST)(\Z_\tr(M),C[i])\]
by adjunction.
This shows that $C=0$. 
The second statement is trivial.
\end{proof}

\begin{prop}\label{l3.4} We have the following isomorphism of functors:
\[ D(\ul{b}_{s,!})R\ul{i}_{s,\Nis}^\fin D(\ul{a}_{s,\Nis}^\fin)\simeq R\ul{i}_{s,\Nis} D(\ul{a}_{s,\Nis})D(\ul{b}_{s,!}). \]
\end{prop}

\begin{proof} In view of the commutativity of the right lower square in \eqref{eq3.7}, it suffices to show the isomorphisms 
\[
D(\ul{b}_{s,!})R(\ul{i}_{s,\Nis}^\amb) \osi R(\ul{b}_{s,!}\ul{i}_{s,\Nis}^\amb) \simeq R(\ul{i}_{s,\Nis}\ul{b}_{s,\Nis}) \iso R(\ul{i}_{s,\Nis})D(\ul{b}_{s,\Nis}).
\]
Indeed, the first isomorphism follows from the exactness of $\ul{b}_{s,!}$ by \cite[Lemma A.2.4]{kmsy2}, and the second one is tautological. 

For the third one, by \cite[Lemma A.2.7]{kmsy2}, it suffices to show that 
$\ul{b}_{s,\Nis}$ sends injectives to $\ul{i}_{s,\Nis}$-acyclic sheaves,
which  holds by \cite[Lemma 4.4.3]{kmsy1},
and that $D(\ul{b}_{s,\Nis})$, $R(\ul{i}_{s,\Nis})$ and $R(\ul{i}_{s,\Nis}\ul{b}_{s,\Nis})$ are strongly additive.  
Since $\ul{b}_{s,\Nis}$ is exact and strongly additive as a left adjoint, $D(\ul{b}_{s,\Nis})$ is strongly additive by \cite[Prop. A.2.8 a)]{kmsy2}.  
For $R(\ul{i}_{s,\Nis})$, we invoke \cite[Prop. A.2.8 c)]{kmsy2}: 
by \cite[Prop. A.2.5]{kmsy2},
$R \ul{i}_{s,\Nis}$ has the left adjoint $D(\ul{a}_{s,\Nis})$, 
which sends $\Z^{p}(M)[n]$ ($M\in \ulMSm$, $n\in \Z$) to  $\Z(M)[n]$.  The first are compact generators of $D(\ulMPS)$ by Example \ref{exA.5}, and the second are compact in $D(\ulMNS^\fin)$ by \cite[Prop. 7.1.1]{kmsy2}. 

Finally, we prove the strong additivity of $R(\ul{i}_{s,\Nis}\ul{b}_{s,\Nis})\simeq R(\ul{b}_{s,!}\ul{i}_{s,\Nis}^\amb)$. For this, we use \cite[Prop. A.2.8 b)]{kmsy2}. We must check that
\begin{itemize}
\item $R^p(\ul{b}_{s,!}\ul{i}_{s,\Nis}^\amb)$ is strongly additive for all $p\ge 0$;
\item there is a set $\sE$ of compact projective generators of $\ulMPS$ and integers $cd(E)$ for $E\in \sE$ such that $\ulMPS(E,R^p(\ul{b}_{s,!}\ul{i}_{s,\Nis}^\amb)(A))=0$ for any $p>cd(E)$ and for any $A\in \ulMNS^\fin$. 
\end{itemize}

Noting that $R^p(\ul{b}_{s,!}\ul{i}_{s,\Nis}^\amb)\simeq \ul{b}_{s,!}R^p\ul{i}_{s,\Nis}^\amb$, the first point follows from the commutation of Nisnevich cohomology with filtering direct limits. For the second one, we take for $\sE$ the collection of $\Z^p(M)$ for $M\in \ulMSm$, and claim that $cd(\Z^p(M))=\dim  M^\o$ works. Indeed:
\begin{multline*}
\ulMPS(\Z^p(M),R^p(\ul{b}_{s,!}\ul{i}_{s,\Nis}^\amb)(A)) = \ulMPS(\Z^p(M),\ul{b}_{s,!} R^p\ul{i}_{s,\Nis}^\amb(A))\\ 
=\colim_{N\in \ulSigma^\fin \downarrow M}\ulMPS^\fin(\Z^p(N),R^p\ul{i}_{s,\Nis}^\amb(A))=\colim_{N\in \ulSigma^\fin \downarrow M}H^p_\Nis(\ol{N},A_N)\\
=0 \text{ for } p>\dim M^\o,
\end{multline*}
where we used \cite[(2.5.1)]{kmsy1} for the second equality.
We are done.
\end{proof}

\subsection{Proof of Theorem \ref{prop:DGguide}}\label{s3.8}
Assertions (3) and (4) follow from (1) and (2).
Assertion (5) is a consequence of Theorem \ref{tA.4} and the fact that $D(\ulMPST)$ and $D(\MPST)$ are compactly generated (see Example \ref{exA.5}). 
The assertion on tensor structures holds since $< \ulMV >$ and $< \MV >$ are $\otimes$-ideals by Lemma \ref{lem:tensor-cd} (cf. the proof of Theorem \ref{thm:DMgm-DM-full-faith} (3)). 

It remains to prove (1) and (2). Let $<\MV >^\perp$ (resp. $< \ulMV >^\perp$) denote the right orthogonal of $< \MV >$ in $\sK$ (resp. of $< \ulMV >$ in $\ul{\sK}$), and define $<\ulMV_s>^\perp$, $<\ulMV_s^\fin>^\perp$ similarly. By  Theorem \ref{tA.5}, we have $<\MV>=\sK$ $\iff$ $<\MV>^\perp=0$, etc. We shall play with these equivalences. More precisely, the layout is:
\begin{thlist}
\item $<\ulMV_s^\fin>^\perp=0$.
\item $<\ulMV_s>=\ul{\sK}_s$.
\item $< \ulMV >^\perp=0$ (i.e. (1)).
\item $<\MV >^\perp=0$ (i.e. (2)).
\end{thlist}

\begin{proof}[Proof of (i)] It follows from \cite[Th. 3.2]{cdstructures},  since the cd-structure $\ulMV^\amb$ is complete and bounded (Proposition \ref{prop:MVamb}).
\end{proof}

\begin{proof}[Proof of (ii)] Let $x \in \ul{\sK}_s$.  Since $\ul{b}_{s,!}$ is exact, the functor $R\ul{b}_s^\ast$ is right adjoint to $D(\ul{b}_{s,!})$ by Proposition \ref{pA.3}. Consider the distinguished triangle in $D(\ulMPS^\amb)$:
\begin{equation}\label{eq:triangle}
R\ul{b}_s^\ast  x\by{f} R\ul{i}^\amb_{s,\Nis}D(\ul{a}^\amb_{s,\Nis})R\ul{b}_s^\ast  x \to z \by{+1},
\end{equation}
where $f$ is the unit of  the adjunction $(D(\ul{a}^\amb_{s,\Nis}),R(\ul{i}^\amb_{s,\Nis}))$ and $z$ is a cone of $f$. Applying $D(\ul{b}_{s,!})$ to \eqref{eq:triangle}, we obtain the following distinguished triangle in $D(\ulMPS)$:
\begin{equation}\label{eq:triangle2}
D(\ul{b}_{s,!})R\ul{b}_s^\ast  x\by{f} D(\ul{b}_{s,!})R\ul{i}^\amb_{s,\Nis}D(\ul{a}^\amb_{s,\Nis})R\ul{b}_s^\ast  x \to D(\ul{b}_{s,!})z \by{+1}.
\end{equation}

By Proposition \ref{l3.4}, the second term of \eqref{eq:triangle2} is isomorphic to $R\ul{i}_{s,\Nis} D(\ul{a}_{s,\Nis})D(\ul{b}_{s,!})R\ul{b}_s^\ast  x$, and we have $D(\ul{b}_{s,!}) R\ul{b}_{s}^\ast = \id$ by Proposition \ref{pA.3} and by the full faithfulness of $\ul{b}_{s}^\ast$ \cite[Prop. 2.5.1]{kmsy1}.
Hence the first term is isomorphic to $x$, and the second term is $0$ by $x \in \ul{\mathcal{K}}_s$. We thus get an isomorphism $x\simeq  D(\ul{b}_{s,!})z[-1]$. Moreover, $z\in \ul{\sK}_s^{\amb}$ as one sees by applying $D(\ul{a}^\amb_{s,\Nis})$ to \eqref{eq:triangle}. By (i) and Proposition \ref{l3.3}, this implies that $x\in <\ulMV_s>$ as requested. 
\end{proof}

\begin{proof}[Proof of (iii)] 
Let $x \in <\ulMV>^\perp$: we must prove that $x=0$. Since $D(\ul{c}^*)$ is conservative by Lemma \ref{l3.5}, it suffices to show $D(\ul{c}^*)x=0$. Since $<\ulMV_s>^\perp =0$ by (ii), it is enough to prove that $D(\ul{c}^*)x \in <\ulMV_s>^\perp$. By Definition \ref{d4.2}, $<\ulMV_s>$ is generated by complexes of the form $\Tot \Z^p (S)$ for $\ulMV$-squares $S$. Therefore, it suffices to prove that $\Hom_{D(\ulMPS)} (\Tot \Z^p (S), D(\ul{c}^*)x) = 0$ for any such $S$. We compute:
\begin{align*}
\Hom_{D(\ulMPS)} (\Tot \Z^p (S), D(\ul{c}^*)x) 
&\cong
\Hom_{D(\ulMPST)} (L(\ul{c}_!)\Tot \Z^p (S), x) \\
&= \Hom_{D(\ulMPST)} (\Tot \Z_\tr (S), x) \\
&=0,
\end{align*}
where the first isomorphism follows from Lemma \ref{lA.5}, since the left derived functor $L(\ul{c}_!)$ is defined at the bounded complex $\Tot \Z^p (S)$ of projective objects.\footnote{In fact, $Lc_!$ is everywhere defined by \cite[Th. 14.4.3]{ks}.}  The second equality follows from the equality $L(\ul{c}_!)\Z^p (M) = \Z_\tr (M)$ for any modulus pair $M$ (this was already used in the proof of Lemma \ref{l3.5}), and the third equality follows from $\Tot \Z_\tr (S) \in <\ulMV>$. This finishes the proof. 
\end{proof}

\begin{proof}[Proof of (iv)]
Since $D(\tau_!)$ is fully faithful by Lemma \ref{claim:ff}, we are reduced by (iii) to proving
\begin{equation}\label{eq:perp}
 D(\tau_!) (< \MV >^\perp ) \subset < \ulMV >^\perp.
\end{equation}

Take any $x \in < \MV >^\perp$. 
It suffices to prove that the abelian group 
\[
\ul{\sK} (\Tot \Z_\tr (S),D(\tau_!)(x)[i]) = D(\ulMPST)(\Tot \Z_\tr (S),D(\tau_!)(x)[i])
\]
is $0$ for any $\ulMV^\amb$-square $S$  and any $i\in\Z$, where $\Tot$ denotes totalisation. 
For each $(ij) \in \Sq$, set $S^N (ij) := (\ol{S(ij)}^N, \pi_{ij}^\ast S(ij)^\infty)$, where $\pi_{ij} : \ol{S(ij)}^N \to \ol{S(ij)}$ is normalisation.
Then the edges $S(ij) \to S(i'j')$ in $S$ uniquely lift to $S^N(ij) \to S^N(i'j')$, and we obtain a new square $S^N$. 
The maps $\pi_{ij}$ induce a morphism $S^N \to S$ in $\ulMSm^\Sq$, which is an isomorphism since normalisation is proper. 
Moreover, the \'etaleness of the edges in $S$ implies that $\ol{S^N (ij)} = \ol{S(ij)} \times_{\ol{S(11)}} \ol{S^N(11)}$, and therefore that $S^N$ is again an $\ulMV^\amb$-square. 
In the following, replacing $S$ with $S^N$, we may assume that the ambient space $\ol{S(ij)}$ is normal for all $(ij) \in \Sq$. 

Now, we compute for $i\in \Z$:
\begin{align*}
&D(\ulMPST) (\Tot \Z_\tr (S),D(\tau_!)(x)[i]) \\
&\cong^1 K(\ulMPST) (\Tot \Z_\tr (S),\tau_! (x)[i]) \\
&= H^0 \Hom^\bullet_{\mathrm{Ch}(\ulMPST)} (\Tot \Z_\tr (S),\tau_!(x)[i]) \\
&\cong^2 \varinjlim_{T \in \Comp (S)} H^0 \Hom^\bullet_{\mathrm{Ch}(\MPST)} (\Tot \Z_\tr (T),x[i]) \\
&=\varinjlim_{T \in \Comp (S)} K(\MPST) (\Tot \Z_\tr (T),x[i]) \\
&\cong^3 \varinjlim_{T \in \Comp (S)} D(\MPST) (\Tot \Z_\tr (T),x[i]) \\
&\cong^4 \varinjlim_{T \in \Comp^{\MV} (S)} D(\MPST) (\Tot \Z_\tr (T),x[i])\\ 
&=^5 0,
\end{align*}
where $\mathrm{Ch}(-)$ denotes the category of chain complexes, and $\Hom^\bullet$ denotes the Hom complex. 
Here $\Comp (S)$  and $\Comp^{\MV} (S)$ are as in Theorem \ref{t1.2}. 

The isomorphisms $\cong^1$ and $\cong^3$ hold because each component of $ \Z_\tr (S)$ and $\Z_\tr (T)$ is projective, and $\cong^2$ follows from the formula in Theorem \ref{t2.1} for the pro-left adjoint of $\tau^!$ of $\tau_!$.
Moreover, $\cong^4$ follows from Theorem \ref{t1.2}.
Finally, the assumptions $x \in < \MV >^\perp$ and $\Tot \Z_\tr (T) \in < \MV >$ imply $=^5$.
\end{proof}

This completes the proof of Theorem \ref{prop:DGguide}.

\begin{remark}\label{rem:bdd} This proof rests fundamentally on the fact that the cd-structure $P_{\ulMV^\fin}$ on $\ulMSm^\fin$ is bounded; an easier but similar proof shows that the kernel of the localisation functor $D(a_\Nis^V):D(\PST)\to D(\NST)$ equals $<\MV_\Nis>$, cf. \cite[Proposition in \S 4.2.1]{be-vo}. The main reason why the boundedness of $P_{\ulMV^\fin}$  is sufficient here seems to be that  sheaves in $\ulMNST$ and $\MNST$ also have cohomological dimensions bounded by the dimension of the total space of a modulus pair \cite[Cor. 2.2.10 and 5.1.4]{kmsy2}. 

We do not know whether the cd-structures $P_{\ulMV}$ and $P_{\MV}$ are themselves bounded. 
\end{remark}

\section{The derived Suslin complex}\label{s.sc}

\subsection{$\bcube$-invariance}\label{s9.1} We start with:

\begin{lemma}\label{l7.1} Let $\bcube=(\P^1, \infty) \in \MP$.  
The interval structure of $\A^1\simeq \P^1-\{ \infty \}\in \Sm$ from \cite{H1} 
induces an interval structure on $\bcube$ for the $\otimes$-structure of Definition \ref{def:tensor-mod}.
\end{lemma}

\begin{proof} 
We need to check that the structure maps $p,i_0,i_1,\mu$ are morphisms in $\MCor$. The unit object is $(\Spec k,\emptyset)$, so  $i_0,i_1$ and $p$ are clearly admissible. As for $\mu$, its points of indeterminacy in $\P^1\times \P^1$ are $(0,\infty)$ and $(\infty,0)$; the closure $\Gamma$ of its graph in $\P^1\times \P^1\times \P^1$ is isomorphic to $\Bl_{(0,\infty),(\infty,0)}(\P^1\times \P^1)$, where the two exceptional divisors are given by $0\times \infty\times \P^1$ and $\infty\times 0\times \P^1$. In particular, $\Gamma$ is smooth. Then
\[p_2^*\infty= \P^1\times \infty \times \infty + \infty \times \P^1\times \infty\]
while
\[p_1^*(\P^1\times \infty+\infty \times \P^1)=\P^1\times \infty\times \infty + 0\times \infty \times \P^1+ \infty \times \P^1\times \infty + \infty \times 0\times \P^1\]
which completes the proof.
\end{proof}

\begin{rk} \label{r5.2} Lemma \ref{l7.1} is false if we replace the $\otimes$-structure of Definition \ref{def:tensor-mod} by the cartesian product structure: indeed, $\mu$ does not factor through the morphism $\bcube\otimes \bcube\to\bcube\times \bcube$ of \eqref{eq1.1}. Conversely, the diagonal $\A^1\to \A^1\times \A^1$ obviously yields a diagonal morphism $\bcube\to \bcube\times \bcube$, but the latter does not factor through \eqref{eq1.1} either.
\end{rk}

The following definition will not be used in the sequel, except in Theorem \ref{t8.2}, but is key to \cite{modrec}.

\begin{definition}\label{def:cube-inv}
We say $F \in \ulMPST$
(resp. $F \in \MPST$) is \emph{$\bcube$-invariant}
if the projection map $p : M \otimes \bcube \to M$ 
induces an isomorphism
$p^* :F(M) \iso F(M \otimes \bcube)$
for any $M \in \ulMP$ (resp. $M \in \MP$). Equivalently, $F\iso \uHom(\Z_\tr(\bcube),F)$.
\end{definition}

\subsection{The derived Suslin complex}\label{s6.2} We shall need:

\begin{prop}\label{prop:tensor1}\ Consider the tensor structures on $D(\ulMNST)$ and $D(\MNST)$ given by Theorem \ref{prop:DGguide} (5). 
Then the interval structure on $\bcube \in \MP$ from
Lemma \ref{l7.1} yields categories with interval 
$(D(\ulMNST),\Z_\tr(\bcube))$,  $(D(\MNST),\Z_\tr(\bcube))$
which verify hypotheses \ref{hB.1} and \ref{h4.1}.
\end{prop}

\begin{proof} We do the proof for $D(\ulMNST)$, the case of $D(\MNST)$ being identical.

Since $\otimes_{D(\ulMPST)}$ is the total derived functor of $\otimes_{\ulMPST}$ 
by  Theorem \ref{t.mon} d),
there is a canonical natural transformation
\[\ul{\lambda}_P C\otimes_{D(\ulMPST)} \ul{\lambda}_P D\Rightarrow \ul{\lambda}_P (C\otimes_{K(\ulMPST)} D)\]
for $(C,D)\in K(\ulMPST)\times K(\ulMPST)$, where $\ul{\lambda}_P:K(\ulMPST)\to D(\ulMPST)$ is the localisation functor. Applying $D(\ul{a}_\Nis)$ to it, we get a natural transformation
\begin{multline*}
\ul{\lambda}_N K(\ul{a}_\Nis) C\otimes_{D(\ulMNST)} \ul{\lambda}_N K(\ul{a}_\Nis) D\\
\simeq D(\ul{a}_\Nis)\ul{\lambda}_P C\otimes_{D(\ulMNST)} D(\ul{a}_\Nis)\ul{\lambda}_P D\simeq D(\ul{a}_\Nis)(\ul{\lambda}_P C\otimes_{D(\ulMPST)} \ul{\lambda}_P D)\\
\Rightarrow D(\ul{a}_\Nis)\ul{\lambda}_P (C\otimes_{K(\ulMPST)} D)\simeq \ul{\lambda}_N K(\ul{a}_\Nis) (C\otimes_{K(\ulMPST)} D)\\
\simeq \ul{\lambda}_N (K(\ul{a}_\Nis) C\otimes_{K(\ulMNST)} K(\ul{a}_\Nis)D)
\end{multline*}
where $\ul{\lambda}_N:K(\ulMNST)\to D(\ulMNST)$ is the localisation functor. Since $K(\ul{a}_\Nis)$ is a localisation, this yields by \cite[Lemma A.3.3]{kmsy1} the desired natural transformation
\begin{equation}\label{eq5.2}
\ul{\lambda}_N C'\otimes_{D(\ulMNST)} \ul{\lambda}_N D'\Rightarrow \ul{\lambda}_N (C'\otimes_{K(\ulMNST)} D')\end{equation}
for $(C',D')\in K(\ulMNST)\times K(\ulMNST)$.

It remains to check properties (iii) and (iv) of Hypothesis \ref{h4.1}: (iii) is obvious by construction, and (iv) is true because it is already true in $D(\ulMPST)$ by the representability of $\Z_\tr(\bcube)$, and $\ul{a}_\Nis$ is exact.
\end{proof}

In the next theorem,
we use the functors $\ul{L}^\bcube, L^\bcube$ from Corollary \ref{c4.2}.

\begin{thm} \label{t5.1} 
The base change morphism 
\begin{equation}\label{eq:Ltau}
L^\bcube \circ D(\tau^\Nis) \Rightarrow \tau^\eff \ul{L}^\bcube
\end{equation}
as in \eqref{eq4.1} is an isomorphism; the functors $\tau_{\eff,\gm}$ and $\tau_\eff$ of Theorem \ref{thm:DMgm-DM-full-faith} (4) are fully faithful.
\end{thm}

\begin{proof}
The first claim follows from Theorem \ref{t7.1}.
By Lemma \ref{claim:ff}, $D(\tau_\Nis)$ is fully faithful, hence\footnote{This is actually part of the proof, see Proposition \ref{pA.3}.} the unit map
\[\id \Rightarrow D(\tau^\Nis) D(\tau_\Nis) \]
is an isomorphism.
Applying $L^\bcube$, we obtain a natural isomorphism 
\[L^\bcube\iso L^\bcube D(\tau^\Nis) D(\tau_\Nis).\]

On the other hand,  \eqref{eq2.10.2} and \eqref{eq:Ltau} yield natural isomorphisms
\[L^\bcube D(\tau^\Nis) D(\tau_\Nis)
 \iso \tau^\eff \ul{L}^\bcube D(\tau_\Nis)\iso \tau^\eff \tau_\eff L^\bcube\]
and one checks that their composition $ L^\bcube\Rightarrow \tau^\eff \tau_\eff L^\bcube$ is induced by the unit of the adjunction $(\tau_\eff,\tau^\eff)$. Since $L^\bcube$ is a localisation, we conclude that this unit is an isomorphism. This implies the full faithfulness of $\tau_\eff$, which in turn implies that of $\tau_{\eff,\gm}$ by Theorem \ref{thm:DMgm-DM-full-faith} (2).
\end{proof}

\begin{defn}[cf. Definition \protect{\ref{def:i-equiv}}]\label{d5.1} For any $K\in D(\ulMNST)$, we set
\[RC_*^\bcube (K)=\uHom(\Z_\tr(\bcube_\nu^\bullet),K)\in D(\ulMNST);\]
this is the \emph{derived Suslin complex of $K$}. Similarly for $K\in D(\MNST)$.  For $\sX\in \ulMCor$ or $\MCor$, we abbreviate $RC_*^\bcube (\Z_\tr(\sX)[0])$ to $RC_*^\bcube (\sX)$.
\end{defn}

Recall from Corollary \ref{c4.2} that
$\ul{L}^\bcube$ and $L^\bcube$ have right adjoints
 $\ul{j}^\bcube, j^\bcube $.
As a consequence of Theorem \ref{t22}, we have:

\begin{thm}\label{thm:jLC-special}
For any $K \in D(\ulMNST)$, we have an isomorphism
\[ \ul{j}^\bcube \ul{L}^\bcube(K) \simeq RC_*^\bcube (K).
\]
Similarly, we have an isomorphism for any $K \in D(\MNST)$
\[ j^\bcube L^\bcube(K) \simeq RC_*^\bcube (K).
\]
In particular, the isomorphisms of Corollary \ref{c4.3} translate as
\begin{align*} 
\ulMDM^{\eff}(M(\sX), \ul{L}^\bcube K)&\simeq \colim_{\sX' \in \ulSigma^\fin \downarrow \sX} \mathbb{H}^0_{\Nis}(\ol{\sX}', (RC_*^\bcube (K))_{\sX'})\\
\MDM^{\eff}(M(\sX), L^\bcube K)&\simeq \colim_{\sX' \in \ulSigma^\fin \downarrow \sX} \mathbb{H}^0_{\Nis}(\ol{\sX}', (RC_*^\bcube (K))_{\sX'})
\end{align*}
for $(\sX,K)\in \ulMCor\times D(\ulMNST)$ (resp. $(\sX,K)\in \MCor\times D(\MNST)$).
\end{thm}

\begin{remark}\label{r8.1} 
Theorem \ref{t22} also yields a version of Voevodsky's results for $\DM^\eff$ and $D(\NST)$ \cite{voetri,mvw}, where he uses simplicial objects rather than cubical objects. Comparing the two, we get an \emph{a posteriori} proof that for any $K\in D(\NST)$ the two ``Suslin'' complexes $RC_*^{\A^1}(K)$ based on simplicial or cubical sets are quasi-isomorphic. 
Hopefully this can be proven by an explicit chain computation.

On the other hand, the theory of intervals does not yield a simplicial theory in the case of $\ulMCor$ and $\MCor$, see Remark \ref{r4.1}. 
\end{remark}

\section{Comparisons}\label{s.comp}

\subsection{Relationship with Voevodsky's categories}\label{s6.1a} 

We start by comparing $\ulMDM^\eff$ with $\DM^\eff$. 
As usual, the functor $\ulomega:\ulMCor \to \Cor$
from \eqref{diag1} defines an adjunction 
\begin{equation}\label{eq5.1}
\ulomega_!:\ulMPST\leftrightarrows \PST:\ulomega^*.
\end{equation}

Thus $\ulomega_!$ is right exact and strongly additive; it is even exact, as the right adjoint of $\lambda_!$ (see Theorem \ref{t2.1} for $\lambda$). 

Since $\lambda$ is fully faithful, so is $\ulomega^*=\lambda_*$ \cite[Prop. 2.3.1]{kmsy1}.

Recall the category $\DM^\eff_{\gm,\Nis}$ from Subsection \ref{s6.1}. Since $\ulomega$ obviously sends $\mathrm{(\ul{CI})}$ to $\mathrm{(HI^V)}$ and $\mathrm{(\ul{MV})}$ to $\mathrm{(MV_\Nis^V)}$,  we get functors $\ulomega_{\eff,\gm}$, $\ulomega_\eff$  in the following commutative diagrams:
\begin{equation}\label{eq6.4}
\begin{CD}
K^b(\ulMCor) @>{K^b(\ulomega)}>> K^b(\Cor) \\
@VVV @VVV \\
\DR_{\gm}^{\eff} @>\ulomega_{\eff,\gm}>> \DM^{\eff}_{\gm,\Nis},
\end{CD}
\quad
\begin{CD}
D(\ulMPST) @>{D(\ulomega_!)}>> D(\PST) \\
@VVV @VVV \\
\DR^{\eff} @>\ulomega_{\eff}>> \DM^{\eff},
\end{CD}
\end{equation}
where the vertical arrows are localisation functors. 

As in the proof of Theorem \ref{thm:DMgm-DM-full-faith} (5), we deduce from the strong additivity of $\ulomega_!$ that $\ulomega_\eff$ is strongly additive and has a right adjoint $\ulomega^\eff$. Recall from \cite[Prop. 6.2.1 d)]{kmsy2} that  the adjunction \eqref{eq5.1} induces an adjunction 
\begin{equation}\label{eq5.3}
\ulomega_\Nis:\ulMNST\leftrightarrows \NST:\ulomega^\Nis
\end{equation}
where $\ulomega_\Nis, \ulomega^\Nis$ are both exact and $\ulomega^\Nis$ is fully faithful. The same picture holds for $\omega_\Nis = \ulomega_\Nis\circ \tau_\Nis$ and its right adjoint $\omega^\Nis$  \cite[Prop. 6.2.1 c)]{kmsy2}.

\begin{proposition}\label{prop:eq6.2}
The functors $\ulomega_{\eff}$ and $\ulomega^{\eff}$
fit in commutative diagrams
\begin{equation}\label{eq6.2}
\begin{CD}
D(\ulMNST) @>{D(\ulomega_{\Nis})}>> D(\NST)
@. \quad
D(\ulMNST) @<{D(\ulomega^{\Nis})}<< D(\NST)
\\
@V{\ul{L}^\bcube}VV @VV{L^{\A^1} }V 
@A{\ul{j}^\bcube }AA @AA{j^{\A^1} }A
\\
\DR^{\eff} @>>{\ulomega_{\eff}}> \DM^{\eff},
@.
\DR^{\eff} @<<{\ulomega^{\eff}}< \DM^{\eff}.
\end{CD}
\end{equation}
Moreover, $\ulomega^{\eff}$ is strongly additive and fully faithful, while 
$\ulomega_{\eff}$ is a localisation  and is symmetric monoidal.
\end{proposition}

\begin{proof} The second diagram of \eqref{eq6.4} factors through the first diagram of \eqref{eq6.2}, thanks to Theorem \ref{prop:DGguide} and its analogue for $\NST$ (Remark \ref{rem:bdd}). This yields the second diagram of \eqref{eq6.2} by adjunction. 
By Proposition \ref{pA.3}, the adjunction \eqref{eq5.3} implies that 
$D(\ulomega^{\Nis}) $ 
is right adjoint to $D(\ulomega_{\Nis})$, and fully faithful.
Hence
$\ulomega^{\eff}$ is 
fully faithful,
so that its  left adjoint
$\ulomega_{\eff}$ is a localisation. Since $\ulomega^\Nis$ is strongly additive \cite[Prop. 6.2.1 d)]{kmsy2}, so is $D(\ulomega^\Nis)$, hence $\ulomega^\eff$ by the diagram.

 The symmetric monoidality of $\ulomega_\eff$ will follow from that of the three other functors in the diagram. This is already known for the vertical ones (see Corollary \ref{c4.2} for $\ul{L}^\bcube$), so we are left to show the monoidality of $D(\ulomega_\Nis)$. By the same trick, the latter is reduced to the monoidality of $D(\ulomega_!)$, which in turn follows from that of $\ulomega$ (Proposition \ref{prop:monoidal-tauomega}).
\end{proof}

Using Corollary \ref{c4.2} and the exactness of $\tau^\Nis$ (Proposition \ref{p3.5} (2)), we now get commutative diagrams
\begin{equation}\label{eq6.3}
\begin{CD}
D(\MNST) @>{D(\omega_{\Nis})}>> D(\NST)
@. \quad
D(\MNST) @<{D(\omega^{\Nis})}<< D(\NST)
\\
@V{L^\bcube}VV @VV{L^{\A^1} }V 
@A{j^\bcube }AA @AA{j^{\A^1} }A
\\
\MDM^{\eff} @>>{\omega_{\eff}}> \DM^{\eff},
@.
\MDM^{\eff} @<<{\omega^{\eff}}< \DM^{\eff},
\end{CD}
\end{equation}
where $\omega^{\eff}$ is right adjoint to $\omega_\eff$ and $\omega_{\eff}$ is symmetric monoidal. Moreover,

\begin{proposition}\label{prop:eq6.2a}
The functor $\omega^{\eff}$ is strongly additive and fully faithful,
hence $\omega_{\eff}$ is a localisation.
\end{proposition}

\begin{proof} Same as for Proposition \ref{prop:eq6.2}, using the full faithfulness and strong additivity of $\omega^\Nis$ \cite[Prop. 6.2.1 c)]{kmsy2}.
\end{proof}

We finally have the following commutative diagram
\begin{equation}\label{eq6.1}
\begin{CD}
 \MDM^{\eff}_{\gm} @>\iota_\eff>> \MDM^{\eff} \\
@V{\tau_{\eff, \gm}}VV @VV{\tau_{\eff}}V \\
\DR^{\eff}_{\gm} @>\ul{\iota}_\eff>> \DR^{\eff} \\
@V{\ulomega_{\eff, \gm}}VV @VV{\ulomega_{\eff}}V \\
\DM^{\eff}_{\gm,\Nis} @>\iota_\eff^V>> \DM^{\eff}
\end{CD}
\end{equation}
in which $\iota_\eff$ and $\ul{\iota}_\eff$ are from \eqref{eq:DRgm-DR}  and $\iota_\eff^V$ is given in the same way 
(see Remark \ref{r3.1} (1)).  All rows are fully faithful by  Theorem \ref{thm:DMgm-DM-full-faith} (2) and \cite[(4.5)]{birat-tri}.

\subsection{Relationship with Chow motives}\label{s.chow} 

In \cite{voetri}, Voevodsky constructs a $\otimes$-functor $\Chow^\eff\allowbreak\to \DM_\gm^\eff$, where $\Chow^\eff$ is the category of effective (covariant) Chow motives. (We refer to \cite{scholl} or \cite{andre} for Chow motives.) This functor sends the Chow motive $h(X)$ of a smooth projective variety $X$ to $M^V(X)$, where $M^V:\Cor\to \DM_\gm^\eff$ is the canonical functor,  and is shown to be fully faithful when $k$ is perfect in \cite[6.7.3]{be-vo} (see also \cite[Th. 4.4.1 (3)]{birat-tri}). 

In fact, Voevodsky's  construction lifts to a $\otimes$-functor 
\begin{equation}\label{eq6.8}
\Phi_V^\eff:\Chow^\eff\allowbreak\to \DM_{\gm,\Nis}^\eff.
\end{equation}

 Indeed, this construction is as follows: let $\sH(\Cor)$ be the homotopy category of $\Cor$; its Hom groups are $h(X,Y) = \Coker( \Cor(X\times \A^1,Y)\to \Cor(X,Y))$. Obviously, the natural functor $\Cor\to \DM_{\gm,\Nis}^\eff$ factors through $\sH(\Cor)$. There is also a map
\begin{equation}\label{eq1hc}
h(X,Y) \to CH_{\dim X}(X\times Y)
\end{equation}
which sends a finite correspondence to the corresponding cycle class. This map is an isomorphism when $X$ and $Y$ are projective \cite[Th. 7.1]{fri-vo}, hence the functor \eqref{eq6.8}.

\begin{thm}\label{thm:chow}
Write $\omega_{\eff,\gm}=\ulomega_{\eff, gm}\circ \tau_{\eff, gm}$ (see \eqref{eq6.1}).  There is  a unique functor $\Phi^\eff:\Chow^\eff\to \MDM_\gm^\eff$ sending $h(X)$ to $M(X,\emptyset)$, whose composition with  $\omega_{\eff,\gm}$ is \eqref{eq6.8}. It is symmetric monoidal.
\end{thm}

(We shall see in Corollary \ref{c8.3} that $\Phi^\eff$ is fully faithful when $k$ is perfect.)

\begin{proof} 
For $X,Y$ as above, the inclusions
\begin{align*}
\MCor((X,\emptyset),(Y,\emptyset))&\subseteq \Cor(X,Y)\\ \MCor((X,\emptyset)\otimes \bcube,(Y,\emptyset))&\subseteq \Cor(X\times \A^1,Y)
\end{align*}
are equalities since the modulus conditions become empty (by definition of the left hand sides:  \ref{dp1}). Hence we get the refined functor from the definition of $\MDM^\eff$ in Definition \ref{def:MDMgm}.  Any other such functor agrees with $\Phi^\eff$ on objects, but also on morphisms by \eqref{eq1hc}, hence the uniqueness. Its symmetric monoidality is obvious.
\end{proof}

\subsection{Empty modulus}

\begin{thm}\label{t8.1} Let $X$ be a smooth proper $k$-variety. Then we have natural isomorphisms
\begin{align*}
\ul{M}(X,\emptyset)&\iso \ulomega^\eff M^V(X)\\
M(X,\emptyset)&\iso \omega^\eff M^V(X)
\end{align*}
where $\ulomega^\eff$ and $\omega^\eff$ are the functors from \eqref{eq6.2} and \eqref{eq6.3}.
\end{thm}

\begin{proof} For any $M\in \ulMCor$, the inclusion
\[ \ulMCor (M,(X,\emptyset))\subseteq \Cor (M^\o ,X)\]
is an equality by \ref{dp1}. 
Equivalently, $\ulomega^\Nis\Z_\tr^V(X)=\Z_\tr(X,\emptyset)$ and $\omega^\Nis\Z_\tr^V(X)=\Z_\tr(X,\emptyset)$. The result now follows from Theorem \ref{t7.1}, which yields natural isomorphisms
\begin{align*}
\ul{L}^\bcube \circ D(\ulomega^\Nis) &\iso  \ulomega^\eff \circ L^{\A^1}\\
L^\bcube \circ D(\omega^\Nis) &\iso  \omega^\eff\circ L^{\A^1}
\end{align*}
that we apply to $\Z_\tr^V(X)[0]$  (recall that $\ulomega^\Nis$ and $\omega^\Nis$ are exact, see \S \ref{s6.1a}.)
\end{proof}

\begin{defn}\label{d6.1}
We let $\DM^\eff_{\gm,\proper}$  (resp. $\DM^\eff_\proper$) be the thick (resp. localising) subcategory of $\DM^\eff_{\gm,\Nis}$  (resp. $\DM^\eff$) generated by the $M(X)$, where $X$ runs through the smooth proper $k$-varieties: it is closed under tensor product. 
\end{defn}

The following facts are well-known:

\begin{lemma}\label{l6.1} Suppose $k$ is perfect. Then $\DM^\eff_{\gm,\Nis}\iso \DM^\eff_\gm$. Under resolution of singularities, we have $\DM^\eff_{\gm,\proper}=\DM^\eff_\gm$. In general, we have $\DM^\eff_{\gm,\proper}[1/p]=\DM^\eff_\gm[1/p]$, where $p$ is the exponential characteristic of $k$.
\end{lemma}

\begin{proof} The first fact was recalled in Subsection \ref{s6.1} (cf. \cite[Th. 5.7]{voepre}). The second one follows from the Gysin distinguished triangles of \cite[Prop. 3.5.4]{voetri}. The last one is proven similarly, by resolution of singularities \`a la de Jong--Gabber plus a transfer argument which refines the one in  \cite[beg. of App. B]{hk}.
\end{proof}

\begin{cor}\label{c8.1} The restriction of $\ulomega^\eff$ to $\DM^\eff_\proper$ is symmetric mo\-no\-id\-al and induces a fully faithful symmetric monoidal functor
\[\ulomega^{\eff,\gm}:\DM_{\gm,\proper}^\eff\to \DR_\gm^\eff\]
which is ``right adjoint'' to the functor $\ulomega_{\eff, \gm}:\DR_\gm^\eff\to \DM_\gm^\eff$ of \eqref{eq6.1}: namely, this right adjoint is defined on $\DM_{\gm,\proper}^\eff$, with value $\ulomega^{\eff,\gm}$. The same holds when replacing $\DR_\gm^\eff$ by $\MDM_\gm^\eff$ and $\ulomega^\eff$ by $\omega^\eff$,
yielding $\omega^{\eff,\gm} : \DM_{\gm,\proper}^\eff\to \MDM_\gm^\eff$. In particular, if $k$ is perfect we have adjoint pairs
\begin{align}
\omega_{\eff, \gm}:\MDM_\gm^\eff[1/p]&\leftrightarrows \DM_\gm^\eff[1/p]:\omega^{\eff, \gm}\label{eq6.6}\\
\ulomega_{\eff, \gm}:\DR_\gm^\eff[1/p]&\leftrightarrows \DM_\gm^\eff[1/p]:\ulomega^{\eff, \gm}\label{eq6.7},
\end{align}
where $p$ is the exponential characteristic of $k$, and the same without inverting $p$ under Hironaka resolution of singularities (in particular, if $\car k=0$).
\end{cor}

\begin{proof} 
Everything
follows from Theorem \ref{t8.1} and Lemma \ref{l6.1}, 
except for the full faith\-ful\-ness of $\ulomega^{\eff,\gm}$, $\omega^{\eff,\gm}$ and the monoidality assertions. The first  follow from Propositions \ref{prop:eq6.2}, \ref{prop:eq6.2a} and the full embedding $\DM_{\gm,\Nis}^\eff\inj \DM^\eff$ of Remark \ref{r3.1} (1). Next, the monoidality of $\ulomega_\eff$ yields a natural transformation on $\DM^\eff$
\[\ulomega^\eff M\otimes \ulomega^\eff N \to \ulomega^\eff(M\otimes N)\]
and similarly for $\omega^\eff$. By Theorem \ref{t8.1}, this is an isomorphism when $M$ and $N$ are of the form $M(X)$ and $M(Y)$ for $X,Y$ smooth proper, hence in general by the strong additivity of $\omega^\eff$ and $\ulomega^\eff$ (Propositions \ref{prop:eq6.2} and \ref{prop:eq6.2a} again). 
\end{proof}

\begin{defn}\label{d6.2} Let $\Z(1):=\Phi^\eff(\L)[-2]$, where $\L\in \Chow^\eff$ is the Lefschetz motive.  For $i \geq 0$ and $M \in \DR^{\eff}_{\gm}$, we put $\Z(i)=\Z(1)^{\otimes i}$  and $M(i)=M \otimes \Z(i)$.
\end{defn}

\begin{cor}\label{c8.2} 
Assume $k$ is perfect.
Let $X$ be a smooth proper $k$-variety of dimension $d$,
$\sY\in \ulMCor$ a modulus pair,
and $i, j$ integers with $i \geq 0$.
Then we have a canonical isomorphism:
\[\Hom_{\DR_\gm^\eff}(\ul{M}_\gm(\sY),\ul{M}_\gm(X,\emptyset)(i)[j])\simeq H^{2d+j}(\sY^\o\times X,\Z(d+i))\]
where the right hand side is Voevodsky's motivic cohomology. 
In particular, this group is isomorphic to the higher Chow group $CH^{d+i}(\sY^\o\times X, 2i-j)$
and vanishes if $j>2i$, by \cite[Cor. 2]{allagree}.\\
The same formula holds in $\MDM_\gm^\eff$ if $\sY\in \MCor$ (with $M_\gm$ instead of $\ul{M}_\gm$).
\end{cor}

\begin{proof} 
By Theorem  \ref{thm:DMgm-DM-full-faith} (2), we may compute the Hom on the left hand side using $\DR^\eff$ instead of $\DR_\gm^\eff$. 
By monoidality (Corollary \ref{c8.1}), 
$\ulomega^{\eff}$ sends
$\Z(i) \in \DM^{\eff}_{\gm}$ to
$\Z(i) \in \DR^{\eff}_{\gm}$.
By adjunction and Theorem \ref{t8.1}, 
we then have an isomorphism
\[\Hom_{\DR^\eff}(\ul{M}(\sY),\ul{M}(X,\emptyset)(i)[j])
\simeq \Hom_{\DM^\eff}(M^V(\sY^\o),M^V(X)(i)[j]). 
\]
The result now follows from Poincar\'e duality for $X$ \cite[Prop. 6.7.1]{be-vo}. The case of $\MDM_\gm^\eff$ is identical.
\end{proof}

\begin{cor}\label{c8.3} 
Assume $k$ is perfect.
The functor $\Phi^\eff:\Chow^\eff\to \MDM_\gm^\eff$ from Theorem \ref{thm:chow} is fully faithful.  \qed
\end{cor}

\subsection{$\pi_0$-invariance} For any modulus pair $\sY\in \ulMCor$, we write $\pi_0(\sY)\allowbreak:=(\pi_0(\sY^\o),\emptyset)$, where $\pi_0(\sY^\o)$ (``scheme of constants'') is the universal \'etale $k$-scheme such that the projection $\sY^\o\to \Spec k$ factors through $\pi_0(\sY^\o)$. This factorisation induces a morphism $p_\sY:\sY\to \pi_0(\sY)$. 
In contrast to Theorem \ref{t8.1}, we have

\begin{thm}\label{t8.2} Let $X$ be smooth and quasi-affine. Then $\Z_\tr(X,\emptyset)$ is $\bcube$-invariant (Definition   \ref{def:cube-inv}) and, more strongly, ``properly $\pi_0$-invariant'': for any 
proper modulus pair $\sY\in \MCor$, we have an isomorphism in $\ulMNST$
\[\uHom(\Z_\tr(\pi_0(\sY)),\Z_\tr(X,\emptyset))\iso \uHom(\Z_\tr(\sY),\Z_\tr(X,\emptyset))\]
induced by $p_\sY$.
\end{thm}

\begin{proof} We may reduce to $\sY^\o$ connected and (up to extending $k$) even geometrically connected. 
Take $\sZ \in \ulMCor$.
It suffices to show that the map
\begin{equation}\label{eq:blue}
p_\sY^* : \ulMCor(\sZ, (X, \emptyset))
\to \ulMCor(\sZ \otimes \sY, (X, \emptyset))
\end{equation}
induced by $p_\sY$ is an isomorphism.
For any closed point $y \in \sY^\o$,
we find that 
$\ulMCor(\sZ, (X, \emptyset))
\to 
\ulMCor(\sZ \otimes(y, \emptyset), (X, \emptyset))$
is injective,
hence \eqref{eq:blue} is injective as well.

To show its surjectivity,
let us take an elementary modulus correspondence 
$V \in \ulMCor(\sZ \otimes \sY, (X, \emptyset))$.
Let $\ol{V}$ be the closure of $V$
in $\ol{\sZ} \times \ol{\sY} \times X$.
We claim that the image $\ol{V'}$
of $\ol{V}$ in $\ol{\sZ} \times X$
is closed and finite surjective over $\ol{\sZ}$.
To prove this claim, 
consider the  commutative diagram
\[
\xymatrix{
\ol{V}
\ar@{^{(}->}[r]^-{i}
\ar[d]_{\pi'}
\ar[rd]^{\pi}
&
\ol{\sZ} \times \ol{\sY} \times X
\ar[r]^-{a}
\ar[d]^{b}
&
\ol{\sZ} \times \ol{\sY}
\ar[d]^{c}
\\
\ol{V'}
\ar@{^{(}->}[r]_-{i'}
&
\ol{\sZ} \times X
\ar[r]_-{d}
&
\ol{\sZ}.
}
\]

Since
$V \in \ulMCor(\sZ \otimes \sY, (X, \emptyset))$,
$ai$ is proper and surjective.
Since the same is true of $c$,
we find that $cai=d\pi$ is proper surjective.
This implies that $\ol{V'}$ is closed
and, combined with the surjectivity of $\pi'$,
that $di'$ is proper \cite[Cor. 5.4.3]{EGA2}.
But $di'$ is also quasi-affine
(since so is $d$), hence finite.
This proves the claim.

Now $V' := \ol{V'} \cap (\sZ^\o \times X)$
is an element of $\ulMCor(\sZ, (X, \emptyset))$.
We clearly have $V \subset V' \times \sY^\o$,
and 
$V' \times \sY^\o$ is irreducible
because $\ol{\sY}$ is geometrically irreducible.
By comparing dimensions,
we get $V=V' \times \sY^\o=p_Y^*(V')$.
This proves the surjectivity of \eqref{eq:blue}.
\end{proof}

\subsection{Inverting the Tate object}\label{s.tate} 
In this subsection, we shall abundantly use the multiplicative localisations introduced by Grothendieck for pure motives (inverting the Lefschetz motive); one may refer to \cite[A.2.4, A.2.5]{kahnzeta} for a detailed discussion, see also \S \ref{susp}. 

\begin{defn}\label{d6.3} We write $\MDM_\gm$ (resp. $\DR_\gm$) for the 
category obtained from $\MDM_\gm^\eff$ (resp. from $\DR_\gm^\eff$) by $\otimes$-inverting $\Z(1)$ (resp. $\tau_\eff \Z(1)$), see Definition \ref{d6.2}. Similarly for $\DM_\gm$ and $\DM_{\gm,\proper}$ from $\DM_\gm^\eff$ and $\DM_{\gm,\proper}^\eff$. 
\end{defn}

The $\otimes$-functor $\Phi^\eff$ of Theorem \ref{thm:chow} extends canonically to a $\otimes$-functor
\begin{equation}\label{eq6.5}
\Phi:\Chow\to \MDM_\gm
\end{equation}
where $\Chow$ is the category of (all) Chow motives. 

\begin{prop}\
\begin{enumerate}
\item The categories $\MDM_\gm$ and $\ulMDM_\gm$ are Karoubian $\otimes$-triang\-ul\-at\-ed categories.
\item The functor $\tau_\gm:\MDM_\gm\to \ulMDM_\gm$ induced by $\tau_{\eff,\gm}$ is $\otimes$-triangulated and fully faithful.
\item The functor $\Phi$ is symmetric monoidal, and fully faithful if $k$ is perfect. For any smooth projective variety $X$, the motive $M(X,\emptyset)$ is strongly dualisable in $\MDM_\gm$ and in $\ulMDM_\gm$.
\end{enumerate}
\end{prop}

\begin{proof} For (1), we use Voevodsky's sufficient condition \cite[Prop. A.31]{kahnzeta} that the switch endomorphism of $\Z(1)^{\otimes 2}$ is the identity in $\MDM_\gm^\eff$ and $\ulMDM_\gm^\eff$, which holds because this is true for the Lefschetz motive $\L$ in the category $\Chow^\eff$. The karoubian assertion follows from Lemma \ref{lA.1} (1). For (2), we apply Lemma \ref{lA.1} (2) together with Theorem \ref{t5.1}. Similarly for (3), with Corollary \ref{c8.3};  the strong dualisability statement holds because $\Chow$ is rigid.
\end{proof}

\begin{prop} \label{p6.1} The $\otimes$-functors $\ulomega_{\eff,\gm}$, $\omega_{\eff,\gm}$, $\ulomega^{\eff,\gm}$ and $\omega^{\eff,\gm}$ of Corollary \ref{c8.1} induce $\otimes$-functors $\ulMDM_\gm\by{\ulomega_\gm} \DM_\gm$, $\MDM_\gm\allowbreak\by{\omega_\gm} \DM_\gm$, $\DM_{\gm,\proper}\by{\ulomega^\gm} \ulMDM_\gm$ and $\DM_{\gm,\proper}\by{\omega^\gm} \MDM_\gm$. The functors $\ulomega^\gm$ and $\omega^\gm$ are fully faithful; when $k$ is perfect,  
the adjoint pairs \eqref{eq6.6} and \eqref{eq6.7} of Corollary \ref{c8.1} induce adjoint pairs
\begin{align}
\omega_\gm:\MDM_\gm[1/p]&\leftrightarrows \DM_\gm:\omega^\gm[1/p]\label{eq6.6a}\\
\ulomega_\gm:\DR_\gm[1/p]&\leftrightarrows \DM_\gm:\ulomega^\gm[1/p].\label{eq6.7a}
\end{align}
Under resolution of singularities in characteristic $p$, we can drop the affixes $[1/p]$.
\end{prop}

\begin{proof} The functors of Corollary \ref{c8.1} induce the said functors because of their monoidality, which also implies the monoidality of these functors. 
The full faithfulness of $\ulomega^\gm$ and $\omega^\gm$ is shown as in the previous proof.
The adjunction identities of \eqref{eq6.6} and \eqref{eq6.7}  are preserved by $\otimes$-inverting $\Z(1)$, which yields corresponding adjunction identities for \eqref{eq6.6a} and \eqref{eq6.7a}.  
\end{proof}

As is well-known, the $\otimes$-category $\DM_\gm[1/p]$ (more generally, the category $\DM_{\gm,\proper}$) is rigid. By contrast, there is evidence that this does not hold for $\MDM_\gm$. Unfortunately, we have to make two assumptions: one is the analogue of Voevodsky's cancellation theorem \cite{voecan} and the other is that ``the derived Suslin complex is quasi-isomorphic to the na\"\i ve Suslin complex''.

\begin{prop}\label{p6.2} If the functor $\MDM_\gm^\eff\to \MDM_\gm$ is fully faithful and if the base change morphism \eqref{eq4.1a} is an isomorphism as in the end of Example \ref{exB.1}, the $\otimes$-category $\MDM_\gm$ is not rigid. More precisely, $M(\bcube^{(2)})$ is not dualisable, with $\bcube^{(2)}=(\P^1,2\infty)$.
\end{prop}

\begin{proof}  Let $M\in \MDM_\gm$, having a dual $M^*$. Suppose that $\omega_\gm(M)=0$. By the monoidality of $\omega_\gm$, we also have $\omega_\gm(M^*)\simeq \omega_\gm(M)^*=0$. Equivalently,
\[\MDM_\gm(M^*,\omega^\gm N)=0\quad \forall N\in \DM_\gm.\]

 Suppose that $N$ also has a dual $N^*$. Applying the above to $N^*$ instead of $N$, and using this time the monoidality of $\omega^\gm$, we find
\[0=\MDM_\gm(M^*,(\omega^\gm N)^*) = \MDM_\gm(\omega^\gm N,M).\]

Take in particular $N=\Z$; by the assumption of full faithfulness and by Theorems \ref{thm:jLC-special} and \ref{t8.1}, we get
\[\mathbb{H}^0_{\Nis}(k, RC_*^\bcube (M)_k)=0.\]

Take for example $M=\text{fibre}(M(\bcube^{(2)})\to \Z)$: clearly, $\omega_\gm(M)=0$. Under the assumption on the base change morphism \eqref{eq4.1a}, we can replace $RC_*^\bcube (M)$ by the na\"\i ve Suslin complex $C_*^\bcube (M)$ used in \cite{ry}. Applying its Theorem 1.1 with $S = \Spec k$, $\sC = \P^1$, $D = 2\infty$, we find
\begin{multline*}
\mathbb{H}^0_{\Nis}(k, RC_*^\bcube (\bcube^{(2)})_k)=\mathbb{H}^0_{\Nis}(k, C_*^\bcube (\bcube^{(2)})_k)\\
=H_0^S(\P^1/k,2\infty)=\Pic(\P^1, 2\infty)=\Z\oplus k
\end{multline*}
where the last term is the relative Picard group. Thus we get an isomorphism $\mathbb{H}^0_{\Nis}(k, RC_*^\bcube (M)_k)\simeq k$, a contradiction.  (Note that the morphism $\bcube^{(2)}\to\un$ is split by the $0$-section $\un\to \bcube^{(2)}$.)
\end{proof}

\enlargethispage*{20pt}

\section{Some computations} \label{s7}

For simplicity, we write $M$ and $\ul{M}$ for $M_\gm$ and $\ul{M}_\gm$ in this section.

\subsection{The tautological isomorphisms and distinguished triangles} These are those which come from the definitions of $\MDM_\gm^\eff$ and $\ulMDM^\eff$:

\begin{description}
\item[Mayer-Vietoris] one has a distinguished triangle in $\MDM_\gm^\eff$ (resp. $\ulMDM_\gm^\eff$):
\[M(T(00))\to M(T(01))\oplus M(T(10))\to M(T(11))\by{+1}\]
for any $\MV$-square (resp. $\ulMV$-square) $T$.
\item[Tensor product] one has canonical isomorphisms $M(\sX\otimes \sY)\simeq M(\sX)\otimes M(\sY)$ for any $\sX,\sY$ in $\MCor$ (resp. $\ulMCor$).
\item[$\bcube$-invariance] the morphism $M(\bcube)\by{M(p)} \Z=:M(\un)$ is invertible,  where $p:\bcube\to \un$ is the structural map.
\end{description}

\subsection{An elementary computation} As a special case, in the situation of Example \ref{kp3}, one has a distinguished triangle in $\MDM_\gm^\eff$
\begin{equation}\label{eq7.1}
M(X,D')\to M(X,D_1)\oplus M(X,D_2)\to M(X,D)\by{+1}
\end{equation}

Let us use this example to reduce the computation of the motive of $(\P^1,D)$ to its essential parts, where $D$ is an effective divisor. Generally, for a modulus pair $\sX\in \MCor$, let us write $\tilde \Z_\tr(\sX)=\Ker(\Z_\tr(\sX)\to \Z_\tr(\un)=\Z)$: in the presence of a $0$-cycle of degree $1$ on $\sX^\o$, this is a direct summand of $\Z_\tr(\sX)$. We define $\tilde M(\sX)$ as the class of $\tilde \Z_\tr(\sX)$ in $\MDM_\gm^\eff$, so that we have a distinguished triangle
\[\tilde M(\sX)\to M(\sX)\to \Z\by{+1}\]
split in the presence of a $0$-cycle of degree $1$.

If $\sX=(\P^1,D)$, write $\tilde M(\sX)=:m(D)$ for simplicity. Then $m(\emptyset)=\Z(1)[2]$ and $m(\infty)=0$. Let $D_1,D_2$ have disjoint supports. 
Choose $0$-cycle of degree $1$ on $\sX^\o$.
(We can take a rational point unless $k$ is finite.)
It splits off a distinguished triangle
\[m(D_1+D_2)\to m(D_1)\oplus m(D_2)\to \Z(1)[2]\by{+1}\]
from the distinguished triangle \eqref{eq7.1} with $D=\emptyset$. 
But the morphism $m(D_i)\to m(\emptyset)$ is $0$ if $k$ is perfect by Corollary \ref{c8.2} (if $D_i$ contains a rational point $p$, an elementary proof is that it factors through $m(p)\simeq m(\infty)=0$). Thus this triangle splits and yields a non-canonical isomorphism
\[m(D_1+D_2)\simeq m(D_1)\oplus m(D_2)\oplus \Z(1)[1].\]

Suppose $k$ algebraically closed, for simplicity. If $D=\sum_{i=1}^r n_ip_i$ with the $p_i$ distinct points, we get inductively an isomorphism
\[m(D)\simeq \bigoplus_{i=1}^r m(n_i\infty) \oplus (r-1)\Z(1)[1].\]

\subsection{Motives of vector bundles and projective bundles} Let $\sY\in \ulMCor$ be a modulus pair, and let $E$ be a vector bundle of rank $n>0$ on $\ol{\sY}$, with associated projective bundle $\P(E)$. We define modulus pairs $\sE$ and $\sP$ with total spaces $E$ and $\P(E)$ by pulling back $\sY^\infty$: the resulting morphisms $\sE\to \sY$, $\sP\to \sY$ are minimal in the sense of Definition \ref{deff}.

(There may be more general notions of vector and projective bundles, but we do not consider them here.)

\begin{remark}\label{c8.4} 
By applying Corollary \ref{c8.2} 
with $X=\Spec(k)$ and $j=2i$,
we get 
$CH^i(\sY^\o) \simeq 
\Hom_{\DR_\gm^\eff}(\ul{M}(\sY), \Z(i)[2i])$.
In particular, 
if
$P(t_1, \dots, t_n) \in \Z[t_1, \dots, t_n]$ 
is a homogeneous polynomial 
of weight $i$ (the weight of $t_s$ being $s$),
then the Chern classes of $E$ yield
a morphism in $\DR_\gm^\eff$
\[P(c_1(E),\dots, c_n(E)) : \ul{M}(\sY)\to \Z(i)[2i]. \]
\end{remark}

\begin{thm}\label{t7.2} 
Assume $k$ is perfect. 
Suppose $\ol{\sY}$ smooth. The projection $\bar p:\sP\to \sY$ yields a canonical isomorphism in $\DR_\gm^\eff$
\begin{equation}\label{eq7.3}
\rho_\sY:\ul{M}(\sP)\iso \bigoplus_{i=0}^{n-1} \ul{M}(\sY)(i)[2i].
\end{equation}
The same holds in $\MDM_\gm^\eff$ if $\sY\in \MCor$ (with $M$ instead of $\ul{M}$).
\end{thm}

\begin{rk} If $\car k=0$ or $\dim \sY\le 3$, the assumption on $\ol{\sY}$ is innocent in view of resolution of singularities.
\end{rk}

\begin{proof} We follow the method of Voevodsky in \cite[proof of Prop. 3.5.1]{voetri}, with a simplification and a complication. The complication is that Voevodsky's construction of the corresponding morphism in $\DM_\gm^\eff$ uses diagonal maps, which cause a problem here (see Remark \ref{w2.1}). We bypass this problem by using the morphism
\begin{equation}\label{eq7.4}
\tilde \Delta: \sP\to \sP \otimes (\P(E),\emptyset) 
\end{equation}
 induced by the diagonal inclusion 
 \begin{equation}\label{eq7.5}
\sP^\o\inj \sP^\o \times  \P(E).
 \end{equation}

Here, the modulus condition is obviously verified. Using the morphisms $\ul{M}(\P(E),\emptyset)\to \Z(i)[2i]$ induced by the powers of 
$c_1(O_{\P(E)}(1))$ (see Remark \ref{c8.4}), 
we get morphisms
 \begin{equation}\label{l7.6}
 \rho_\sY^i:\ul{M}(\sP)\by{\ul{M}(\tilde \Delta)}  
 \ul{M}(\sP)\otimes \ul{M}(\P(E),\emptyset)\to \ul{M}(\sY)\otimes \Z(i)[2i],
\end{equation} 
whence $\rho_\sY$. To prove that it is an isomorphism, we first consider the case where the vector bundle $E$ is trivial. We then have an isomorphism of modulus pairs
\[\sP\simeq \sY\otimes (\P^{n-1},\emptyset)\]
hence a corresponding isomorphism of motives. By using either Theorem \ref{t8.1} or, more directly, the functor $\Phi^\eff$ of Theorem \ref{thm:chow} and the computation of the Chow motive of $\P^{n-1}$, one has a canonical isomorphism
\[\theta:\bigoplus_{i=0}^{n-1} \Z(i)[2i]\iso \ul{M}(\P^{n-1},\emptyset).\]

Tensoring it with $\ul{M}(\sY)$ and composing with $\rho_\sY$, we get a morphism
\[\bigoplus_{i=0}^{n-1} \ul{M}(\sY)(i)[2i]\to \bigoplus_{i=0}^{n-1} \ul{M}(\sY)(i)[2i]\]
which is seen to be the identity by definition of $\theta$ and $\rho_\sY$.

In general, we argue by induction on the number $m$ of terms of an open cover of $\ol{\sY}$ trivialising $E$. For notational simplicity, write $\Xi(\sY)$ for the right hand side of \eqref{eq7.3}. Write $\ol{\sY}=\ol{\sY'} \cup \ol{U}$, where $E$ is trivial over $\ol{U}$ and $\ol{\sY'}$ has an $(m-1)$-fold trivialising open cover. Provide $\ol{\sY'}$, $\ol{U}$ and $\ol{\sY'}\cap \ol{U}$ with the induced modulus structures $\sY'$, $U$, $\sY'\cap U$, and pull $\sP$ back similarly. We claim that the diagram of distinguished triangles (with obvious notation)
\begin{equation}\label{eq:deglise}\Small
\begin{gathered}
\xymatrix{
\ul{M}(\sP_{|\sY'\cap U})\ar[r]\ar[d]& \ul{M}(\sP_{|\sY'})\oplus \ul{M}(\sP_{|U})\ar[r]\ar[d]& \ul{M}(\sP) \ar[r]\ar[d] &\ul{M}(\sP_{|\sY'\cap U})[1]\ar[d]\\
\Xi(\sY'\cap U)\ar[r]& \Xi(\sY') \oplus \Xi(U)\ar[r]&  \Xi(\sY)\ar[r]&  \Xi(\sY'\cap U)[1]
}
\end{gathered}
\end{equation}
commutes, which will conclude the proof. The commutations of the left and middle square follow from the naturality of $\rho$. For the right one\footnote{We thank one of the referees for stressing this issue.}, consider the morphisms $f:\sY'\cap U\to \sY'\oplus U$ and $f_\sP:\sP_{|\sY'\cap U}\to \sP_{|\sY'}\oplus \sP_{|U}$ and their associated motives $\ul{M}[f]$, $\ul{M}[f_\sP]$ (see Definition \ref{def:MDMgm}). We also have an obviously defined motive  $\Xi[f]$, which is a canonical cone of the left bottom map. Observe now that \eqref{eq7.5} induces morphisms
\[\sP_{|\sY'}^\o\inj \sP_{|\sY'}^\o \times  \P(E), \sP_{|U}^\o\inj \sP_{|U}^\o \times  \P(E), \sP_{|\sY'\cap U}^\o\inj \sP_{|\sY'\cap U}^\o \times  \P(E)  \]
which in turn induce morphisms analogous to \eqref{eq7.4}, a morphism in $K^b(\ulMCor)$ 
\[[f_\sP]\to [f]\otimes [(\P(E),\emptyset)]\]
compatible with \eqref{eq7.4},  and finally a morphism $\rho_f:\ul{M}[f_\sP]\to \Xi[f]$ analogous to $\rho_\sY$ (see \eqref{l7.6}).
The Mayer-Vietoris property says that there are 
canonical horizontal isomorphisms 
in the 
diagram:
\[
\xymatrix{
\ul{M}[f_\sP] \ar[r]^{\sim} \ar[d]^{\rho_f} &
\ul{M}(\sP) \ar[d]^{\rho_\sY} 
\\
\Xi[f] \ar[r]^{\sim} &
\Xi(\sY).
}
\]

Since the Chern class $c_1(O_{\P(E)}(1))$ 
restricts to those of 
$c_1(O_{\P(E|_V)}(1))$ for $V= \ol{\sY}', U, \ol{\sY}' \cap U$,
this diagram commutes.
Therefore we may replace 
$\ul{M}(\sP)$ and $\Xi(\sY)$
by 
$\ul{M}[f_{\sP}]$ and $\Xi(f)$
in \eqref{eq:deglise}. 
But then the commutation is obvious.
\end{proof}

\begin{qn} When $E$ is trivial, the isomorphism $\sE\simeq \sY\otimes (\A^n,\emptyset)$ yields an isomorphism $\ul{M}(\sE)\iso \ul{M}(\sY)\otimes \ul{M}(\A^n,\emptyset)$. Can one extend this isomorphism to the general case?
\end{qn}

\subsection{Further results} In this subsection, we present results which were obtained (in anticipation to the release of this paper!) in different works.

\begin{prop}[Toric invariance \protect{\cite[Lemma 10]{KeSa}}]
For any positive integer $n \geq 1$, consider the standard closed embedding $\P^{n-1} \to \P^n$ (setting $\P^{0} := \{*\}$), and the proper modulus pair $(\P^n ,\P^{n-1})$. Then the projection $M(\P^n,\P^{n-1})\to \Z$ is an isomorphism in $\MDM_\gm^\eff$.
\end{prop}

In \cite{KeSa}, Kelly and Saito also provide a very concise proof of a modulus version of \cite[Prop. 3.5.2]{voetri}. Recall from \cite{shuji-purity} the following definition:

\begin{defn}\label{d6.4} A modulus pair $(\ol{\sX},\sX^\infty)$ is said to be \emph{log smooth} (in short: ls) if $\ol{\sX}$ is smooth and $|\sX^\infty|$ is a simple normal crossing divisor. 
\end{defn}

\begin{thm}[Smooth blowup triangle with modulus]\label{thm:smbumod}
Let $X \in \ulMCor$ be an ls  modulus pair. Let $i : Z \to X$ be an ambient minimal morphism with $\ol{i} : \ol{Z} \to \ol{X}$ a closed immersion such that $\ol{Z}$ is smooth.
Assume moreover that $i$ is transversal (see \cite[Def. 7]{KeSa} for the definition of transversality).
Let $\ol{\pi} : \Bl_{\ol{Z}}(\ol{X}) \to \ol{X}$ be the blowup of $\ol{X}$ along $\ol{Z}$, and $\ol{i}':\ol{E} \to \Bl_{\ol{Z}}(\ol{X})$ the exceptional divisor.
Set 
\[
\Bl_{Z}(X) := (\Bl_{\ol{Z}}(\ol{X}),\ol{\pi}^\ast X^\infty), \quad
E:=(\ol{E}, \ol{i}^{\prime \ast} \Bl_{Z}(X)^\infty).
\]
Note that the resulting morphisms $\pi : \Bl_{Z}(X) \to X$, $i':E \to \Bl_{Z}(X)$ and $\pi|_E : E \to Z$ are minimal, where $\pi|_E$ is the restriction of the natural morphism $E \to X$.

Then there exists a distinguished triangle in $\ulMDM^\eff$ (hence in $\ulMDM_\gm^\eff$) of the form
\[
\ul{M} (E) \xrightarrow{i' \oplus -\pi|_E} \ul{M}(\Bl_Z (X)) \oplus \ul{M}(Z) \by{\pi \oplus i} \ul{M}(X) \by{+1}.
\]
\end{thm}

Matsumoto established in \cite{matsumoto} the following interesting distinguished triangle in $\ulMDM^\eff$, which lifts the classical Gysin triangle when the closed subset is of codimension $1$ (see Remark \ref{r7.1}). 

\begin{thm}\label{thm:tame-gysin}
Let $(\ol{X},X^\infty)$ be an ls modulus pair, and let $\ol{Z} \subset \ol{X}$ be an effective Cartier divisor which is integral and smooth. 
Assume that $\ol{Z}$ is not contained in $X^\infty$, and that the support of the divisor $X^\infty + \ol{Z}$ is a strict normal crossing divisor on $\ol{X}$.
Set $Z^\infty := X^\infty \times_{\ol{X}} \ol{Z}$.

Then one has the following distinguished triangle in $\ulMDM^\eff$:
\begin{equation}\label{eq:tame-gysin}
\ul{M}(\ol{X},X^\infty + \ol{Z}) \to \ul{M}(\ol{X},X^\infty) \to \ul{M}(\ol{Z},Z^\infty)(1)[2] \by{+1},
\end{equation}
where $[-](1)$ denotes the Tate twist from Definition \ref{d6.2}.
\end{thm}

\begin{remark}\label{r7.1}
Applying the triangulated functor $\ulomega_\eff$ to the distinguished triangle \eqref{eq:tame-gysin}, we recover the Gysin triangle in $\DM^\eff$:
\begin{equation}\label{eq7.2}
M^V(X^\o - |Z^\o|) \to M^V(X^\o) \to M^V(Z^\o)(1)[2] \by{+1}.
\end{equation}
\end{remark}

\begin{remark}
In \cite{matsumoto}, under the same assumption as in Theorem \ref{thm:tame-gysin}, a second lifting of \eqref{eq7.2} is  constructed  in $\ulMDM^\eff$:
\[
\ul{M}(\ol{X} - |Z^\infty|) \to \ul{M}(X) \to \mathrm{Th}(N_Z X,op) \by{+},
\]
where $\mathrm{Th}(N_Z X,op)$ is a suitable ``Thom space'' in the modulus setting, whose definition we do not recall here.
\end{remark}

\appendix

\section{Categorical toolbox, III}\label{sect:app}

\subsection{Monoidal categories \cite[VII.1]{mcl}} Recall that a monoidal category $(\sC,\otimes)$ is \emph{closed} if $\otimes$ has a right adjoint $\uHom$. 
We shall use the following lemma several times:

\begin{lemma}\label{l1.5} Let $\sC$ and $\sD$ be two closed monoidal categories, and let $u:\sC\to \sD$ be a lax $\otimes$-functor: this means that we have a natural transformation
\begin{equation}\label{eqA.6}
u X\otimes u Y\to u(X\otimes Y).
\end{equation}
Assume that $u$ has a right adjoint $v$. Then, for any $(X,Y)\in \sC\times \sD$, there is a canonical morphism
\[\uHom_\sC(X,v Y)\to v\uHom_\sD(uX,Y)\]
bivariant in $(X,Y)$, which is an isomorphism if \eqref{eqA.6} is a natural isomorphism.
\end{lemma}

\begin{proof} Applying $u$ to the evaluation morphism $\uHom_\sC(X,vY)\otimes X\to vY$ and using the counit of the adjunction, we get a composite morphism $u\uHom_\sC(X,vY)\otimes uX\to uvY\to Y$, hence a morphism
\[u\uHom_\sC(X,vY)\to \uHom_\sD(uX,Y) \]
and finally a morphism
\[\uHom_\sC(X,vY)\to v\uHom_\sD(uX,Y) \]
which is checked by Yoneda's lemma to be an isomorphism when \eqref{eqA.6} is.
\end{proof}

\subsection{Categories with suspension (\protect{\cite[Ex. 11.1]{ks}, \cite[A.2.4]{kahnzeta}})} \label{susp} A category $\sA$ provided with an endofunctor $L$ of $\sA$ is called a \emph{category with suspension}. They form a $2$-category as in \cite[Def. A.25]{kahnzeta}: a $1$-morphism is a functor with a natural isomorphism of commutation with the suspensions, and a $2$-morphism is the ``obvious'' notion (it will not be used in this paper). We say that $L$ is \emph{invertible} if it is a self-equivalence. By \cite[Lemma A.26]{kahnzeta}, the full embedding of the $2$-category of categories with invertible suspension into that of all categories with suspension has a $2$-left adjoint, which sends $(\sA,L)$ to $(\sA[L^{-1}],\tilde L)$ where $\sA[L^{-1}]$ has objects $(A,n)$ for $A\in \sA$, $n\in\Z$, morphisms
\begin{equation}\label{eqA.1}
\sA[L^{-1}]((A,m),(B,n))=\colim_{k+m\ge 0,k+n\ge 0}\sA(L^{k+m}A,L^{k+n}B)
\end{equation}
and $\tilde L(A;n)=(A,n+1)$. This yields:

\begin{lemma}\label{lA.1} Let $(\sA,L)$, $(\sA',L')$ be two categories with suspension.
\begin{enumerate}
\item If $\sA$ is Karoubian, so is $\sA[L^{-1}]$.
\item Let $F:(\sA,L)\to (\sA',L')$ be a $1$-morphism of categories with suspension. If $F$ is full (resp. faithful), so is the induced $1$-morphism $\tilde F:\sA[L^{-1}]\to \sA'[{L'}^{-1}]$.
\end{enumerate}
\end{lemma}

\begin{proof} (2) is obvious in view of Formula \eqref{eqA.1}. For (1),  let $e=e^2$ be an endomorphism of $(A,n)\in \sA[L^{-1}]$. By \eqref{eqA.1} again, there exists $k\gg 0$ such that $n+k\ge 0$ and some $e_k=e_k^2\in \mathrm{End}(L^{n+k}A)$ mapping to $e$ via the canonical functor $\rho:\sA\to \sA[L^{-1}]$. Let $B=\IM e_k$; then $(B, -k)$ is an image of $e$.
\end{proof}

\subsection{Brown representability and compact generation}

Recall the following definitions and results of Neeman:

\begin{defn}\label{dbrown} A triangulated category $\sT$ has the \emph{Brown representability property} if 
\begin{enumerate}
\item It is cocomplete.
\item Any homological functor $H:\sT^\op\to \Ab$ which converts infinite direct sums into products is representable.
\end{enumerate}
\end{defn}

\begin{lemma}[\protect{\cite[Cor. 10.5.3]{ks}}]\label{lA.3} If $\sT$ has the Brown representability property, it is complete; a triangulated functor $F:\sT\to \sT'$ has a right adjoint $G$ if and only if it is strongly additive (Definition \ref{def:stradd}), and $G$ is triangulated.\qed
\end{lemma}

\begin{ex}\label{exA.4} Suppose $\sT$ is cocomplete and let $\sR\subset\sT$ be a localising subcategory: $\sR$ is triangulated and closed under direct sums. Then the inclusion functor $\sR\inj \sT$ and the localisation functor $\sT\to \sT/\sR$ are strongly additive \cite[Lemma 1.5]{bn}.
\end{ex}

\begin{defn}\label{dA.4} Let  $\sT$ be a triangulated category.\\
a) An object $X\in \sT$ is \emph{compact} if the functor $Y\mapsto \sT(X,Y)$ is strongly additive. We write $\sT^c$ for the thick subcategory of $\sT$ consisting of compact objects.\\
b) A subset $\sX$ of $Ob(\sT)$ \emph{generates $\sT$} if its right orthogonal is $0$.\\
c) $\sT$ is \emph{compactly generated} if it is cocomplete and generated by a (small) set of compact objects.\\
d) Given a subset $\sX$ of $Ob(\sT)$, the \emph{thick hull of $\sX$ in $\sT$} is the smallest triangulated subcategory of $\sT$ which contains $\sX$ and is closed under direct summands.
\end{defn}

\begin{rk}\label{rA.2} Suppose that $\sT$ is cocomplete. Then a set $\sX\subset Ob(\sT)$ of compact objects generates $\sT$ in the sense of Definition \ref{dA.4} b) if and only if the smallest localising subcategory of $\sT$ containing $\sX$ is equal to $\sT$ \cite[Lemma 2.2.1]{ss}.
\end{rk}

\begin{ex}\label{exA.5} Let $\sA$ be an essentially small additive category and $\sB=\Mod\sA$. Then $\sT=D(\sB)$ is compactly generated and $K^b(\sA)\iso \sT^c$ \cite[Prop. A.4.1]{birat-tri}.
\end{ex}

We have the following very useful result of Beilinson-Vologodsky \cite[Proposition in \S 1.4.2]{be-vo}  (see also \S 1.2 of op. cit.):

\begin{thm}\label{tA.5} Let $\sT$ be a cocomplete triangulated category and let $\sS\subseteq \sT$ be a localising subcategory which is generated by a set of compact objects of $\sT$. Then the localisation functor $\sT\to \sT/\sS$ has a right adjoint whose essential image is the right orthogonal $\sS^\perp$ of $\sS$. In particular, $\sS=\sT$ $\iff$ $\sS^\perp=0$.
\end{thm}

The two main results on compactly generated triangulated categories are:

\begin{thm}[\protect{\cite[Th. 4.1]{neeman4}}]\label{tA.3} Any compactly generated triangulated category has the Brown representability property. In particular \cite[Th. 14.3.1]{ks}, this is the case for the unbounded derived category of a Grothendieck abelian category. 
\end{thm}

\begin{thm}[\protect{\cite[Th. 2.1]{neeman}}]\label{tA.4} Let $\sT$ be a compactly generated triangulated category. Let $\sS\subset \sT$ be a localising subcategory generated by a set of compact objects of $\sT$. Then $\sT/\sS$ is compactly generated and compact objects of $\sT$ remain compact in $\sT/\sS$;  the induced functor $\sT^c/\sS^c\to (\sT/\sS)^c$ is fully faithful and $(\sT/\sS)^c$ is the thick hull of $\sT^c/\sS^c$ in $\sT/\sS$.
\end{thm}

\begin{cor}[cf. \protect{\cite[Th. A.2.6]{birat-tri}}]\label{cA.1} In the situation of Theorem \ref{tA.4}, the localisation functor $\sT\to \sT/\sS$ has a right adjoint, which also has a right adjoint.
\end{cor}

We shall also use the following lemma of Neeman, a special case of \cite[Lemma 4.4.5]{birat-tri}:

\begin{lemma}\label{lA.4} Let $\sT$ be a cocomplete triangulated category and let $\sX\subset Ob(\sT)$ be a set of compact objects. If $\sX$ generates $\sT$ (see Definition \ref{dA.4} and Remark \ref{rA.2}), then the thick hull of $\sX$ is $\sT^c$.
\end{lemma}

\subsection{Unbounded derived categories: complements}

\begin{thm}\label{t.mon} Let $\sA$ be an additive category.\\
a) $\Mod\sA$ is a Grothendieck category with a set of projective generators.\\ 
b) If $\sA$ is monoidal, its tensor structure canonically extends to $\Mod\sA$ through the additive Yoneda functor, and provides $\Mod\sA$ with the structure of a closed additive monoidal category.\\
c) The $\otimes$-structure of $\sA$ extends uniquely to a $\otimes$-triangulated structure on the homotopy category $K^b(\sA)$.\\
d) The $\otimes$-structure of $\Mod\sA$ has a total left derived functor, which  is strongly additive and provides $D(\Mod\sA)$ with a closed $\otimes$-triangulated structure. \\
e)  If $u:\sA \to \sB$ is a monoidal functor, $u_!:\Mod\sA\to \Mod\sB$ is monoidal, and so are the functors $K^b(u):K^b(\sA)\to K^b(\sB)$ and $Lu_!:D(\Mod\sA)\to D(\Mod\sB)$.
\end{thm}

\begin{proof} a) See e.g. \cite[Prop. 1.3.6]{ak} for the first statement; the projective generators are given by $\sE=\{y(A)\mid A\in \sA\}$.  
b) is  \cite[Def. 8.2]{mvw} or \cite[A.8]{KY}. c) is easy (define $\otimes$ termwise). 
 d) This applies to any right exact covariant bifunctor $T:\Mod\sA\times \Mod\sA\to \sC$, where $\sC$ is abelian and cocomplete: by a) and \cite[Th. 14.4.3]{ks}, $K(\Mod\sA)$ has enough homotopically projective objects (K-projective in the sense of Spaltenstein \cite{spaltenstein}), which means that the localisation functor $\lambda:K(\Mod\sA)\to D(\Mod\sA)$ has a left adjoint $\gamma$. Then the formula
\[LT(C,D):=\lambda T(\gamma C, \gamma D)\]
provides the desired total left derived functor. By Example \ref{exA.4}, $\lambda$ and $\gamma$ are strongly additive; thus if $T$ is strongly additive, so is $LT$. 
Similarly, a left exact contravariant/covariant bifunctor $S$ has a total right derived functor $RS$ given by the formula
\[RS(C,D) =\lambda S(\gamma C,\rho D)\]
where $\rho$ is right adjoint to $\lambda$ (also apply a) and \cite[Th. A.2.1 b)]{kmsy2}). In the case $T=\otimes_{\Mod\sA}$, $S=\uHom_{\Mod\sA}$, these formulas immediately imply that $LT$ is left adjoint to $RS$, which gives a second justification of the strong biadditiveness of $\otimes_{\Mod\sA}$.

e) is \cite[A.12]{KY} for the first statement; the second one is easy and the third follows from the universal property of left derived functors as Kan extensions.
\end{proof}

\begin{prop}\label{pA.3} Let $G:\sA\leftrightarrows \sB:F$ be a pair of adjoint functors between Grothendieck  abelian categories, with $G$ exact. Then $RF$ is right adjoint to $RG=D(G)$.  If further $F$ is fully faithful,  then so is $RF$, and $D(G)$ is a localisation. If, on the other hand, $G$ is fully faithful and $F$ is exact, then $D(G)$ is fully faithful.
\end{prop}

\begin{proof} The first statement is a special case of \cite[Th. 14.4.5]{ks}. By  \cite[Lemma A.3.1]{kmsy1}, the next ones are equivalent to saying that the counit morphism $D(G)RF\Rightarrow Id_\sA$ is an isomorphism, which follows from \cite[Lemma A.2.4]{kmsy2}. In the last case, the full faithfulness of $D(G)$ is proven dually since $D(FG)\Rightarrow D(F)D(G)$ is an isomorphism, again by  \cite[Lemma A.2.4]{kmsy2}.
\end{proof}

We shall also need the following elementary lemma, that we borrow from the Stacks project:
\begin{lemma}[\protect{\cite[Lemma 13.30.2]{stack}}]\label{lA.5}
Let $F:\sA\to \sB$ and $G:\sB\to \sA$ be functors of abelian categories such that $F$ is a right adjoint to $G$. Let $K\in D(\sA)$ and let $M\in D(\sB)$. If $RF$ is defined at $K$ and $LG$ is defined at $M$, then there is a canonical isomorphism
\[D(\sB)(M,RF(K))\simeq D(\sA)(LG(M),K).\]
\end{lemma}

\section{Cubical objects and intervals}\label{section:interval}

\subsection{Cubical objects and associated complexes}
We follow \cite{Levine} but we omit the use of
permutations and involutions.
Let $\Cb$ be the subcategory of $\Sets$
which has as objects 
$\ul{n} = \{ 0, 1\}^n$ for $n \in \Z_{\geq 0}$ 
(with $\ul{0}=*$ the terminal object of $\Sets$)
 and whose morphisms are generated by
\begin{align*}
&p^n_i : \ul{n} \to \ul{n-1} \qquad
(n \in \Z_{>0}, ~ i \in \{1, \dots, n\}),
\\
&\delta^n_{i, \epsilon} : \ul{n} \to \ul{n+1} \qquad
(n \in \Z_{\geq 0}, ~ i \in \{1, \dots, n+1\},~
\epsilon \in \{ 0, 1 \}),
\end{align*}
where
$p_i^n$ omits the $i$-th component
and $\delta^n_{i, \epsilon}$ inserts
$\epsilon$ at the $i$-th component.

\begin{defn}\label{d3.1} Let $\sA$ be a category.
A covariant (resp. contravariant)
functor $A : \Cb \to \sA$ 
is called a 
{\it co-cubical}
(resp. {\it cubical}) {object} in $\sA$;
\end{defn}

\begin{remark}
The definition of $\Cb$ 
in \cite{Levine}
is different from ours.
(It also contains other morphisms
called permutations and involutions.)
However, 
concerning the following lemma,
the same proof as in loc. cit. 
works in our more basic setting.
\end{remark}

\begin{lemma}\label{lem:levine}
Let $A: \Cb^{\op} \to \sA$ be a cubical object in 
a pseudo-abelian category $\sA$.
Put $A_n:=A(\ul{n})$.
\begin{enumerate}
\item
We have well-defined objects 
\begin{align*}
& A_n^{\deg} := 
\Im\Big(\oplus p_i^{n*}:\bigoplus_{i=1}^n A_{n-1} \to A_n\Big)
\iso
\Im\Big(\oplus \delta_{i, 1}^{(n-1)*}
:A_{n} \to \bigoplus_{i=1}^n A_{n-1}\Big),
\\
& A_n^{\nu} := 
\ker\Big(\oplus \delta_{i, 1}^{(n-1)*}
:A_{n} \to \bigoplus_{i=1}^n A_{n-1}\Big)
\iso
\Coker\Big(\oplus p_i^{n*}:\bigoplus_{i=1}^n A_{n-1} \to A_n\Big)
\end{align*}
in $\sA$,
and $A_n^{\nu} \oplus A_n^{\deg} \iso A_n$ holds.
\item
Let 
$d_n:=\sum_{i=1}^{n+1}(-1)^i
(\delta^{n *}_{i, 1}-\delta^{n *}_{i, 0})
: A_{n+1} \to A_n$.
This makes $A_\bullet$ a complex,
of which 
$A_\bullet^{\nu}$ and $A_\bullet^{\deg}$ 
are subcomplexes.
The two complexes $A_\bullet/A_\bullet^{\deg}$ and
$A_\bullet^{\nu}$ are isomorphic.

\end{enumerate}
\end{lemma}
\begin{proof}
See \cite[Lemmas 1.3, 1.6]{Levine}.
\end{proof}

\begin{remark}\label{rem:obvious-dual}
We have obvious dual statements of 
Lemma \ref{lem:levine} for co-cubical objects. 
We state them here for later use.
Let $A: \Cb \to \sA$ be a co-cubical object in 
a pseudo-abelian category $\sA$.
Put $A^n:=A(\ul{n})$.
\begin{enumerate}
\item
We have well-defined objects 
\begin{align*}
& A^n_{\deg} := 
\Im\Big(
\oplus \delta_{i, 1 *}^{(n-1)} :
\bigoplus_{i=1}^n A^{n-1} \to A^{n}\Big)
\iso
\Im\Big(\oplus p_{i*}^{n}:A^n 
\to \bigoplus_{i=1}^n A^{n-1}\Big),
\\
& A^n_{\nu} := 
\Ker\Big(\oplus p_{i*}^{n}:A^n 
\to \bigoplus_{i=1}^n A^{n-1}\Big)
\iso
\Coker\Big(
\oplus \delta_{i, 1 *}^{(n-1)} :
\bigoplus_{i=1}^n A^{n-1} \to A^{n}\Big),
\end{align*}
in $\sA$,
and $A^n \iso A^n_{\nu} \oplus A^n_{\deg}$ holds.
\item
Let 
$d^n:=\sum_{i=1}^{n+1}(-1)^i
(\delta^{n}_{i, 1 *}-\delta^{n}_{i, 0 *})
: A^n \to A^{n+1}$. This makes $A^\bullet$ a complex,
of which
$A^\bullet_{\nu}$ and $A^\bullet_{\deg}$
are subcomplexes. 
The two complexes 
$A^\bullet/A^\bullet_{\deg}$ and
$A^\bullet_{\nu}$ are isomorphic.
\end{enumerate}
\end{remark}

\begin{remark}\label{rem:decom} Let $\sA$ be pseudo-abelian and provided with an additive unital symmetric monoidal structure $\otimes$. 
Let $A : \Cb \to \sA$
be a co-cubical object, 
and suppose that $A$ is strict monoidal
(i.e. $A(\ul{m} \times \ul{n})=A(\ul{m}) \otimes A(\ul{n})$).
Then: \\
1) $A^0=A^0_\nu=\un$ is the unit object of $\sA$,
and $A^0_{\deg}=0$.
For $n>0$, 
combining 
$A^1 = A^1_\nu \oplus A^1_{\deg}$
and
$A^n = A^1 \otimes \dots \otimes A^1$,
we get a decomposition
\begin{align*}
&A_\nu^n = A_\nu^1 \otimes \dots \otimes A_\nu^1,
\\
&A_{\deg}^n = 
\bigoplus_{\sigma \not\equiv \nu}
A_{\sigma(1)}^1 \otimes \dots \otimes A_{\sigma(n)}^1,
\end{align*}
where $\sigma$ ranges over all maps
$\{1, \dots, n \} \to \{ \nu, \deg \}$
except for the constant map with value $\nu$.\\
2) $A^\bullet$ has a canonical comonoid structure 
where the counit
and comultiplication
are respectively given by 
\begin{align}\label{eqproj}
&\pi^\bullet : A^\bullet \to A^0[0]=\un,
\quad
\pi^n = 0 ~(n>0) \text{ and }\pi^0 = \id_{A^0},\\
\label{eqdelta}
&\Delta^\bullet : A^\bullet \to 
\Tot(A^\bullet \otimes A^\bullet)
\end{align}
where $\Delta^n = \sum_{p+q=n} \Delta^{p, q}$ with
$\Delta^{p, q} : A^{p+q} \overset{=}{\longrightarrow} A^p \otimes A^q$.
In view of 1),
we see that $A^\bullet_\nu$ 
inherits the same structure:
\[ 
\pi^\bullet_{\nu} : A^\bullet_{\nu} \to \un,
\quad
\Delta^\bullet_{\nu} : A^\bullet_{\nu} \to 
\Tot(A^\bullet_{\nu} \otimes A^\bullet_{\nu}).
\]
\end{remark}

\subsection{Interval structure} 
Let $\sA$ be a unital symmetric monoidal category.
 Recall from Voevodsky \cite{H1} the notion of interval:

\begin{defn}\label{d4.1} Let $\un$ be the unit object of $\sA$. An \emph{interval} in $\sA$ is a quintuple $(I,p,i_0,i_1,\mu)$, with $I\in \sA$, $p:I\to \un$, $i_0,i_1:\un\to I$, $\mu:I\otimes I\to I$, verifying the identities
\begin{align*}
&pi_0=pi_1=1_\un, 
\\
&\mu\circ (1_I\otimes i_0)=
\mu\circ (i_0\otimes 1_I)=i_0p,
\\
&\mu\circ (1_I\otimes i_1)=
\mu\circ (i_1\otimes 1_I)=1_I.
\end{align*}
\end{defn}

\begin{defn}\label{rem:interval-cube}
Given an  interval 
$(I,p,i_0,i_1,\mu)$ in $\sA$,
we define 
a strict monoidal co-cubical object $A : \Cb \to \mathcal{A}$ by
\[A^n=I^{\otimes n},~
p^n_{i*}
=1_I^{\otimes(i-1)} \otimes p \otimes 1_I^{\otimes (n-i)},~
\delta^n_{i \epsilon*}
=1_I^{\otimes(i-1)} \otimes i_\epsilon
\otimes 1_I^{\otimes (n-i)}
\]
(this does not use the morphism $\mu$). When $\sA$ is pseudo-abelian,
we write $I^\bullet, ~I^\bullet_\nu, ~I^\bullet_{\deg}$
for the associated complexes introduced in 
Remark \ref{rem:obvious-dual}. 
\end{defn}

By definition and Remark \ref{rem:decom}, we have
\[ I_\nu^n=I_\nu\otimes \cdots\otimes I_\nu\;\;\text{ with } I_\nu=\Ker(I\rmapo{p} \un).\]

\begin{remark}
Conversely, Levine introduced in \cite{Levine}  a notion of \emph{extended co-cubical object} $A : \ECb \to \mathcal{A}$, where $\ECb$ is the smallest symmetric monoidal subcategory of $\Sets$ that contains $\Cb$ and the morphism
\[ \tilde{\mu} : \ul{2} \to \ul{1}; \quad
  (a, b) \mapsto ab.
\]
Given such a  (strict monoidal)  extended co-cubical object $A$, 
we may define an interval 
$(I,p,i_0,i_1,\mu)$ in $\sA$ by
\[ I=A(\ul{1}),~ p=p^1_{1*},~ 
i_0=\delta^0_{1, 0*},~ i_1=\delta^0_{1, 1*},~
\mu=\tilde{\mu}_*.
\]

Such intervals are not arbitrary, as $\mu$ makes $I$ a commutative monoid (because so does $\tilde \mu$ with $\ul{1}$). However, all intervals encountered in practice are commutative monoids, including in \cite{H1,voetri} and here (Lemma \ref{l7.1}). 
\end{remark}

\begin{defn}\label{def:i-local}
a) An object $X\in \sA$ is 
\emph{$I$-local at $Y\in \sA$}\footnote{This notion is useful in \cite{shuji-purity}.} if $p$ induces an isomorphism $\sA(Y,X)\iso \sA(Y\otimes I,X)$; $X$ is \emph{$I$-local} if it is $I$-local at $Y$ for any $Y\in \sA$. If $\sA$ is closed, it is equivalent to ask for the morphism
\[X\by{p^*} \uHom(I, X)\]
to be an isomorphism.
\\
b) A morphism $f:Y\to Z$ in $\sA$  is called an \emph{$I$-equivalence} if $\sA(Z,X)\by{f^*} \sA(Y,X)$ is an isomorphism for any $I$-local $X$.
\end{defn}

\begin{lemma}\label{l4.1} Let $X,Y\in \sA$. Then 
\begin{enumerate}
\item If $X$ is $I$-local at $Y$, the maps $1_Y\otimes i_0^*,1_Y\otimes i_1^*:\sA(Y\otimes I,X)\to \sA(Y,X)$  are equal.
\item If the maps $1_{Y\otimes I}\otimes i_0^*,1_{Y\otimes I}\otimes i_1^*:\sA(Y\otimes I\otimes I,X)\to \sA(Y\otimes I,X)$  are equal, then $X$ is $I$-local at $Y$.
\item $X$ is $I$-local if and only if the maps $i_0^*,i_1^*:\sA(Y\otimes I,X)\to \sA(Y,X)$  are equal for all $Y\in \sA$ (equivalently when $\sA$ is closed: if and only if the maps $i_0^*,i_1^*:\uHom(I,X)\to X$ are equal).
\end{enumerate}
\end{lemma}

\begin{proof} 
For (2), the last two identities of Definition \ref{d4.1} imply that $p^*i_0^*:\sA(Y\otimes I,X)\to \sA(Y\otimes I,X)$ is the identity, hence the claim since $i_0^*p^*$ is also the identity. (3) now follows from (1) and (2). 
\end{proof}

\begin{rk}\label{r4.1} Actually, Definition \ref{d4.1} is more general than Voevodsky's definition in \cite[2.2]{H1} or (with Morel) \cite[2.2.3]{mv}. There, the $\otimes$-category $\sA$ is a site with products (in \cite{H1}) or the category of sheaves on a site (in \cite{mv}), and the tensor structure is the one given by products of objects or of sheaves. Voevodsky constructs in \cite[loc. cit.]{H1} a universal cosimplicial object, whose general term is $I^n$.  Unfortunately, the formulas of loc. cit. implicitly use diagonal morphisms which are not available in general $\otimes$-categories,  in particular in the ones we use here (see Remark \ref{w2.1}). So, while one can develop a cubical theory out of Definition \ref{d4.1}, we do not know if this definition is sufficient to develop a simplicial theory. 
\end{rk}

\subsection{Homotopy equivalences}

\begin{prop}\label{phot2} Let $\sA$ be a pseudo-abelian $\otimes$-category, provided with an interval $I$. Let $I^\bullet$ be as in Definition \ref{rem:interval-cube}. 
Then the morphisms
\begin{align}\label{eq:1st-layer0}
&1 \otimes p^1_{1*} : 
I^\bullet \otimes I^1[0]\to I^\bullet,
\\\label{eq:1st-layer}
&1 \otimes p^1_{1*} : 
I^\bullet_\nu \otimes I^1[0]\to I^\bullet_\nu,
\\
\label{eq:delta}
&\Delta^\bullet_{\nu} : I^\bullet_{\nu} \to 
\Tot(I^\bullet_{\nu} \otimes I^\bullet_{\nu})
\end{align}
are homotopy equivalences.
\end{prop}

\begin{proof}
For \eqref{eq:1st-layer0}, since $p^1_{1}\delta^0_{1, 0} = 1_{\ul{0}}$, the composition
$(1 \otimes p^1_{1*}) (1 \otimes \delta^0_{1, 0*}):I^\bullet\to I^\bullet$ is the identity. Let $s^{n}: I^{n+1}\iso I^n\otimes I^1$ be the tautological isomorphism. The identities
\[s^n \delta^{n}_{j,\epsilon*} =
\begin{cases}
((\delta^{n-1}_{j,\epsilon*}\otimes 1) s^{n-1} &\text{if $j<n+1$}\\
1_{I^n}\otimes i_\epsilon &\text{if $j=n+1$}
\end{cases}
 \]
yield
\[s^n d^n - (d^{n-1}\otimes 1) s^{n-1}=1\otimes i_1-1\otimes i_0.\]

Then the composition
\[\sigma^{n+1}: I^{n+1}\otimes I^1\by{s^n \otimes 1} I^{n}\otimes  I^1\otimes  I^1\by{1\otimes \mu} I^{n}\otimes  I^1\]
yields a chain homotopy from $1\otimes (\delta^0_{1, 0*}p^1_{1*})$ to $1\otimes 1$, which concludes the proof.
Now \eqref{eq:1st-layer} is also homotopy equivalence
as a direct summand of \eqref{eq:1st-layer0}.

Consider \eqref{eq:delta}.
By induction and the homotopy equivalence \eqref{eq:1st-layer0},
we find that for any $q > 0$
\begin{equation}\label{eq:qth-layer}
\id \otimes (p^1_1  p^2_1 \dots p^q_1)_{*} : 
I^\bullet \otimes I^q[0]\to I^\bullet
\end{equation}
is a homotopy equivalence.
Since $I^q_{\nu}$ is a direct summand of $I^q$
contained in $\Ker((p^1_1  p^2_1 \dots p^q_1)_{*})$
by Remark \ref{rem:obvious-dual},
we find that $I^\bullet \otimes I^q_{\nu}[0]$
is contractible for $q > 0$.
The same is true of
$I^\bullet_{\nu} \otimes I^q_{\nu}[0]$
because it is a direct summand of 
$I^\bullet \otimes I^q_{\nu}[0]$.
Lemma \ref{lem:oesterle} (2) below then shows
that
$\Tot(1 \otimes \pi^\bullet) 
: \Tot(I^\bullet_{\nu} \otimes I^\bullet_{\nu}) \to 
I^\bullet_{\nu}$
is a homotopy equivalence, where $\pi^\bullet$ is as in \eqref{eqproj}.
Since 
$\Tot(1 \otimes \pi^\bullet)$
is left inverse to $\Delta_\nu^\bullet$,
this shows that $\Delta_\nu^\bullet$ 
is a homotopy equivalence.
\end{proof}

\begin{lemma}\label{lem:oesterle}
Let $\sA$ be an additive category.
Let us call a double complex
$S^{\bullet, \bullet}$ in $\sA$
{\it locally finite} if
$\{ p \in \Z ~|~ S^{p, n-p} \not= 0 \}$
is a finite set for each $n \in \Z$.
\begin{enumerate}
\item
Let $S^{\bullet, \bullet}$ 
be a locally finite double complex in $\sA$.
Suppose that the single complex
$S^{\bullet, q}$ is contractible
for each $q \in \Z$.
Then $\Tot(S^{\bullet, \bullet})$ is contractible.

\item
Let $f^{\bullet, \bullet} : 
S^{\bullet, \bullet} \to T^{\bullet, \bullet}$
be a morphism of 
locally finite double complexes in $\sA$.
If $f^{\bullet, q}$ is a homotopy equivalence
for each $q \in \Z$,
then so is
$\Tot(f^{\bullet, \bullet})
: S^{\bullet, \bullet} \to T^{\bullet, \bullet}$.
\end{enumerate}
\end{lemma}

\begin{proof}
(1)\footnote{
We learned this proof from
J. Oesterl\'e. We thank him.
} Let us write 
$d_1^S : S^{\bullet, \bullet} \to S^{\bullet+1, \bullet},~
d_2^S : S^{\bullet, \bullet} \to S^{\bullet, \bullet+1}$
for the differentials of $S^{\bullet, \bullet}$,
and set
$d^S = d_1^S + d_2^S$.
By assumption we have a map $s : S^{\bullet, \bullet} \to S^{\bullet, \bullet}$ 
of bidegree $(-1, 0)$ such that $d_1^S s + sd_1^S = \id_{S^{\bullet, \bullet}}$.
Thus $d^Ss+sd^S-\id_{S^{\bullet, \bullet}}$ is an endomorphism of $S^{\bullet, \bullet}$
of bidegree $(-1, 1)$,
which defines an endomorphism $u$ of $\Tot(S^{\bullet, \bullet})$
of degree $0$.
By assumption, $u$ restricted to each degree
is nilpotent.
Hence $\id+u$ is an isomorphism,
which implies that $\Tot(S)$ is contractible.

(2) We shall use the following fact:
\begin{quote}
(*) A morphism $g$ of (simple) complexes
is a homotopy equivalence 
if and only if $\Cone(g)$ is contractible.
\end{quote}
Let $U^{\bullet, \bullet}$ be a cone of $f$,
that is,
$U^{p, q}=T^{p, q} \oplus S^{p+1, q}$
equipped with 
$d_1^U = 
\begin{pmatrix} d_1^T & f \\ 0 & d_1^S \end{pmatrix}
: U^{p, q} \to U^{p+1, q}$
and
$d_2^U = 
\begin{pmatrix} d_2^T & 0 \\ 0 & d_2^S \end{pmatrix}
: U^{p, q} \to U^{p, q+1}$.
For each $q \in \Z$,
we have 
$U^{\bullet, q} = \Cone(f^{\bullet, q})$,
as (single) complexes.
By assumption and (*), they are contractible.
Then (1) shows that $\Tot(U)$ is contractible.
Since we have $\Cone(\Tot(f)) = \Tot(U)$ by definition,
this implies that
$\Tot(f)$ is contractible by (*).
\end{proof}

\subsection{An adjunction}\label{s4.1} 

Let $\sT$ be a tensor triangulated category,  compactly generated (Definition \ref{dA.4}) and equipped with an interval \allowbreak 
$(I,p,i_0,i_1,\mu)$.  We assume: 

\begin{hyp}\label{hB.1}
The tensor structure of $\sT$ is strongly biadditive (i.e., $- \otimes -$ is strongly additive in each entry), and $-\otimes I$ preserves the full subcategory $\sT^c$ of compact objects.
\end{hyp}

By Theorem \ref{tA.3}, $\sT$ enjoys the Brown representability property  of Definition \ref{dbrown}. By Lemma \ref{lA.3}, $\otimes$ therefore has a right adjoint $\uHom$.

\begin{defn}\label{def:r_i}
Let $\sR_I \subset \sT$ be the localising subcategory generated by objects of the form $\Cone(X\otimes I \by{1\otimes p} X)$ for $X\in \sT$. We write $\sT_I$ for the Verdier quotient $\sT/\sR_I$.
\end{defn}

\begin{prop}\label{prop:right-adjoint-j-interval}\
\begin{enumerate}
\item The functor $\uHom_\sT(I,-)$ is strongly additive.
\item The category $\sT_I$ is compactly generated, hence has the Brown representability property.
\item The localisation functor $L^I :\sT \to \sT_I$ has a (fully faithful) right adjoint $j^I$, which also has a right adjoint $R^I$.
\item The essential image of $j^I$ consists of the $I$-local objects (Definition \ref{def:i-local} a)).
\item The tensor structure on $\sT$ induces a tensor structure on $\sT_I$. 
\end{enumerate}
\end{prop}

\begin{proof} For $(X_j)_{j\in J}$ a family of objects of $\sT$, the invertibility of the map
\[\bigoplus  \uHom_\sT(I,X_j)\to \uHom_\sT(I,\bigoplus  X_j)\]
can be tested on a set of compact generators; it then follows from Hypothesis \ref{hB.1}. This also implies that $\sR_I$ is generated by a set of compact objects of $\sT$, hence (2) follows from Theorem \ref{tA.4}. Then (3) follows from Corollary \ref{cA.1}. (4) is obvious by adjunction, and (5) follows from the fact that if $A\in \sR_I$ and $B\in \sT$, then $A\otimes B\in \sR_I$. 
\end{proof}

\begin{rk} The functor $j^IR^I$ can be described by a double adjunction  as follows: for $X,Y\in \sT$, we have
\[\sT(X,j^IR^I Y)=\sT(L^I  X,R^I Y) =\sT(j^IL^I  X,Y). \]
\end{rk}

Our main result in this appendix, Theorem \ref{thm:10.34}, is a computation of the localisation functor $j^IL^I $ in terms of $I^\bullet_{\nu}$ (see Definition \ref{rem:interval-cube}). Ideally it should be expressed in the above framework. Unfortunately, we do not know how to totalise $I^\bullet_{\nu}$ into an object of $\sT$ in general (compare \cite[\S 3]{bn}). A nice setup would be to assume that $\sT$ is provided with a $t$-structure with heart $\sA$ for which $\otimes$ is $t$-exact and such that $I\in \sA$; unfortunately, the inclusion $\sA\inj \sT$ does not extend to $D(\sA)$ (or even $K(\sA)$) in this generality. So, for simplicity, we take refuge in the situation where $\sT$ is of the form $D(\sA)$ (and where $I\in \sA$).

The  proof of the following theorem will occupy the next two subsections
(see Theorem \ref{t22}).

\begin{thm}\label{thm:10.34} Under suitable additional hypotheses (\ref{h4.1} below), there is a canonical isomorphism
\[j^IL^I (K)\cong \uHom_{D(\sA)}(I^\bullet_{\nu},K)\]
for any $K\in D(\sA)$.
\end{thm}

\subsection{Monadic intermezzo} Let $\sC$ be a category and $(C,\eta,\mu)$ be a monad in $\sC$ in the sense of \cite[Ch. VI]{mcl}. Recall what this means:
\begin{itemize}
\item $C$ is an endofunctor of $\sC$.
\item $\eta:\id\to C$ is a natural transformation (unit).
\item $\mu:C^2\to C$ is a natural transformation (multiplication).
\item  For any $X\in \sC$, we have the identities 
\begin{align}
\mu_X \circ C(\mu_X) &= \mu_X \circ \mu_{C(X)}\label{mon1-2}\\
\mu_X\circ C(\eta_X)&=\mu_X \circ \eta_{C(X)}=1_{C(X)}.\label{mon2-2}
\end{align}
\end{itemize}

We shall not use \eqref{mon1-2} in the sequel.

 Let $C(\sC)$ be the strictly full subcategory of $\sC$ generated by the image of $C$: an object of $\sC$ is in $C(\sC)$ if and only if it is isomorphic to $C(X)$ for some $X\in \sC$; the morphisms of $C(\sC)$ are the morphisms of $\sC$.

\begin{prop}\label{pmon2}  a) If $\mu$ is a natural isomorphism, the full embedding $j:C(\sC)\inj \sC$ has the left adjoint $C$.\\ 
b) Let $C_*$ be a second monad in $\sC$. Assume that the condition of a) holds for $C$ and $C_*$, and that
\begin{thlist}
\item $C_*(\sC)\subseteq C(\sC)$.
\item For any $X\in C(\sC)$, the unit map $X\to C_*(X)$ is an isomorphism.
\end{thlist}
Then there is a natural isomorphism $C\cong C_*$.
\end{prop}

\begin{proof} a) Let $Y\in C(\sC)$ and choose an isomorphism $u:Y\iso C(X)$ with $X\in \sC$. By assumption,  $\eta_Y: Y \to C(Y)$ is an isomorphism, thus the second equality of \eqref{mon2-2} and the naturality of $\eta$ imply that the composite
\[ \epsilon_Y : C(Y) \rmapo {C(u)} C^2(X) \rmapo {\mu_X} C(X) \rmapo {u^{-1}} Y \]
is the inverse of $\eta_Y$, hence does not depend on the choice of $u,X$. One then easily checks that $\epsilon_Y$ for $Y\in C(\sC)$ defines a natural transformation $\epsilon: Cj\to \id$ and that $(\eta,\epsilon)$ 
provides the unit and counit of the desired adjunction.

In b), (i) implies that for any $X\in \sC$, the unit $X\to C_*(X)$ factors through the unit $X\to C(X)$ (use a)). On the other hand, (ii) implies that  $C(\sC)\subseteq C_*(\sC)$, so the same reasoning shows that, conversely, the unit $X\to C(X)$ factors through the unit $X\to C_*(X)$.
\end{proof}

\begin{rk}\label{rff2} The converse of a) is certainly false in general. The point is that a given endofunctor $C$ on $\sC$ might have two completely different monad structures. However, if $(\eta,\mu)$ yields an adjunction between $j$ and $C$, then $\mu$ must be a natural isomorphism because $j$ is fully faithful. In particular, if we start from an adjunction $(j,C)$ with $j$ fully faithful, then the multiplication of the monad $jC$ is a natural isomorphism.
\end{rk}

\subsection{A formula for $j^IL^I $}\label{formulaforjLC} Let $\sT$ be as in Subsection \ref{s4.1}.
We use the notation introduced in Definition \ref{def:r_i}. We assume here that $\sT$ is of the form $D(\sA)$ for some Grothendieck abelian category $\sA$, whence a canonical $t$-structure. To Hypothesis \ref{hB.1}, we add:

\begin{hyp}\label{h4.1}\
\begin{thlist}
\item $\otimes_{D(\sA)}$ is right $t$-exact, hence induces a right exact tensor structure on $\sA$ denoted by $\otimes_\sA$ \cite[Prop. 1.3.17 (i)]{bbd}. ($A\otimes_\sA B:= H_0(A[0]\otimes_{D(\sA)} B[0])$.)
\item Let $\otimes_{K(\sA)}$ be the canonical extension of $\otimes_\sA$ to $K(\sA)$. Then the localisation functor $\lambda:K(\sA)\to D(\sA)$ is lax monoidal, i.e., there is a collection of morphisms
\[\lambda C\otimes_{D(\sA)} \lambda D\to \lambda(C\otimes_{K(\sA)} D)\]
binatural in $(C,D)\in K(\sA)\times K(\sA)$ and commuting with the associativity and commutativity constraints.
\item $\un_{D(\sA)},I\in \sA$ (hence $I=\lambda I[0]$).
\item The map $(\lambda I[0])^{\otimes_{D(\sA)} n}\to \lambda(I^{\otimes_{\sA} n}[0])$ induced by (ii) is an isomorphism for all $n\ge 0$.
 \end{thlist}
 \end{hyp}
 
By adjunction,
the composed functor $j^IL^I $ 
has a canonical monad structure.
Note that its multiplication is an isomorphism 
because $j^I$ 
is fully faithful (compare Remark \ref{rff2}).

\begin{defn}\label{def:i-equiv}
For $K\in D(\sA)$, we let
\[RC_*^I (K) = \uHom_{D(\sA)}(I^\bullet_{\nu},K)\in D(\sA).\]
Here we view the complex $I_\nu^\bullet$ as an object of $D(\sA)$. We call $RC_*^I (K)$ the \emph{derived cubical Suslin complex of $K$} (relative to $I$).
\end{defn}

The comonoidal structure on $I^\bullet_{\nu}$: 
\[\pi^\bullet : I^\bullet_{\nu} \to \un,\quad
\Delta^\bullet : I^\bullet_{\nu} \to 
\Tot(I^\bullet_{\nu} \otimes I^\bullet_{\nu})\]
given by \eqref{eqproj}, \eqref{eqdelta}
induces a monad structure on $RC_*^I $. For example the multiplication is given by
\[\begin{aligned}
RC_*^I (RC_*^I (K)) &=\uHom_{D(\sA)}(I^\bullet_{\nu},\uHom_{D(\sA)}(I^\bullet_{\nu},K))\\
& \cong \uHom_{D(\sA)}(I^\bullet_{\nu}\otimes I^\bullet_{\nu},K) \\
& \rmapo{(\Delta^\bullet)^* } \uHom_{D(\sA)}(I^\bullet_{\nu},K)= RC_*^I (K)\\
\end{aligned}\]
Note that the last map is an isomorphism by Proposition \ref{phot2}. 
The following theorem completes
the proof of Theorem \ref{thm:10.34}. 

\begin{thm}\label{t22}
The two monads $j^IL^I $ and $RC_*^I $ are naturally isomorphic.
\end{thm}

For any $K \in D(\sA)$,
the monad structure on $RC_*^I $
provides us with a natural morphism in $D(\sA)$:
\begin{equation}\label{eq22}
\eta_K:K \to RC_*^I (K).
\end{equation}
We prove the following result together
with Theorem \ref{t22}.

\begin{thm}\label{t2-2} Let $K\in D(\sA$).\\
a) The complex $RC_*^I (K)$ is $I$-local
(Definition \ref{def:i-local} a)).\\
b) The morphism
\eqref{eq22} is an isomorphism if and only if $K$ is $I$-local.\\
c) The morphism \eqref{eq22} is an $I$-equivalence (Definition \ref{def:i-local} b)).
\end{thm}

\begin{proof}[Proof of Theorems \ref{t22} and \ref{t2-2}]
(Compare \cite[proof of Lemma 3.2.2]{voetri} or \cite[proof of Lemma 9.14]{mvw}.) 
We first prove Theorem \ref{t2-2} a) and b).
In view of Definition \ref{def:i-local} 
and Hypothesis \ref{h4.1} (iv), 
a) follows from Proposition \ref{phot2} by adjunction. 
In b),  if $K$ is $I$-local, we have $\uHom(I_\nu,K)=0$ and hence
\begin{multline*} \uHom_{D(\sA)}(I^n_{\nu},K)\cong 
\uHom_{D(\sA)}(I^{n-1}_\nu\otimes I_\nu,K)\\
\cong \uHom_{D(\sA)}(I^{n-1}_\nu,\uHom_{D(\sA)}(I_{\nu},K))=0 \;\;\text{for } n>0,
\end{multline*}
which implies that \eqref{eq22} is an isomorphism.
Conversely, if \eqref{eq22} is an isomorphism, then $K$ is $I$-local by a).

Next we prove Theorem \ref{t22}.
As mentioned before Definition \ref{def:i-equiv}, 
the multiplication 
of the monad $j^IL^I $ is an isomorphism,
and
the same is true for $RC_*^I $ as proven above.
Theorem \ref{t22} now follows from 
Theorem \ref{t2-2} a), b) and Proposition \ref{pmon2} b). 

Finally, Theorem \ref{t2-2} c) follows from 
Theorem \ref{t22}.
\end{proof}

\begin{cor}\label{c22} 
a) For any $K\in \sR_I$, $RC_*^I (K)=0$ in $D(\sA)$.\\
b) The functor $RC_*^I $ is strongly additive.\\
c) The localising subcategory $\sR_I\subset D(\sA)$ is generated by the cones of the $X\to RC_*^I (X)$ for $X\in D(\sA)$. 
In particular, $K\in D(\sA)$ is $I$-local if and only if the natural map
\[\Hom_{D(\sA)}(RC_*^I (X),K[i]) \to 
\Hom_{D(\sA)}(X, K[i])
\]
is an isomorphism for any $X\in D(\sA)$ and any $i\in \Z$.
\end{cor}

\begin{proof} a) This is obvious from Theorem \ref{t22} since $\sR_I \cap j^ID(\sA)_I=0$, the two categories being mutually orthogonal.
 
b) This follows from Theorem \ref{t22} and the strong additivity of $j^I$ and $L^I $ (Example \ref{exA.4}).

c) By Theorem \ref{t2-2} c), for any $X\in D(\sA)$
the cone of $X\to RC_*^I (X)$ vanishes in $D(\sA)_I$,
hence it is in $\sR_I$. 
Conversely, let $\sR_I'\subset D(\sA)$ be the localising subcategory generated by these cones. In the commutative diagram
\[\begin{CD}
I[0] \otimes_{D(\sA)} X@>p>> X\\
@VVV @VVV\\
RC_*^I (I[0] \otimes_{D(\sA)} X) @>p'>> RC_*^I (X)
\end{CD}\]
$p'$ is an isomorphism by  a), 
hence the cone of $p$ belongs to $\sR_I'$. The last statement follows.
\end{proof}

\subsection{Comparison of intervals}\label{s4.7} Let $(\sA,I),(\sA',I')$ be as in \S \ref{formulaforjLC}. We give ourselves a right exact cocontinuous monoidal functor $T:\sA\to \sA'$ sending $I$ to $I'$ and respecting the constants of structure of $I$ and $I'$. By \cite[Th. A.10.1 b)]{kmsy1}, $T$ has a right adjoint $S$. We assume that $T$ has a total left derived functor $LT:D(\sA)\to D(\sA')$, which is strongly additive, a monoidal functor and sends $I[0]$ to $I'[0]$ (this is automatic if $T$ is exact). By Brown representability (Lemma \ref{lA.3} and \cite[Th. A.2.1 a)]{kmsy2}), $LT$ has a right adjoint $RS$, which is the total right derived functor of $S$. Then $LT$ induces a triangulated monoidal functor $\ol{LT}:D(\sA)_I\to D(\sA')_{I'}$ via $L^I$ and $L^{I'}$.

The following lemma is obvious:

\begin{lemma} Let $j^I$ and $j^{I'} $ be the right adjoints of the localisation functors $L^I :D(\sA)\to D(\sA)_I$ and $L^{I'}:D(\sA')\to D(\sA')_{I'}$. Then $RS$ sends $j^{I'} D(\sA')_{I'}$ into $j^I D(\sA)_I$, and the induced functor $\ol{RS}:D(\sA')_{I'}\to D(\sA)_I$ is right adjoint to $\ol{LT}$.
\end{lemma}

By construction, we have a natural isomorphism
\begin{equation}\label{eq4.0}
RS j^{I'} \simeq j^I \ol{RS}
\end{equation}
from which we deduce two ``base change morphisms''
\begin{align}
L^I \circ RS &\Rightarrow \ol{RS} \circ L^{I'}\label{eq4.1}\\
LT \circ j^I  &\Rightarrow j^{I'} \circ \ol{LT}.\label{eq4.1a}
\end{align}

\begin{thm}\label{t7.1} \eqref{eq4.1} is an isomorphism.
\end{thm}

\begin{proof} The monoidality of $LT$ yields the following identity, for $(X,K)\in D(\sA)\times D(\sA')$ (Lemma \ref{l1.5}):
\begin{equation}\label{eq4.2}
\uHom_{D(\sA)}(X,RS K)\cong RS\uHom_{D(\sA')}(LTX,K).
\end{equation}

Apply \eqref{eq4.2} to $X=I_\nu^\bullet$: we get an isomorphism
\[RC_*^I (RSK)\cong RS RC'_*(K).\]

In view of Theorem \ref{t22}, this converts to an isomorphism
\[j^I L^I RS(K)\cong RSj^{I'} L^{I'} (K)\]
hence to an isomorphism $L^I RS(K)\cong \ol{RS}L^{I'} (K)$ in view of \eqref{eq4.0} and the full faithfulness of $j^I$. One checks that this isomorphism coincides with \eqref{eq4.1}.
\end{proof}

\begin{defn} We say that $T$ \emph{verifies Condition (V)} if \eqref{eq4.1a} is an isomorphism.
\end{defn}

\begin{lemma}\label{l4.4} $T$ verifies Condition (V) if and only if $LT(j^ID(\sA)_I) \subseteq j^{I'} D(\sA')_{I'}$.
\end{lemma}

\begin{proof} ``Only if'' is obvious. Conversely, let $X\in D(\sA)_I$ be such that $LT j^I(X)\cong j^{I'}  Y$ for some $Y\in D(\sA')_{I'}$. Applying $L^{I'} $, we get
\[Y\cong L^{I'} j^{I'}  Y\cong L^{I'} LT j^I(X)\cong \ol{LT} L^I j^I(X)\cong \ol{LT}(X).\]

Applying $j^{I'} $, it gives an isomorphism
\[LTj^I(X)\cong j^{I'} Y \cong j^{I'} \ol{LT}(X)\]
and one checks that this is induced by \eqref{eq4.1a}.
\end{proof}

\begin{ex}[see also \protect{\cite[Remark (c) in 4.4]{be-vo}}]\label{exB.1}  Applying $L^I$ to the right of \eqref{eq4.1a} and using Theorem \ref{t22}, one gets a natural transformation
\begin{equation}\label{eqB.1}
LT\circ RC_*^I\Rightarrow RC_*^{I'}\circ  LT.
\end{equation}

Take $\sA=\PST$, $\sA'=\NST$, $T=a_\Nis^V$, $I=I'=\Z_\tr^V(\A^1)$. Then the condition of Lemma \ref{lA.4} translates as: the sheafification of an $\A^1$-invariant complex of presheaves with transfers is $\A^1$-invariant. When $k$ is perfect, this is a theorem of Voevodsky \cite[Th. 5.6]{voepre}, which can then be used to prove the equivalence of categories mentioned in \S \ref{s6.1}.  Moreover, $RC_*^I$ yields the \emph{na\"\i ve} Suslin complex, because $\Z_\tr^V((\A^1)^n)$ is projective in $\PST$ for any $n\ge 0$. Thus, the invertibility of \eqref{eqB.1} means in this case that the derived Suslin complex is quasi-isomorphic to the sheafification of the na\"\i ve Suslin complex.

So, while the invertibility of \eqref{eq4.1}  is a formal and general fact, this is far from being the case for  \eqref{eq4.1a}.
If we take $\sA=\MPST$, $\sA'=\MNST$, $T=a_\Nis$ and $I=I'=\Z_\tr(\bcube)$, results in this direction have been obtained by the third author in \cite[Th. 0.4 and 0.6]{shuji-purity}.
\end{ex}

\enlargethispage*{20pt}

\end{document}